%% file: main.tex
\documentclass[11pt,reqno,a4paper]{amsart}

\usepackage[utf8]{inputenc}
\usepackage[english]{babel}
\usepackage{bookmark}
\usepackage{amsmath}
\usepackage{amsthm}
\usepackage{mathtools}

\usepackage{bm}
\usepackage{bbm}
\usepackage{amssymb}
\usepackage{esint}
\usepackage[dvipsnames]{xcolor}
\usepackage{graphicx}
\usepackage{tcolorbox}
\usepackage{float}

\usepackage{caption}
\usepackage{subcaption}
\usepackage{listings}
\usepackage{enumitem}
\usepackage[
  hmarginratio={1:1},               
  vmarginratio={1:1},               
  left=80pt,
  top=80pt,
  heightrounded                     
]{geometry}
\usepackage{csquotes}
\MakeOuterQuote{"}

\usepackage{tikz}
\usetikzlibrary{positioning,
decorations.pathreplacing,
patterns,arrows,math}
\usepackage{pgfplots}
\pgfplotsset{compat=1.15}
\usepackage{mathrsfs}

\usepackage{anyfontsize}            

\DeclareMathOperator*{\argmin}{argmin}

\begin{document}

\theoremstyle{plain}
\newtheorem{theorem}{Theorem}
\newtheorem{corollary}[theorem]{Corollary}
\newtheorem{proposition}[theorem]{Proposition}
\newtheorem{lemma}[theorem]{Lemma}
\newtheorem{definition}[theorem]{Definition}
\newtheorem{remark}[theorem]{Remark}
\newtheorem{example}[theorem]{Example}
\numberwithin{equation}{section}
\newtheorem{assumption}{Assumption}

\newcommand{\N}{\mathbb{N}}
\newcommand{\Z}{\mathbb{Z}}
\newcommand{\Q}{\mathbb{Q}}
\newcommand{\R}{\mathbb{R}}
\newcommand{\C}{\mathbb{C}}
\newcommand{\D}{\mathcal{D}}
\newcommand{\E}{\mathrm{E}}
\newcommand{\e}{\mathrm{e}}
\newcommand{\F}{\mathcal{F}}
\newcommand{\U}{\mathcal{U}}
\newcommand{\X}{\mathcal{X}}
\newcommand{\Y}{\mathcal{Y}}
\renewcommand{\S}{\mathbb{S}}
\renewcommand{\L}{\mathcal{L}}
\renewcommand{\H}{\mathcal{H}}
\renewcommand{\O}{\mathcal{O}}
\renewcommand{\P}{\mathcal{P}}
\newcommand{\eps}{\varepsilon}
\newcommand{\Cc}{\mathcal{C}}
\newcommand{\Rr}{\mathcal{R}}
\newcommand{\Id}{\mathrm{Id}}
\newcommand{\spt}{\mathrm{spt}}
\newcommand{\diam}{\mathrm{diam}}
\newcommand{\Lip}{\mathrm{Lip}}
\newcommand{\inc}{\subseteq}
\newcommand{\mto}{\mapsto}
\newcommand{\isEquivTo}[1]{\underset{#1}{\sim}}
\newcommand{\mres}{\mathbin{
\vrule height 1.6ex depth 0pt width
0.13ex\vrule height 0.13ex depth 0pt width 1.3ex}}
\newcommand{\Lag}{\mathrm{Lag}}
\newcommand{\RLag}{\mathrm{RLag}}
\DeclarePairedDelimiter{\Iint}{\llbracket}{\rrbracket}

\date{\today}
\title[Particle method for a nonlinear
multimarginal OT problem]{Particle method for a nonlinear
multimarginal optimal transport problem}

\author{Adrien Cances}
\address{Adrien Cances.
Université Paris-Saclay, CNRS, Laboratoire de mathématiques d’Orsay, ParMA, Inria Saclay, 91405, Orsay, France}
\email{adrien.cances@universite-paris-saclay.fr}

\author{Quentin Mérigot}
\address{Quentin Mérigot.
Université Paris-Saclay, CNRS, Inria, Laboratoire de mathématiques d’Orsay, 91405, Orsay, France
/
DMA, École normale supérieure, Université PSL, CNRS, 75005 Paris, France 
/
Institut universitaire de France (IUF)}
\email{quentin.merigot@universite-paris-saclay.fr}

\author{Luca Nenna}
\address{Luca Nenna.
Université Paris-Saclay, CNRS, Laboratoire de mathématiques d’Orsay, ParMA, Inria Saclay, 91405, Orsay, France /
Institut universitaire de France (IUF)} 
\email{luca.nenna@universite-paris-saclay.fr}

\begin{abstract}
We study a nonlinear multimarginal optimal transport problem arising in risk management, where the objective is to maximize a spectral risk measure of the pushforward of a coupling by a cost function. Although this problem is inherently nonlinear, it is known to have an equivalent linear reformulation as a multimarginal transport problem with an additional marginal.
We introduce a Lagrangian particle discretization of this problem, in which admissible couplings are approximated by uniformly weighted point clouds, and marginal constraints are enforced through Wasserstein penalization. We prove quantitative convergence results for this discretization as the number of particles tends to infinity. The convergence rate is shown to be governed by the uniform quantization error of an optimal solution, and can be bounded in terms of the geometric properties of its support, notably its box dimension.
In the case of univariate marginals and supermodular cost functions, where optimal couplings are known to be comonotone, we obtain sharper convergence rates expressed in terms of the asymptotic quantization errors of the marginals themselves.
We also discuss the particular case of conditional value at risk, for which the problem reduces to a multimarginal partial transport formulation. Finally, we illustrate our approach with numerical experiments in several application domains, including risk management and partial barycenters, as well as some artificial examples with a repulsive cost.
\end{abstract}

\maketitle

\section{Introduction}

Many problems in risk management require estimating a specific joint law of several known individual distributions. Typically, we have a function $c : \X_1\times\dots\times\X_D\to\R$ which takes as argument $D$ different risk factors, all of which are modeled as random variables $X_j$ with values in $\X_j$, and whose output is a scalar value representing some level of danger.
For instance, $X_1,\dots,X_D$ could be the characteristics of a river (length, width, maximal annual flow rate, etc.) as well as some given design parameters (height of the dyke that surrounds the watercourse), and $c$ some function giving the height of the river in terms of the possible values for these input variables. The randomness of the latter stems both from inaccuracies in their measurements and from their variability in time and/or space.
We refer to the article~\cite{iooss2015review} of Iooss and Lemaître, which is mainly concerned with the topic of sensitivity analysis, for the details of this particular setting. In the latter, the input variables are assumed to be independent, but finding out the \textit{riskiest} dependence structure is an interesting problem in itself, both mathematically speaking and in terms of applications to risk management.

In Iooss and Lemaître's example, several industrial facilities are located near the river, along which an artificial dyke has been built as a precaution. But alas, there is still a risk of flooding. Given some way to quantify the danger of any generic joint law of $X_1,\dots,X_D$, one may therefore may therefore be interested in the worst-case scenario, in order to get a deeper understanding of the different risk factors at stake.
In~\cite{ennaji2024robust}, the authors study the problem in question when the risk is quantified by a generic \textit{spectral risk measure}.

\subsection{Spectral risk measures}

We define the spectral risk measure as a real-valued functional on $\P(\R)$ of the form
\begin{equation}
\Rr_\alpha(\mu) = \int_0^1 F_\mu^{-1}(t) \alpha(t) dt,
\end{equation}
with $\alpha : (0,1) \to [0,+\infty)$ a bounded, nondecreasing, nonnegative function of integral one, and $F_\mu^{-1}$ the quantile function (or inverse cdf) of $\mu$.
Spectral risk measures are a specific form of weighted averages that are biased towards high values. In particular, for any non-constant spectral function $\alpha$, the functional $\Rr_\alpha$ is nonlinear.
Note that for $\alpha \equiv 1$, we recover the expected value $\E\mu = \int_\R z\,d\mu(z)$, since any univariate probability measure is the pushforward of the uniform measure on $(0,1)$ by its quantile function. Another well-known risk measure is the \textit{conditional value at risk} $\mathrm{CVaR}_m$ at level $m \in (0,1)$, obtained via $\alpha \propto \mathbbm{1}_{(1-m,1)}$.
The $\mathrm{CVaR}$, also known as \textit{expected shortfall} or \textit{superquantile}, is the expected value of the (normalized) restriction of $\mu$ to its fraction of mass $m$ with highest values, or in other words, to its right-hand side tail of mass~$m$, as illustrated in Figure~\ref{fig:expected_shortfall}.
We refer to~\cite{rockafellar2014random} for a more comprehensive but concise introduction to the notion of spectral risk measures and their motivating connections with risk.

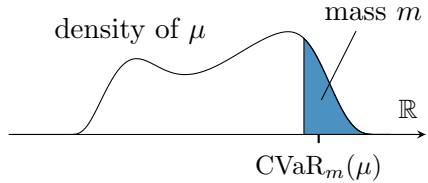
\begin{figure}[!ht]
\centering
\input{expected_shortfall.tikz}
\caption{Conditional value at risk at level $m$, for an absolutely continuous probability measure $\mu$.}
\label{fig:expected_shortfall}
\end{figure}

\subsection{The risk estimation problem}

Denoting $\rho_j \in \P(\X_j)$ the respective marginals laws of $X_1,\dots,X_D$, the problem studied in~\cite{ennaji2024robust} reads
\begin{equation}
\label{eq:intro-risk_problem}
\tag{$\mathcal{P}$}
\max_{\gamma\in\Gamma(\rho_1,\dots,\rho_D)} \Rr_\alpha(c_\#\gamma),
\end{equation}
where $\Gamma(\rho_1,\dots,\rho_D)$ is the set of probability measures on the product space $\boldsymbol{\X} := \X_1\dots\times\dots\X_D$ whose projections on the $\X_j$ are the corresponding $\rho_j$.
Under very general assumptions, the existence of a maximizer is guaranteed by a standard compactness argument. Note that due to the generic spectral risk measure, the maximized functional $\gamma \mapsto \Rr_\alpha(c_\#\gamma)$ is a priori nonlinear. This is a major difference with respect to standard multimarginal optimal transport, for which the objective function $\gamma \mapsto \int c\,d\gamma$ is linear, and corresponds to the trivial spectral function $\alpha = 1$.

\subsubsection*{Formulation as a linear multimarginal transport problem}
As shown in~\cite{ennaji2024robust}, it turns out that~\eqref{eq:intro-risk_problem} is equivalent to the following \textit{linear} multimarginal optimal transport problem
\begin{equation}
\label{eq:intro-equivalent_linear_problem}
\tag{$\mathcal{L}$}
\max_{\eta\in\Gamma(\rho_0,\rho_1,\dots,\rho_D)}
\int_{\R\times\boldsymbol{\X}}
\omega c(x) \,d\eta(\omega,x),
\end{equation}
where the additional marginal is the probability measure on $\R$ defined by $\rho_0 := \alpha_\#\L_{(0,1)}$, and the variable $x = (x_1,\dots,x_D)$ denotes a point in the product space $\boldsymbol{\X}$.
The equivalence essentially stems from the variational formulation of spectral risk measures, which reads
\begin{equation*}
\Rr_\alpha(\mu) =
\max_{\tau \in \Gamma(\rho_0,\mu)} \int_{\R\times\R} \omega z \, d\tau(\omega,z).
\end{equation*}
Indeed, gluing $\gamma$ and $\tau$ along the intermediate marginal $\mu = c_\#\gamma$ allows to maximize a linear objective over a single probability measure $\eta \in \Gamma(\rho_0,\rho_1,\dots,\rho_D)$, which contains both the coupling of the $D$ main marginals and the distribution of the weight $\alpha$ over this coupling (thanks to the additional marginal $\rho_0$).
As highlighted in~\cite{ennaji2024robust}, the above formulation allows existing results on linear multimarginal transport to be used in order to better understand the structure of the solutions to~\eqref{eq:intro-risk_problem}.

\subsection{Discretized problem}

The solutions of optimal transport problems are typically supported on sets of dimension much smaller than that of the ambient space --- especially in multimarginal transport. This renders Eulerian grid-based methods inherently inefficient.
In this paper, we opt for a simple Lagrangian particle discretization, where the admissible discrete measures are those of the form $\delta_Y :=\frac{1}{N}\sum_{i=1}^N\delta_{y_i}$, with $Y = (y_1,\dots,y_N) \in \mathcal{\boldsymbol{\X}}^N$ a cloud of $N$ points in the product space $\X$ whose positions are to be optimized, and $N$ a fixed positive integer.
The marginal constraints are relaxed through quadratic Wasserstein penalization terms, each of which corresponds to the squared Wasserstein distance from some prescribed measure $\rho_j$ to the projection of $\delta_X$ on $\X_j$.
The discretized version of~\eqref{eq:intro-risk_problem} reads as the following finite-dimensional problem, which we prefer to formulate as a minimization problem,
\begin{equation}
\label{eq:intro-discretized_risk_problem}
\tag{$\mathcal{P}_N$}
\min_{Y \in \boldsymbol{\X}^N} -\Rr_\alpha(c_\#\delta_Y) + \lambda_N \sum_{j=1 }^D W_p(\rho_j,\pi^j_\#\delta_Y)^p.
\end{equation}
Here, $\pi^j : \boldsymbol{\X} \to \X_j$ denotes the projection on the $j$-th factor of the product space $\boldsymbol{\X} = \X_1\times\dots\times\X_D$, and $\lambda_N > 0$ is a penalty parameter, whose adequate value will be discussed when dealing with the convergence of~\eqref{eq:intro-risk_problem} to~\eqref{eq:intro-discretized_risk_problem} as $N \to \infty$.
Note that $\Rr_\alpha(c_\#\delta_Y)$ is straightforward to compute numerically. Indeed, it is by definition a weighted sum of the $c(y_i)$, where the weights depend on the order of these values. Denoting $\mathrm{p}_i = \int_{(\frac{i-1}{N},\frac{i}{N})}\alpha(t)dt$, we have
\begin{align*}
\Rr_\alpha(c_\#\delta_Y)
&= \sum_{i=1}^N
\mathrm{p}_{\sigma_{Y}(i)} c(y_i),
\end{align*}
where $\sigma_Y \in \mathfrak{S}_N$ is any permutation such that
$c(y_{\sigma_Y^{-1}(1)}) \leq \dots \leq c(y_{\sigma_Y^{-1}(N)})$.

\subsection{Convergence results}

The convergence of~\eqref{eq:intro-discretized_risk_problem} to the original problem is closely linked to the uniform quantization of a solution of~\eqref{eq:intro-risk_problem}. Our two convergence results are the following, and respectively correspond to Corollary~\ref{cor:convergence_with_dimension} and Corollary~\ref{cor:convergence_with_comonotonicity} of Section~\ref{sec:discretization_and_convergence}.
We denote
\begin{equation*}
\tau_{p,d}(N) =
\begin{cases}
N^{-\frac{1}{\max(p,d)}}
\quad &\mathrm{if} \quad d \neq p,
\\
(\log N)^{\frac{1}{d}} N^{-\frac{1}{d}}
\quad &\mathrm{if} \quad d = p.
\end{cases}
\end{equation*}

\begin{theorem}
\label{meta-thm-1}
Suppose that the marginal distributions $\rho_j$ are compactly supported and that $c$ is $\beta$-Hölder continuous. Let $d$ be the box dimension (see Definition~\ref{def:box_dimension}) of the support of some minimizer of $\F$.
Then, for the sequence of penalty coefficients $\lambda_N = \tau_{p,n}(N)^{-(p-\beta)}$,
we have
\begin{equation}
|\min\F_N - \min\F| = \O
\begin{cases}
N^{-\frac{\beta}{d}}
\quad &\mathrm{if} \quad d > p,
\\
(\log N)^{\frac{\beta}{d}} N^{-\frac{\beta}{d}}
\quad &\mathrm{if} \quad d = p,
\\
N^{-\frac{\beta}{p}}
\quad &\mathrm{if} \quad d < p.
\end{cases}
\end{equation}
Moreover, any weak limit point of an arbitrary sequence of minimizers $\delta_{Y_N} \in \argmin\F_N$ is a minimizer of $\F$.
\end{theorem}

We can be more precise when the marginals are univariate and the cost function $c$ is supermodular (see Definition~\ref{def:supermodularity}).

\begin{theorem}
\label{meta-thm-2}
Suppose that the marginal distributions $\rho_j$ are univariate, i.e. $\X_j \inc \R$, and that $c$ is $\beta$-Hölder and supermodular. For a probability measure $\nu$ on a Euclidean space, denote
\begin{equation}
e_{p,N}(\nu) =
\min_{z_1,\dots,z_N}
W_p\left(\nu,N^{-1}\sum_{i=1}^N\delta_{z_i}\right)
\end{equation}
its $N$-point uniform quantization error.
Then, for the sequence of penalty coefficients $\lambda_N = h_N^{-(p-\beta)}$,
where
\begin{equation}
h_N = \max_{j \in \{1,\dots,D\}} e_{p,N}(\rho_j),
\end{equation}
we have
\begin{equation}
|\min\F_N - \min\F| = \O(h_N^\beta).
\end{equation}
Moreover, any weak limit point of an arbitrary sequence of minimizers $\delta_{Y_N} \in \argmin\F_N$ is a minimizer of $\F$.
\end{theorem}

\begin{remark}
In Theorem~\ref{meta-thm-2}, if the $\rho_j$ are moreover assumed to have connected support and to be absolutely continuous with density isolated from $0$ and $\infty$ on their respective supports, then $h_N \lesssim N^{-1}$ and so the infimum of $\F_N$ converges to that of $\F$ at the rate $N^{-\beta}$.
\end{remark}

\subsubsection*{Contributions}
In this paper, we present a particle discretization method for a nonlinear multimarginal transport problem studied in~\cite{ennaji2024robust}, where the maximized quantity is some given spectral risk measure of the pushforward of the plan by a cost function $c$.
We describe the Lagrangian particle discretization in question, and prove a general quantitative convergence rate
of the discretized Problem~\eqref{eq:intro-discretized_risk_problem} to the original Problem~\eqref{eq:intro-risk_problem} as the number of particles goes to infinity. The convergence rate in question corresponds to that of the optimal uniform quantization error of some arbitrary solution of~\eqref{eq:intro-risk_problem}, as the number of Dirac masses goes to infinity. Using results by Mérigot and Mirebeau in~\cite{merigot2016minimal} on the uniform quantization of measures, we find an upper bound for the convergence rate in terms of the box dimension of the support of a solution of~\eqref{eq:intro-risk_problem}.
In~\cite{ennaji2024robust}, the authors prove that under supermodularity and monotonicity assumptions on the cost function $c$, the comonotone plan solves~\eqref{eq:intro-risk_problem}. We leverage this result to provide a more explicit upper bound on the convergence rate, in terms of the asymptotic optimal uniform quantization errors of the marginals $\rho_j$, under these more restrictive assumptions (Theorem~\ref{meta-thm-1}).
We then deal with the specific case where the spectral risk measure is the conditional value at risk, a case for which our discretization is of particular interest. In this framework, Problem~\eqref{eq:intro-risk_problem} becomes an instance of the multimarginal partial optimal transport problem.
At last, we show our numerical simulations for a few interesting cases, in risk management, partial barycenters, and density functional theory.

\section{Preliminaries and problem setting}

In this section, we first recall the definition of spectral risk measures, as well as their variational formulation. We then give the corresponding reformulation of Problem~\eqref{eq:intro-risk_problem} into an ordinary multimarginal transport Problem~\eqref{eq:intro-equivalent_linear_problem} with an additional marginal induced by the risk measure, an equivalence first established in~\cite{ennaji2024robust}.

\subsubsection*{Assumptions}
Throughout the whole paper, the $\X_j \inc \R^{d_j}$ are convex subsets of Euclidean spaces (with potentially different dimensions $d_j$) and the $\rho_j \in \P_p(\X_j)$ are probability measures with finite $p$-moment, where $p \in [1,\infty)$ is a fixed finite exponent.
As a direct consequence, the product set $\boldsymbol{\X} = \X_1\times\dots\times\X_D$ is also convex --- we naturally endow it with Euclidean norm of the embedding product space $\R^{d_1}\times\dots\times\R^{d_D}$ ---
and we have the inclusion $\Gamma(\rho_1,\dots,\rho_D) \inc \P_p(\boldsymbol{\X})$, i.e. any joint measure of the $\rho_j$ also has finite $p$-moment.
For now, we only assume the cost function $c : \boldsymbol{\X} \to \R$ to be measurable.

\medskip

In the introduction, we assumed that the spectral function $\alpha$ was bounded. In fact, up to restricting the domain of $\Rr_\alpha$ to a smaller set of probability measure, it suffices to assume $\alpha$ is in some Lebesgue space $L^q(0,1)$, $q \in (1,\infty]$.
In the remainder of this article, we adopt this more general framework, which accommodates spectral risk measures that exhibit an even stronger preference for large values.

\begin{definition}
\label{def:spectral_risk_measure}
A \textnormal{spectral risk measure} is a functional $\Rr_\alpha : \P_{\bar{q}}(\R) \to \R$ of the form
\begin{equation}
\label{eq:spectral_risk_measure}
\Rr_\alpha(\mu)
= \int_0^1 F_\mu^{-1}(t)\alpha(t)dt,
\end{equation}
where $\alpha : (0,1) \to [0,+\infty)$ is a nonnegative, nondecreasing function of integral one in $L^q(0,1)$ for some $q \in (1,\infty]$, called \textnormal{spectral function}. Hölder's inequality ensures that $\Rr_\alpha$ is indeed well-defined on $\P_{\bar{q}}(\R)$, where $\bar{q}$ is the conjugate Hölder exponent of $q$.
\end{definition}

\begin{remark}[Domain of the spectral risk measure]
If $\alpha \in L^q(0,1)$ for some $q \in (1,\infty]$, the right-hand side of~\eqref{eq:spectral_risk_measure} is well-defined and finite for a much larger collection of probability measures than $\P_{\bar{q}}(\R)$.
For instance, since $\alpha$ is both nonnegative and nondecreasing, it suffices that $\mu$ has left-tail with finite first moment and right-tail with finite $\bar{q}$-moment for the integral in~\eqref{eq:spectral_risk_measure} to be well-defined.
Moreover, decomposing quantile functions into negative and positive parts and using $F_{\mu_s}^{-1} \leq \max\{F_{\mu_0}^{-1}, F_{\mu_1}^{-1}\}$ for $\mu_s = (1-s)\mu_0+s\mu_1$, it is straightforward to show that the domain of any spectral risk measure is a convex subset of $\P(\R)$.
\end{remark}

We will sometimes refer to $\Rr_\alpha$ as the \textit{$\alpha$-risk}, as $\alpha$ uniquely determines the spectral risk measure.
Note that by monotonicity of $\alpha$, its set of discontinuity points is at most countable, and thus Lebesgue negligible, which means $\Rr_\alpha$ is unaffected by the values at these points.

\begin{remark}
For $\alpha\equiv 1$ we naturally retrieve the expected value, since $(F_\mu^{-1})_\#\mathcal{L}_{(0,1)} = \mu$. On the other hand, taking $\alpha_m = m^{-1}\mathbbm 1_{(1-m,1)}$ yields the \textit{conditional value at risk} (or \textit{expected shortfall}) at level $m \in (0,1]$, which we denote $\mathrm{CVar}_{m}(\mu)$. The latter is defined as the expected value of the (normalized) restriction of $\mu$ to its fraction of mass $m$ with highest values.
\end{remark}

Spectral risk measures have the following well known variational characterization in terms of one-dimensional optimal transport with the standard bilinear cost.

\begin{lemma}
\label{lem:spectral_risk_measure_and_OT}
Let $\Rr_\alpha$ be a spectral risk measure and denote $q \in (1,\infty]$ an exponent such that $\alpha \in L^q(0,1)$. Then, for any $\mu \in \P_{\bar{q}}(\R)$ we have
\begin{equation}
\label{eq:spectral_variational}
\Rr_\alpha(\mu) =
\max_{\tau\in\Gamma(\alpha_{\#}\L_{(0,1)},\mu)}
\int_{\R\times\R} z \omega \,d\tau(\omega,z).
\end{equation}
In particular, $\Rr_\alpha$ is concave on $\P_{\bar{q}}(\R)$.
\end{lemma}

\subsubsection*{Equivalence of~\eqref{eq:intro-risk_problem} with a linear problem}
Before addressing Lagrangian discretization, let us briefly precise the equivalence of Problem~\eqref{eq:intro-risk_problem} with a linear transport problem that has an additional marginal. This equivalence was first highlighted by Ennaji et al. (see~\cite[Theorem~18]{ennaji2024robust}) and is a direct consequence of the variational formulation described in Lemma~\ref{lem:spectral_risk_measure_and_OT} for spectral risk measures.

\begin{proposition}[{\cite[Theorem~18]{ennaji2024robust}}]
\label{prop:linear_risk_problem_equivalence}
Consider the linear multimarginal problem
\begin{equation}
\label{eq:equivalent_linear_problem}
\tag{$\mathcal{L}$}
\max_{\eta\in\Gamma(\rho_0,\rho_1,\dots,\rho_D)}
\int_{\R\times\boldsymbol{\X}} s\,d\eta,
\end{equation}
where $\rho_0 = \alpha_\#\L_{(0,1)} \in \P(\R)$ and $s(\omega,x) := \omega \,c(x)$ for $(\omega,x) \in \R\times\boldsymbol{\X}$.
A probability measure $\eta$ is a solution to~\eqref{eq:equivalent_linear_problem} if and only if the following conditions hold:
\begin{itemize}
\item
the probability measure 
$\gamma := \pi^{\boldsymbol{\X}}_\#\eta \in \Gamma(\rho_1,\dots,\rho_D)$
is a solution to Problem~\eqref{eq:intro-risk_problem},
\item
the pushforward
$\tau := (\pi^0,c\circ\pi^{\boldsymbol{\X}})_\#\eta \in \Gamma(\rho_0,c_\#\gamma)$
has monotone increasing support,
\end{itemize}
where 
$\pi^{\boldsymbol{\X}}(\omega,x) := x_j \in \X_j$
and
$\pi^0(\omega,x) := \omega \in \R$
are the canonical projections for the product space $\R\times\boldsymbol{\X} = \R\times\X_1\times\dots\times\X_D$.
\end{proposition}

\section{Discretization and convergence results}
\label{sec:discretization_and_convergence}

We wish to numerically solve the following \textit{convex} optimization problem
\begin{equation}
\label{eq:risk_problem}
\tag{$\mathcal{P}$}
\max_{\gamma \in \Gamma(\rho_1,\dots,\rho_D)}
\Rr_\alpha(c_\#\gamma).
\end{equation}
The Lagrangian particle discretization scheme we investigate reads
\begin{equation}
\label{eq:discretized_risk_problem}
\tag{$\mathcal{P}_N$}
\min_{Y \in \boldsymbol{\X}^N} -\Rr_\alpha(c_\#\delta_Y) + \lambda_N \sum_{j=1 }^D W_p(\rho_j,\pi^j_\#\delta_Y)^p,
\end{equation}
with $\lambda_N > 0$ a penalty parameter whose dependence on the number of particles $N$ we shall deal with in the present section.
Note that thanks to the $\rho_j$ having finite $p$-moment, the penalty terms are all finite.

In order to prove (some notion of) convergence of~\eqref{eq:discretized_risk_problem} to~\eqref{eq:risk_problem}, one needs some control on the spectral risk measures of the $c$-pushforwards in terms of the joint measures themselves. To this end, we make the following assumptions.

\begin{assumption}
\label{a:p_Lq_and_beta-Hölder}
We have $p \in [1,\infty)$, and there exists $q \in (1,\infty)$ and $\beta \in (0,1]$ such that
\begin{enumerate}
\item
$\alpha \in L^q(0,1)$,
\item
$c$ is $\beta$-Hölder, with Hölder constant $C_{\mathrm{H}}$,
\item
$\frac{\beta}{p} + \frac{1}{q} \leq 1$.
\end{enumerate}
We denote $\bar{q} \in (1,\infty)$ the conjugate Hölder exponent of $q$, i.e.
$1/q + 1/\bar{q} = 1$.
\end{assumption}

\begin{remark}
Under these assumptions, Proposition~\ref{prop:linear_risk_problem_equivalence} still holds without any boundedness assumption on $\alpha$.
\end{remark}

Unless stated otherwise, the symbol $\lesssim$ will hide positive constants that only depend on $D$, $p$, $q$, $\alpha$, $\beta$, and $c$.
The following lemma, valid under the above assumptions, is a key ingredient for the convergence results that will follow. It is merely an extension of the Lipschitz continuity results in~\cite[Lemma~33,~Proposition~34]{ennaji2024robust} to the more general case of Hölder cost functions and $L^p$ spectral functions.

\begin{lemma}[Hölder continuity of the objective function]
\label{lem:control_on_the_spectral_risk_measure}
Suppose that Assumption~\ref{a:p_Lq_and_beta-Hölder} holds.
Then, $\Rr_\alpha(c_\#\cdot)$ is finite on the convex set $\P_p(\boldsymbol{\X})$, and for any $\gamma,\tilde{\gamma} \in \P_p(\boldsymbol{\X})$ we have
\begin{equation}
\label{ineq:control_on_the_spectral_risk_measure}
|\Rr_\alpha(c_\#\gamma) - \Rr_\alpha(c_\#\tilde{\gamma})|
\leq
\|\alpha\|_{L^q} C_{\mathrm{H}}
W_p(\gamma,\tilde{\gamma})^\beta.
\end{equation}
In other words, the map $\gamma \mto \Rr_\alpha(c_\#\gamma)$ is both finite and $\beta$-Hölder on the $p$-Wasserstein space on $\boldsymbol{\X}$, with Hölder constant $\|\alpha\|_{L^q}C_{\mathrm{H}}$.
\end{lemma}
\begin{proof}
We first show that $c$ pushes forwards any measure with finite $p$-moment to a measure with finite 
$\bar{q}$-moment. Indeed, for $\gamma \in \P_p(\boldsymbol{\X})$, the $\bar{q}$-moment of $c_\#\gamma$ with respect to any given $x_o \in \boldsymbol{\X}$ writes
\begin{align*}
\int_{\boldsymbol{\X}} |c(x)-c(x_o)|^{\bar{q}} d\gamma(x)
&\leq
C_{\mathrm{H}}\int_{\boldsymbol{\X}} |x-x_o|^{\beta \bar{q}} d\gamma(x).
\end{align*}
The last integral is equal to $W_{\beta\bar{q}}(\gamma,\delta_{x_o})$ to the power $\beta\bar{q}$, and since the third point of Assumption~\ref{a:p_Lq_and_beta-Hölder} reads $\beta\bar{q} \leq p$, this Wasserstein distance is bounded by the finite quantity $W_p(\gamma,\delta_{x_o})$. This shows that $c_\#\gamma \in \P_{\bar{q}}(\boldsymbol{\X})$, which implies $\Rr_\alpha(c_\#\gamma)$ is finite thanks to Hölder's inequality and to $\alpha$ being $L^q$.
If $\tilde{\gamma}$ is another measure in $\P_p(\boldsymbol{\X})$, Hölder's inequality also yields
\begin{align*}
|\Rr_\alpha(c_\#\gamma) - \Rr_\alpha(c_\#\tilde{\gamma})|
&\leq
\int_0^1 \alpha |F_{c_\#\gamma^\ast}^{-1}-F_{c_\#\gamma_N}^{-1}|
\\
&\leq
\|\alpha\|_{L^q}
\|F_{c_\#\gamma}^{-1} - F_{c_\#\tilde{\gamma}}^{-1}\|_{L^{\bar{q}}}.
\end{align*}
Since
$\|F_{c_\#\gamma}^{-1}-F_{c_\#\tilde{\gamma}}^{-1}\|_{L^{\bar{q}}} = W_{\bar{q}}(c_\#\gamma,c_\#\tilde{\gamma})$, it remains to control the distance between the two pushforwards. Precisely, $\beta$-Hölder regularity of $c$ implies that
$W_{\bar{q}}(c_\#\gamma,c_\#\tilde{\gamma})
\leq C_\mathrm{H}
W_{\beta\bar{q}}(\gamma,\tilde{\gamma})^\beta$,
and since by assumption $\beta\bar{q} \leq p$ the desired inequality follows.
\end{proof}

To study convergence of~\eqref{eq:discretized_risk_problem} to~\eqref{eq:risk_problem}, define the functionals
$\F, \F_N : \P(\boldsymbol{\X})
\rightarrow \R\cup\{+\infty\}$
by
\begin{align*}
\begin{cases}
\F(\gamma)
&= -\Rr_\alpha(c_\#\gamma)
+ \chi_{\Gamma(\rho_1,\dots,\rho_D)}(\gamma),
\\[4pt]
\F_N(\gamma)
&= -\Rr_\alpha(c_\#\gamma)
+ \lambda_N \sum_{j=1}^D W_p(\rho_j,\pi^j_\# \gamma)^p
+ \chi_{\D_N}(\gamma),
\end{cases}
\end{align*}
where
$\D_N = \{\delta_Y : Y \in \boldsymbol{\X}^N\}$
is the domain of the discretized functional, consisting of all uniform clouds of $N$ points in the product space. Of course, Problems~(\ref{eq:risk_problem}) and~(\ref{eq:discretized_risk_problem}) respectively correspond to the minimization of $\F$ and $\F_N$. Here, the notation $\chi_E$ corresponds to the characteristic function of the set $E$ in the sense of convex analysis: $\chi_E(x)=0$ if $x\in E$, $+\infty$ otherwise.

\begin{proposition}[Existence of solutions]
\label{prop:existence_of_solutions}
Under Assumption~\ref{a:p_and_beta-Hölder}, Problems~\eqref{eq:discretized_risk_problem} and~\eqref{eq:risk_problem} both have at least one solution.
\end{proposition}
\begin{proof}
Thanks to finite $p$-moment of the marginal distributions $\rho_j$ and to continuity of the objective function with respect to $W_p$, existence of a solution to~\eqref{eq:risk_problem} follows from standard arguments.
To see that~\eqref{eq:discretized_risk_problem} also has a solution, consider a minimizing sequence $\{\delta_{Y^\ell}\}_{\ell=0}^\infty \inc \D_N$. By coercivity of the penalty terms with respect to the positions of the particles, the latter remain in some fixed compact set throughout the sequence.
Up to extraction of a subsequence, we can therefore assume that $Y^\ell$ converges to some limit cloud $Y$ as $\ell \to \infty$, and uniform boundedness of the $Y^\ell$ ensures that the corresponding probability measures $\delta_{Y^\ell}$ converge to $\delta_Y$ for the $W_p$ distance.
Convergence of the values $\F_N(\delta_{Y^\ell})$ to $\F_N(\delta_Y)$ then directly follows from continuity of $\Rr_\alpha(c_\#\cdot)$ and of projections on subspaces, with respect to the $p$-Wasserstein distance.
\end{proof}

\subsection{Convergence of the discretized problems}
\label{susbec:convergence_standard}

We start by analyzing the two different parts for the study of the convergence.
In particular in order to get an estimation of $(\min\F_N - \min\F)$ we provide an upper and a lower bound and then optimize them in order to find the optimal $\lambda_N$ giving the matching bound. We describe briefly the strategy of the proof before detailing the rigorous step:

$\bullet$ {\bf Upper bound.} To estimate $(\min\F_N - \min\F)$, one must approach a fixed minimizer $\gamma^* \in \argmin \F$ by some $N$-point uniform cloud, and thanks to the control on the spectral risk measures given by Lemma~\ref{lem:control_on_the_spectral_risk_measure}, it suffices that the point cloud be close to $\gamma^*$ for the $p$-Wasserstein distance. The penalty terms will indeed also be controlled, since projecting two measures one some common set can only decrease their Wasserstein distance. In the end, we get a bound of the form
\begin{equation}
\label{eq:general_bound_on_FN-F}
\min\F_N - \min\F
\lesssim
u_N^\beta
+ \lambda_N u_N^p,
\end{equation}
where $u_N$ is the $p$-Wasserstein distance from the considered $N$-point uniform cloud to $\gamma^*$.
To minimize the right-hand side, one first naturally takes the point cloud which minimizes the latter distance, so that $u_N$ is the optimal $N$-point uniform quantization error of $\gamma^*$ with respect to $W_p$.

Note that since the minimizer $\gamma^*$ is arbitrary, properties of a \textit{single} minimizer can be harnessed to obtain asymptotic bounds.

$\bullet$ {\bf Lower bound.} On the other hand, lower bounding $(\min\F_N - \min\F)$ requires to approach a point cloud $\delta_{Y^*_N} \in \argmin\F_N$ minimizing the discretized functional by some joint measure $\gamma$ of the marginals $\rho_j$. 
We exploit the discrete nature of the point cloud to construct a particular joint measure $\gamma_N$ whose distance to the point cloud is controlled by that of the respective projections.
This constructed measure is a block approximation made of $N$ equal-mass blocks --- one for each of the $N$ points of the cloud --- induced by the $W_p$-optimal transports from the $\rho_j$ to the respective projections of the point cloud.
Plugging in $\gamma_N$ to upper bound the minimum of $F$, we thus have
\begin{equation}
\label{eq:general_bound_on_F-FN}
\min\F_N - \min\F
\gtrsim
\lambda_N v_N^p
- Cv_N^\beta,
\end{equation}
where $v_N$ denotes the $p$-Wasserstein distance from $\gamma_N$ to $\delta_{Y_N^*}$ and $C > 0$ is some positive constant depending on the cost function $c$, on the spectral function $\alpha$, and on $D$ the number of marginals.

$\bullet$ {\bf Matching bound.} Taking now $\lambda_N=\big(\frac{1}{u_N}\big)^{p-\beta}$ in the upper bound and using mainly Young's inequality in the lower bound, \eqref{eq:general_bound_on_FN-F} and \eqref{eq:general_bound_on_F-FN} rewrite respectively
\[
\min\F_N - \min\F \lesssim u_N^\beta,
\]
\[
\min\F_N - \min\F
\gtrsim
-\lambda_N^{-\frac{\beta}{p-\beta}},
\]
from which the matching bound easily follows
\[\boxed{|\min\F_N - \min\F|=\O(u_N^\beta).}\]

We now explicitly construct the joint measure of the $\rho_j$ used to approach a uniform cloud of $N$ points.

\begin{lemma}[Block approximation]
\label{lem:N-block_joint_measure}
Let $\delta_Y \in \D_N$ be a uniform cloud of $N$ points in $\boldsymbol{\X}$, and denote $\delta_{Y_j}$ its projection on $\X_j$. There exists a joint measure $\gamma \in \Gamma(\rho_1,\dots,\rho_D)$ such that
\begin{equation}
W_p(\gamma,\delta_Y)
\lesssim
\left(\sum_{j=1}^D W_p(\rho_j,\delta_{Y_j})^p\right)^{\frac{1}{p}}.
\end{equation}
\end{lemma}
\begin{proof}
We denote $y_1,\dots,y_N \in \boldsymbol{\X}$ the points of the cloud. For each $j$, denote $\rho_j = \sum_{i=1}^N\rho_{i,j}$ the uniform decomposition of $\rho_j$ induced by its $W_p$-optimal transport to $\delta_{Y_j}$. That is, the positive measure $\rho_{i,j}$ of mass $1/N$ is sent to $y_{i,j} \in \X_j$, the $j$-th coordinate of $y_i$. We patch together the components corresponding to the same index $i$ by defining
$\gamma= N^{D-1}\sum_{i=1}^N
(\rho_{i,1}\otimes\dots\otimes\rho_{i,D})$, where $N^{D-1}$ is a normalization constant.
Then $\gamma \in \Gamma(\rho_1,\dots,\rho_D)$, and its Wasserstein distance to the point cloud $\delta_Y$ can be estimated by sending each patch to its corresponding Dirac mass.
To decouple the $D$ components in the distance from a point $x$ to some Dirac position $y_i$, we use equivalence of norms in $\R^D$, which yields
\begin{align*}
W_p(\gamma,\delta_Y)^p
&\leq N^{D-1} \sum_{i=1}^N \int_{\boldsymbol{\X}} \|x-y_i\|^p d(\rho_{i,1}\otimes\dots\otimes\rho_{i,D})(x)
\\
&\lesssim N^{D-1} \sum_{i=1}^N
\int_{\boldsymbol{\X}} \sum_{j=1}^D \|x_j-y_{i,j}\|^p d(\rho_{i,1}\otimes\dots\otimes\rho_{i,D})(x)
\\
&= \sum_{j=1}^D \left(
\sum_{i=1}^N \int_{\X_j}
\|x_j-y_{i,j}\|^p d\rho_{i,j}(x_j) \right)
\\
&= \sum_{j=1}^D
W_p(\rho_j,\delta_{Y_j})^p,
\end{align*}
where, in the last line, we used the fact that sending each $\rho_{i,j}$ to $y_{i,j}$ is a $W_p$-optimal transport from $\rho_j$ to $\delta_{Y_j}$.
Note that the choice of the product measure for each block is not important for the proof, as long as each block has the correct marginals.
\end{proof}

Thanks to the considerations in the beginning of this subsection, we have the following convergence result.

\begin{theorem}
\label{thm:general_convergence_result}
Suppose that Assumption~\ref{a:p_Lq_and_beta-Hölder} holds, and let $u_N$ be some upper bound on the $W_p$-optimal uniform quantization error for some fixed minimizer $\gamma^*$ of $\F$. Then, letting $\lambda_N = u_N^{-(p-\beta)}$, we have
\begin{equation}
\label{eq:general_asymptotic_rate}
|\min \F_N - \min \F| = \O(u_N^\beta).
\end{equation}
Moreover, for any sequence of minimizers $\delta_{Y_N^*} \in \argmin \F_N$, $N \in \N$, every limit point is a minimizer of $\F$, and for every $j \in \{1,\dots,D\}$ we have
\begin{equation}
\label{ineq:estimate_for_error_on_marginals}
W_p(\rho_j, \pi^j_\#\delta_{Y_N^*})
= \O(u_N).
\end{equation}
\end{theorem}
\begin{proof}
The estimate on $(\min\F_N - \min\F)$ has already been dealt with in the beginning of this section, so we focus on the full proof for the upper bound on $(\min\F - \min\F_N)$.
Take $\delta_{Y_N^*}$ a minimizer of $\F_N$ and let $\gamma_N \in \Gamma(\rho_1,\dots,\rho_D)$ be the joint measure given by Lemma~\ref{lem:N-block_joint_measure}. Thanks to the latter and to $\beta$-Hölder continuity of $\Rr_\alpha(c_\#\cdot)$ derived in Lemma~\ref{lem:control_on_the_spectral_risk_measure}, we have
\begin{align*}
\min\F_N - \min\F
\gtrsim
-C W_p(\gamma_N,\delta_{Y_N^*})^\beta
+ \lambda_N W_p(\gamma_N,\delta_{Y_N^*})^p
\end{align*}
where $C$ is constant that only depends on $p$, $q$, $\alpha$, and $c$.
We then apply Young's inequality with exponents $p/\beta$ and its conjugate to find
\begin{align*}
CW_p(\gamma_N,\delta_{Y_N^*})^\beta
&=
\left(\frac{p}{\beta}\lambda_N W_p(\gamma_N,\delta_{Y_N^*})^p\right)^{\frac{\beta}{p}}
\left(\frac{p}{\beta}\lambda_N C^{-\frac{p}{\beta}}\right)^{-\frac{\beta}{p}}
\\
&\leq \lambda_N W_p(\gamma_N,\delta_{Y_N^*})^p
+ \frac{p-\beta}{p} \left(\frac{p}{\beta}\lambda_N C^{-\frac{p}{\beta}}\right)^{-\frac{\beta}{p-\beta}},
\end{align*}
and the estimate~\eqref{eq:general_asymptotic_rate} follows from $\lambda_B = u_B^{-(p-\beta)}$.

Let us now prove the bound~\eqref{ineq:estimate_for_error_on_marginals} on the marginal discrepancies. We denote
$S_N = \sum_{j=1}^D W_p(\rho_j,\pi^j_\#\delta_{Y_N^*})^p$.
Thanks to $S_N \leq W_p(\delta_{Y_N^*},\gamma)^p$ for every $\gamma \in \Gamma(\rho_1,\dots,\rho_D)$ and to $\beta$-Hölder continuity of $\Rr_\alpha(c_\#\cdot)$, we have
\[
\F_N(\delta_{Y_N^*})
\geq
(\min \F
- C S_N^{\beta/p}) + \lambda_N S_N.
\]
By optimality of $\delta_{Y_N^*}$, Equation~\eqref{eq:general_asymptotic_rate} writes
\[
\F_N(\delta_{Y_N^*})
\leq
\min \F + C'u_N^\beta,
\]
and together with the previous inequality we obtain
\begin{equation}
\label{ineq:intermediate_step}
- C S_N^{\beta/p} + \lambda_N S_N
\leq
C' u_N^\beta.
\end{equation}
We apply Young's inequality to write
$C S_N^{\beta/p} \leq \frac{1}{2} \lambda_N S_N + C''\lambda_N^{-\beta/(p-\beta)}$, and injecting this inequality in~\eqref{ineq:intermediate_step} yields, after simplification,
\begin{equation*}
\lambda_N S_N
\lesssim u_N^\beta,
\end{equation*}
which implies $S_N \lesssim u_N^p$ thanks to $\lambda_N = u_N^{-(p-\beta)}$.
\end{proof}

\subsubsection*{Uniform quantization}
We now briefly introduce the uniform quantization of measures on an arbitrary metric space $\Y$, and state a result of~\cite{merigot2016minimal} linking the asymptotic uniform quantization error of a measure and the dimension of its support. We consider quantization with respect to the $p$-Wasserstein distance, with $p \in [1,\infty)$ an arbitrary exponent.
Let $\nu \in \P_p(\Y)$ be a probability measure with finite $p$-moment. Its $N$-point \textit{uniform quantization error} of order $p$ is defined by
\begin{equation}
e_{p,N}(\nu) = \inf_{z_1,\dots,z_N \in \Y} W_p(\nu,N^{-1}\textstyle\sum_{i=1}^N \delta_{z_i}),
\end{equation}
and a discrete measure $N^{-1}\sum_{i=1}^N \delta_{z_i}$ for which the $z_i$ form a minimizer is called an \textit{optimal ($N$-point) uniform quantizer} of $\nu$. Up to considering the completion of $\mathcal{Y}$ rather than $\mathcal{Y}$ directly, the infimum is always achieved and $\nu$ has optimal uniform quantizers of all cardinals.
Before discussing estimates on the asymptotic quantization error rate, let us recall the definition of box dimension for compact metric spaces.

\begin{definition}
\label{def:box_dimension}
Let $K$ be a compact metric space. For any positive integer $N$, define the \textnormal{optimal $N$-point covering radius} by
\begin{equation}
r_N(K) = \min\left\{ r \geq 0 : \exists x_1,\dots,x_N \in K, \, K \inc \bigcup_{i=1}^N \bar{B}(x_i,r) \right\},
\end{equation}
where the infimum is indeed a minimum thanks to compactness of $K$ and to the covers consisting of closed balls. We call \textnormal{(upper) box dimension} --- or \textnormal{Minkowski dimension} --- of $K$ the nonnegative scalar
\begin{equation}
\mathrm{d_{box}}(K) =
\limsup_{N\to\infty} \frac{\log N}{-\log r_N(K)}.
\end{equation}
\end{definition}

\begin{remark}
Box dimension is always an upper bound for Hausdorff dimension, and is much simpler to compute, as it does not involve measures, in particular.
Both dimensions coincide for compact manifolds, but differ in general. For instance, while all countable sets have Hausdorff dimension zero, we have
\[
\mathrm{d_{box}}\left(([0,1]\cap\Q)^d\right) =d,
\qquad \qquad
\mathrm{d_{box}}\left(\left\{
\frac{1}{n}:n\in\N^*\right\}\right) =\frac{1}{2}.
\]
\end{remark}

In~\cite[Proposition 12]{merigot2016minimal},
Mérigot and Mirebeau make a link between the asymptotic uniform quantization error of a measure and the (upper) box dimension of its support. They state their result in a generic metric space, and although only quantization with respect to the quadratic Wasserstein distance is needed in their work, the proof extends \textit{mutatis mutandis} to the case of a general Wasserstein exponent $p \in [1,\infty)$.

\begin{proposition}[{\cite[Proposition~12]{merigot2016minimal}}]
\label{prop:uniform_quantization_Merigot-Mirebeau}
Consider an exponent $p \in [1,\infty)$ and a compactly supported probability measure $\nu \in \P(\Y)$ on a metric space $\Y$.
If the support of $\nu$ has finite box dimension $d = \mathrm{d_{box}}(E)$, then
\begin{equation}
e_{p,N}(\nu) \lesssim \,\tau_{p,d}(N),
\end{equation}
where
\begin{equation}
\label{eq:tau_of_N}
\tau_{p,d}(N) =
\begin{cases}
N^{-\frac{1}{\max(p,d)}}
\quad &\mathrm{if} \quad d \neq p,
\\
(\log N)^{\frac{1}{d}} N^{-\frac{1}{d}}
\quad &\mathrm{if} \quad d = p,
\end{cases}
\end{equation}
and where $\lesssim$ hides a constant which only depends on $p$, $d$, and on the set $\spt\,\nu$.
\end{proposition}

Thanks to the above proposition, we deduce the following convergence result as an immediate corollary of Theorem~\ref{thm:general_convergence_result}.

\begin{corollary}
\label{cor:convergence_with_dimension}
Suppose that Assumption~\ref{a:p_Lq_and_beta-Hölder} holds and that the $\rho_j$ are all compactly supported. Let $d \in [1,D]$ be the box dimension of the support of some minimizer of $\F$. Then, letting $\lambda_N = \tau_{p,d}(N)^{-(p-\beta)}$ where $\tau_{p,d}$ is defined in~\eqref{eq:tau_of_N}, we have
\begin{equation}
\label{eq:asymptotic_rate_dimension}
|\min\F_N - \min\F| = \O
\begin{cases}
N^{-\frac{\beta}{d}}
\quad &\mathrm{if} \quad d > p,
\\
(\log N)^{\frac{\beta}{d}} N^{-\frac{\beta}{d}}
\quad &\mathrm{if} \quad d = p,
\\
N^{-\frac{\beta}{p}}
\quad &\mathrm{if} \quad d < p,
\end{cases}
\end{equation}
where $\O$ hides a constant which depends on the support of the minimizer in question.
Moreover, for any sequence of minimizers $\delta_{Y_N^*} \in \argmin \F_N$, $N \in \N$, every weak limit point is a minimizer of $\F$.
\end{corollary}

\begin{remark}[Univariate marginals]
\label{rem:convergence_rate_box_dimension_1}
In the case of univariate marginals with compact supports and under a few assumptions on the cost function $c$, a minimizer of $\F$ will have support of dimension $d=1$, leading to an asymptotic convergence rate $\beta/p$.
In the next section, we will see that when the cost function is supermodular, one can in fact obtain a much faster convergence rate, which directly depends on the asymptotic uniform quantization errors of the fixed (unidimensional) marginals $\rho_j$.
\end{remark}

\subsection{The case of a supermodular cost function}

In this section, we assume that the marginals are univariate measures, that each $\X_j \inc \R$ is an interval, and that the cost function $c : \X_1\times\dots\times\X_D \to \R$ is supermodular, see Definition~\ref{def:supermodularity}.
Submodularity has gained ground in the multimarginal optimal transport literature, with a seminal result by Carlier in~\cite{carlier2003class} for strictly supermodular costs, and an extension to the case of multidimensional variables by Pass in~\cite{pass2012structural}. 

\begin{definition}
\label{def:supermodularity}
A function $c : I_1\times\dots\times I_D \to \R$ defined on some product of real intervals $I_j \inc \R$ is said to be \textnormal{supermodular} if for every $x,\tilde{x} \in I_1\times\dots\times I_D$ we have
\begin{equation}
\label{ineq:supermodularity}
c(x\wedge \tilde{x}) + c(x\vee \tilde{x}) \geq c(x) + c(\tilde{x}),
\end{equation}
where $x\wedge\tilde{x}$ and $x\vee\tilde{x}$ respectively denote the coordinate-wise minimum and maximum of $x$ and $\tilde{x}$. We say $c$ is \textnormal{strictly supermodular} when the equality is strict for any pair of distinct points $x \neq \tilde{x}$.
\end{definition}

For $\Cc^2$ functions, supermodularity is equivalent to $\frac{\partial^2 c}{\partial x_j\partial x_k} \geq 0$ on the whole domain for every $j \neq k$, and the strict version of these inequalities implies strict supermodularity, although it is not a necessary condition.

\subsubsection*{Supermodularity and comonotonicity}
It is a standard result that for a strictly supermodular cost function $c$, the support of any $\gamma$ maximizing $\E[c_\#\gamma]$ over $\Gamma(\rho_1,\dots,\rho_D)$ is \textit{totally} ordered for the coordinate-wise order on $\R^D$.
This may be seen as an immediate consequence of the well-known $c$-cyclical monotonicity of the support of any optimal transport plan.
Total ordered-ness of the support is a very strong property, as it completely determines the joint measure in terms of the prescribed marginals. The unique coupling with totally ordered support is called the \textit{comonotone} plan, and reads
$\gamma_{\mathrm{mon}} := (F_{\rho_1}^{-1},\dots,F_{\rho_D}^{-1})_\#\L_{(0,1)}$.
The reader may consult~\cite{deelstra2011overview} for a rather extensive bibliographic overview of comonotonicity and its applications in risk management.


As highlighted in~\cite[Lemma~21]{ennaji2024robust}, optimality of the comonotone plan for \textit{standard} multimarginal transport with a supermodular cost may be generalized to the spectral risk measures framework, up to the additional assumption of nondecreasing monotonicity of the cost in each variable.
This additional assumption ensures that the weighting $\alpha$ increases with the coordinate-wise order along the comonotone support of $\gamma$. We summarize the result in question in the following lemma, which still holds for $L^q$ spectral functions under our assumptions.

\begin{lemma}[{\cite[Lemma~21]{ennaji2024robust}}]
\label{lem:risk_problem_and_supermodularity}
In addition to Assumption~\ref{a:p_Lq_and_beta-Hölder}, suppose that $c : \X_1\times\dots\times\X_D \to \R$ is supermodular, and nondecreasing in each variable. Then the comonotone couplings of $\rho_1,\dots,\rho_D$ and of $\rho_0,\rho_1,\dots,\rho_D$ respectively solve Problems~\eqref{eq:risk_problem} and~\eqref{eq:equivalent_linear_problem}, and their common maximal value is
\begin{equation}
\int_0^1 c(F_{\rho_1}^{-1}(t),\dots,F_{\rho_D}^{-1}(t)) \, \alpha(t) \, dt.
\end{equation}
Moreover, if $c$ is \textnormal{strictly} supermodular, and \textnormal{strictly} increasing in each variable, then the solution of~\eqref{eq:equivalent_linear_problem} is unique on the support of $\alpha$, in the sense that any other solution coincide with the comonotone plan on the set $(0,+\infty)\times\X_1\times\dots\times\X_D$. This naturally translates to uniqueness of the solution of~\eqref{eq:risk_problem} when restricted to its fraction of mass $m = 1 - \inf\spt\,\alpha$ that is maximal for the coordinate-wise order on $\R^D$.
\end{lemma}

\begin{remark}[Relaxing the supermodularity assumption]
In the case of a generic supermodular cost function $s$ for a \textit{linear} multimarginal optimal transport problem, the comonotone plan is a solution that does \textnormal{not} depend on $s$. In the spectral risk measure framework of Lemma~\ref{lem:risk_problem_and_supermodularity}, not only is the comonotone plan independent of the (supermodular) cost function $c$, but it is also independent of the spectral function $\alpha$, and a fortiori of the spectral risk measure considered. Supermodularity is indeed a very strong and restrictive assumption.
Yet, we emphasize that for a large class of non-supermodular cost functions, for which the solution \textnormal{does} depend on the given spectral risk measure, the said solution has a low-dimensional support. In such cases, Corollary~\ref{cor:convergence_with_dimension} gives a \textnormal{sharp} convergence rate.
We refer to the work of Pass~\cite{pass2012structural, pass2012local} for various criteria yielding bounds on the dimension of the support, as well as his joint work~\cite{kim2014general} with Kim for a general condition under which any solution is concentrated on a graph of one of the marginal variables.
\end{remark}

\begin{remark}[The compatibility condition]
\label{rem:compatibility}
In~\cite{ennaji2024robust}, the authors mention a condition known as \textnormal{compatibility}, first introduced by Pass in~\cite[Sect~6]{pass2012structural} as an invariant form of the supermodularity condition.
Indeed, a cost function of $D$ variables is (strictly) compatible if and only if it is (strictly) supermodular up to a change of sign of some subset of the $D$ variables.
It follows that any result that holds under the (strict) supermodularity assumption directly extends to the case of (strict) compatibility, with the adequate sign changes.
In particular, Lemma~\ref{lem:risk_problem_and_supermodularity} still holds for compatible cost functions, up to replacing the increasing monotonicity assumption by decreasing monotonicity and taking the reverse function $t \in (0,1) \mto F_{\rho_j}^{-1}(1-t)$ instead of the standard quantile function, for the relevant set of variables.
We shall refer to the optimal plans in question for Problems~\eqref{eq:risk_problem} and~\eqref{eq:equivalent_linear_problem} as the \textnormal{$c$-monotone plan} and the \textnormal{$s$-monotone plan}, respectively.
\end{remark}


The proof of our convergence result in the case of a supermodular cost function relies on the fact that, broadly speaking, the comonotone coupling can be quantized as well as its marginals, as stated in the next proposition.

\begin{proposition}[Quantizing the comonotone coupling]
\label{prop:quantization_of_comonotone_plan}
Suppose that the marginals are one-dimensional, i.e.~$\X_j \inc \R$ for all $j$, and let $\gamma_{\mathrm{mon}} = (F_{\rho_1}^{-1},\dots,F_{\rho_D}^{-1})_\#\L_{(0,1)}$ be the comonotone coupling of the $\rho_j$.
Then, for every $N \in \N_1$, we have
\begin{equation}
e_{p,N}(\gamma_{\mathrm{mon}}) \lesssim
\left(\sum_{j=1}^D
e_{p,N}(\rho_j)^p\right)^{\frac{1}{p}},
\end{equation}
and the comonotone coupling of $W_p$-optimal uniform quantizers of the $\rho_j$ yields a $W_p$-quantization error of $\gamma_{\mathrm{mon}}$ no greater than the right-hand side.
\end{proposition}
\begin{proof}
For each $j$, let $\delta_{Y_j} =N^{-1}\sum_{i=1}^N\delta_{y_{i,j}}$ be an optimal uniform quantizer of $\rho_j$ in the sense of $W_p$, and assume without loss of generality that the positions of the Dirac masses are ordered, i.e.~$y_{1,j} \leq \dots \leq y_{N,j}$.
Since we want to approach the comonotone coupling of the $\rho_j$, we naturally pair these one-dimensional point clouds in a monotone way. That is, we define
$\delta_Y = N^{-1}\sum_{i=1}^N\delta_{y_i}$ where $y_i = (y_{i,j})_{j=1}^N \in \R^D$.
By monotonicity, the optimal transport from $\gamma_{\mathrm{mon}}$ to $\delta_Y$ in the sense of $W_p$ consists in sending the $i$-th lowest part of the comonotone plan to the $i$-th Dirac mass of lowest coordinates. More precisely, we send the submeasure $\gamma_i = (F_{\rho_1}^{-1},\dots,F_{\rho_D}^{-1})_\#\L_{(\frac{i-1}{N},\frac{i}{N})}$ of mass $1/N$ to the point $y^N_i$. Note that $\gamma_i$ is a joint measure --- in fact, the unique comonotone one --- of the $i$-th lowest submeasures in the respective uniform decompositions of each marginal $\rho_j$ into $N$ ordered parts.
By equivalence of norms in $\R^D$, we thus find
\begin{align*}
W_p(\gamma_{\mathrm{mon}},\delta_Y)^p
&= \sum_{i=1}^N \int_{\R^D} \|x-y_i\|^p d\gamma_i(x)
\\
&\lesssim \sum_{i=1}^N \int_{\R^D} \left(\sum_{j=1}^D |x_j-y_{i,j}|^p\right) d\gamma_i(x)
\\
&=
\sum_{j=1}^D
\left(\sum_{i=1}^N \int_{\R^D} |x_j-y_{i,j}|^p d[\pi^{j}_\#\gamma_i](x_j)\right),
\end{align*}
and the sum in parentheses is the optimal cost $e_{p,N}(\rho_j)^p = W_p(\rho_j,\delta_{Y_j})^p$, since the monotone transport of $\gamma_{\mathrm{mon}}$ to $\delta_Y$ projects to comonotone transports from the $\rho_j$ to their respective optimal quantizers $\delta_{Y_j}$.
\end{proof}

From Proposition~\ref{prop:quantization_of_comonotone_plan} we deduce the following convergence result, as an immediate corollary of Theorem~\ref{thm:general_convergence_result}. Note that we do not need any compactness assumption on the $\rho_j$.

\begin{corollary}
\label{cor:convergence_with_comonotonicity}
Suppose that Assumption~\ref{a:p_Lq_and_beta-Hölder} holds and that the cost function $c$ is supermodular, and monotone increasing in each variable. Then, letting $\lambda_N = h_N^{-(p-\beta)}$ where
\begin{equation}
h_N = \max_{j \in \{1,\dots,D\}}
e_{p,N}(\rho_j),
\end{equation}
we have
\begin{equation}
|\min\F_N - \min\F| =
\O(h_N^\beta).
\end{equation}
Moreover, for any sequence of minimizers $\delta_{Y_N^*} \in \argmin \F_N$, $N \in \N$, every weak limit point is a minimizer of $\F$.
\end{corollary}

\subsubsection*{Uniform quantization in dimension one}
Bencheikh and Jourdain have recently obtained a necessary and sufficient conditions for the asymptotic uniform quantization error of a univariate probability density $\nu \in \P_p(\R)$ to go to zero with order $\theta \in (0,\frac{1}{p})$. Their main result~\cite[Theorem~2.2]{bencheikh2022approximation} reads
\begin{equation}
\sup_{N\in\N_1} N^\theta e_{p,N}(\nu) < \infty
\qquad
\Longleftrightarrow
\qquad
\sup_{x\geq 0} \ x^{\frac{p}{1-\theta p}}(F_\nu(-x) + 1 - F_\nu(x)),
\end{equation}
where the supremum on the right-hand side is essentially an estimate on the left and right tail distribution of $\nu$. The authors also recall the fact that for measures with unbounded support, the error cannot go to zero with order strictly greater than $\frac{1}{p}$.

On the other hand, when $\nu$ is compactly supported, the order of convergence for the uniform quantization error is at least $\frac{1}{p}$, and the best possible order is $\theta=1$.
A well-known sufficient condition for order $\theta = 1$ is that $\nu$ is absolutely continuous with support a compact interval, and density isolated from zero on this interval.
Broadly speaking, when $\nu$ has connected compact support, it is the local behavior of the density near its zeros that determines the convergence rate of the quantization error.
In a first version of~\cite{bencheikh2022approximation}, Bencheikh and Jourdain prove that for any $a > 1$, the probability density $\rho(x) = a x^{a-1}$ with support $[0,1]$ has a uniform quantization error going to zero at order exactly $\min\{\frac{1}{p} + \frac{1}{a},1\}$, with a logarithmic multiplicative correction factor $(\log N)^{\frac{1}{p}}$ for the limit case $\frac{1}{p}+\frac{1}{a} = 1$.

\subsubsection*{Comparison criterion}
As a final remark concerning uniform quantization in dimension one, we give a simple criterion for upper and lower bounds on the asymptotic error rate of a density with connected compact support and a \textit{finite} number of zeros in the support.
If the density is locally bounded below by some common power $|x-x_0|^{b-1}$ around each of the zeros $x_0$, its asymptotic error rate is at least that of the power-law in question.
Conversely, if the density is locally bounded above by some power-law $|x-x_0|^{a-1}$ in the neighborhood of \textit{some} arbitrary zero, the asymptotic order of convergence cannot exceed that of the corresponding power-law.
In fact, for the latter property, it suffices to consider \textit{lateral} neighborhoods, i.e. $(x_0-\eps,x_0]$ or $[x_0,x_0+\eps)$.

\subsection{The partial transport case}

In many cases, one is in fact interested in the \textit{riskiest part} of the riskiest dependence structure, rather than in the \textit{whole} dependence structure. Mathematically, this translates into a spectral function with non-full support $(1-m,1)$, so that only the most dangerous fraction of prescribed mass $m \in (0,1)$ is taken into account in the spectral risk measure. This section is devoted to the specific case where the spectral risk measure in question is the conditional value at risk $\mathrm{CVaR}_m$, corresponding to $\alpha \propto \mathbbm{1}_{(1-m,1)}$.

Since $\alpha$ vanishes on the interval $(0,1-m)$ and is constant on $(1-m,1)$, Problem~\eqref{eq:risk_problem} has the following \textit{partial} multimarginal optimal transport formulation
\begin{equation}
\label{eq:partial_risk_problem}
\tag{$\mathcal{P}^m$}
\max_{\gamma \in \Gamma_{m}(\rho_1,\dots,\rho_D)}
\frac{1}{m} \int_{\boldsymbol{\X}} c \, d\gamma,
\end{equation}
where $\Gamma_m(\rho_1,\dots,\rho_D)$ denotes the set of (positive) measures of mass $m$ on $\boldsymbol{\X}$ whose marginals are respectively dominated by the $\rho_j$. That is, a measure $\gamma \in \mathcal{M}^+(\boldsymbol{\X})$ is in $\Gamma_m(\rho_1,\dots,\rho_D)$ if and only if we have $\gamma(\boldsymbol{\X}) =m$ and for every $j\in\{1,\dots,D\}$,
\begin{equation}
\label{ineq:constraints_partial}
[\pi^j_\#\gamma](B) \leq \rho_j(B), \quad \forall \ \text{Borel set} \ B \inc \X_j.
\end{equation}
We refer to the work of Kitagawa and Pass in~\cite{kitagawa2015multi} for more details on partial multimarginal transport.

As a consequence, we look for a discrete measure of the form $\delta^m_Y := m\delta_Y$ with $Y \in \boldsymbol{\X}^N$, and the corresponding discretized problem reads
\begin{equation}
\label{eq:discretized_partial_risk_problem}
\tag{$\mathcal{P}^m_N$}
\min_{Y \in \boldsymbol{\X}^N}
-\frac{1}{N}\sum_{i=1}^N c(y_i)
+ \lambda_N \sum_{j=1 }^D W_{p,\max}(\rho_j,\pi^j_\#\delta_Y)^p,
\end{equation}
where $W_{p,\max}$ denotes the \textit{partial $p$-Wasserstein cost}. For two (positive) measures $\rho,\nu \in \mathcal{M}^+_p(\R^n)$ with finite $p$-moment, the latter is defined by
\begin{equation}
\label{eq:partial_Wasserstein_cost}
W_{p,\max}(\rho,\mu)^p :=
\min_{\zeta\in\Gamma_{\max}(\rho,\nu)}
\int_{\R^n\times\R^n} |x-y_j|^p d\zeta(x,y),
\end{equation}
where $\Gamma_{\max}(\rho,\nu)$ is the set of measures on $\R^n\times\R^n$ of mass $\min\{\rho(\R^n),\nu(\R^n)\}$ whose marginals are respectively dominated by $\rho$ and $\nu$. The subscript \textit{max} refers to the fact that the maximum of quantity is transported between the two measures. Unlike the usual Wasserstein distance, the partial cost $W_{p,\max}$ is not a distance in the mathematical sense, as it lacks the separation property. Indeed, it is straightforward to show that $W_{p,\max}(\rho,\nu)$ is zero if and only if one of $\rho$ or $\nu$ is dominated by the other one.

Since the discrete measure in~\eqref{eq:discretized_partial_risk_problem} is of mass $m < 1$, the penalty terms penalize how far its marginals are from being respectively dominated by the $\rho_j$, which is precisely what we desire to take into account the inequality constraints~\eqref{ineq:constraints_partial} via the penalty method.
For the convergence result of this section, we will use the following assumption.

\begin{assumption}
\label{a:p_and_beta-Hölder}
We have $p \in [1,\infty)$, and there exists $\beta \in (0,1]$ such that $c$ is $\beta$-Hölder, with Hölder constant $C_{\mathrm{H}}$.
\end{assumption}

The following lemma is a direct consequence of the above assumption. Of course, the estimate in question is deducible from Lemma~\ref{lem:control_on_the_spectral_risk_measure} by noticing that $\alpha_m \propto \mathbbm{1}_{(1-m,1)}$ satisfies Assumption~\ref{a:p_Lq_and_beta-Hölder} with $q = \infty$, but we prefer to provide a separate proof for the sake of completeness.

\begin{lemma}[Hölder continuity of the $\mathrm{CVaR}$ objective function]
\label{lem:control_on_the_CVaR}
Under Assumption~\ref{a:p_and_beta-Hölder}, for any pair of measures $\gamma,\tilde{\gamma}$ on $\boldsymbol{\X}$ of mass $m$, we have
\begin{equation}
\label{ineq:control_on_the_CVaR}
\left|
\frac{1}{m}\int_{\boldsymbol{\X}} c\,d\gamma
- \frac{1}{m}\int_{\boldsymbol{\X}} c\,d\tilde{\gamma}
\right|
\leq
C_{\mathrm{H}}
W_p(\gamma,\tilde{\gamma})^\beta.
\end{equation}
\end{lemma}
\begin{proof}
For any coupling $\Pi \in \Gamma(m^{-1}\gamma,m^{-1}\tilde{\gamma}) \inc \P(\boldsymbol{\X}\times\boldsymbol{\X})$, the left-hand side of~\eqref{ineq:control_on_the_CVaR} is bounded above by
\begin{align*}
\left|
\int_{\boldsymbol{\X}} (c(x) -c(\tilde{x})) \,d\Pi(x,\tilde{x})
\right|
&\leq
\int_{\boldsymbol{\X}\times\boldsymbol{\X}} |c(x)-c(\tilde{x})|d\Pi(x,\tilde{x})
\\
&\leq
C_{\mathrm{H}}
\int_{\boldsymbol{\X}\times\boldsymbol{\X}} |x-\tilde{x}|^\beta d\Pi(x,\tilde{x}),
\end{align*}
where we used $\beta$-Hölder continuity of $c$. Taking the infimum over $\Pi$ and using the standard inequality $W_\beta \leq W_p$ derived from $\beta \leq p$ yields the desired estimate.
\end{proof}

\begin{lemma}
For any two measures $\gamma,\tilde{\gamma} \in \mathcal{M}^+_p(\boldsymbol{\X})$ with finite $p$-moment and of same finite mass, we have
\begin{equation}
W_p(\gamma,\tilde{\gamma})^p
\gtrsim
\sum_{j=1}^D W_{p,\max}(\pi^j_\#\gamma,\pi^j_\#\tilde{\gamma})^p.
\end{equation}
\end{lemma}
\begin{proof}
Take an arbitrary joint measure $\Pi \in \Gamma(\gamma,\tilde{\gamma})$, of mass $m$ by definition. Thanks to equivalence of norms in $\R^D$, we find
\begin{align*}
\int_{\boldsymbol{\X}\times\boldsymbol{\X}} \|x-\tilde{x}\|^p \,d\Pi(x,\tilde{x})
&\gtrsim
\sum_{j=1}^D \int |x_j -\tilde{x}_j|^p \, d[\pi^{j,j+D}_\#\Pi](x_j,\tilde{x}_j).
\end{align*}
Since $\pi^{j,j+D}_\#\Pi$ is a joint measure of the respective $j$-th marginals of $\gamma$ and of $\tilde{\gamma}$, the $j$-th term in the sum is bounded below by $W_p(\pi^j_\#\gamma,\pi^j_\#\tilde{\gamma})^p$, and the desired inequality is obtained by taking the infimum in $\Pi$.
\end{proof}

As in the previous section, we introduce the functionals $\F^m,\F^m_N : \mathcal{M}^+_p(\boldsymbol{\X}) \to \R\cup\{+\infty\}$ that are implicitly minimized in Problems~\eqref{eq:partial_risk_problem} and~\eqref{eq:discretized_partial_risk_problem}, respectively. These are defined by
\begin{equation*}
\begin{cases}
\F^m(\gamma)
&= -\frac{1}{m}\int_{\boldsymbol{\X}}c\,d\gamma
+ \chi_{\Gamma_m(\rho_1,\dots,\rho_D)}(\gamma),
\\[4pt]
\F^m_N(\gamma)
&= -\frac{1}{m}\int_{\boldsymbol{\X}}c\,d\gamma
+ \lambda_N \sum_{j=1}^D W_{p,\max}(\rho_j,\pi^j_\# \gamma)^p
+ \chi_{\D^m_N}(\gamma),
\end{cases}
\end{equation*}
where $\D^m_N = \{m\delta_Y : Y \in \boldsymbol{\X}^N\}$ is the set of uniform clouds of $N$ points in $\boldsymbol{\X}$ that are of mass $m$.

By the same arguments as in the beginning of Section~\ref{susbec:convergence_standard}, with the additional inequality $W_{p,\max}(\rho_j,m\delta_{Y_j}) \leq W_p(\pi^j_\#\gamma^*,m\delta_{Y_j})$, we find
\[\boxed{|\min\F_N^m - \min\F^m|=\O(u_N^\beta),}\]
where $u_N$ is the optimal $N$-point uniform quantization error of some solution $\gamma^*$.
The construction of $\gamma_N$ is detailed in the next lemma, and is essentially the same as in Section~\ref{susbec:convergence_standard}, except that we consider the \textit{partial} optimal transports from the $\rho_j$ to the respective projections of the point cloud.

\begin{lemma}[Block approximation]
\label{lem:N-block_partial_joint_measure}
Let $m\delta_Y \in \D_N$ be a uniform cloud of mass $m$ consisting of $N$ points in $\boldsymbol{\X}$, and denote $m\delta_{Y_j}$ its projection on $\X_j$. There exists a measure $\gamma \in \Gamma_m(\rho_1,\dots,\rho_D)$ such that
\begin{equation}
W_p(\gamma,m\delta_Y)
\lesssim
\left( \sum_{j=1}^D W_{p,\max}(\rho_j,m\delta_{Y_j})^p \right)^{\frac{1}{p}}.
\end{equation}
\end{lemma}
\begin{proof}
Consider the active parts $\rho_{j}^{\mathrm{act}}$ of the probability measures $\rho_j$ in their respective $W_p$-optimal transport to the projections $m\delta_{Y_j}$. That is, the partial optimal transport plan for $W_{p,\max}(\rho_j,m\delta_{Y_j})$ has first marginal $\rho_j^{\mathrm{act}}$.
Applying Lemma~\ref{lem:N-block_joint_measure} to $\delta_Y$ and the normalized measures $m^{-1}\rho_j^{\mathrm{act}}$, we obtain some probability measure $\gamma^{\mathrm{prob}} \in \Gamma(m^{-1}\rho_1^{\mathrm{act}},\dots,m^{-1}\rho_D^{\mathrm{act}})$
satisfying
\begin{equation*}
W_p(\gamma^{\mathrm{prob}},\delta_Y)
\lesssim
\left( \sum_{j=1}^D W_p(m^{-1}\rho_j^{\mathrm{act}},\delta_{Y_j})^p \right)^{\frac{1}{p}}.
\end{equation*}
Letting $\gamma := m\gamma^{\mathrm{prob}} \in \Gamma_m(\rho_1^{\mathrm{act}},\dots,\rho_D^{\mathrm{act}}) \inc \Gamma_m(\rho_1,\dots,\rho_D)$ then yields the desired estimate, thanks to the scaling
$W_p(m\mu,m\nu) = m^{\frac{1}{p}}W_p(\mu,\nu)$ and to the equality $W_p(\rho_j^{\mathrm{act}},m\delta_{Y_j}) = W_{p,\max}(\rho_j,m\delta_{Y_j})$, by optimality of the $\rho_j$.
\end{proof}

Lemma~\ref{lem:N-block_joint_measure} leads to the following global convergence results for the partial version of the problem.

\begin{theorem}
\label{thm:general_convergence_result_partial}
Suppose that Assumption~\ref{a:p_and_beta-Hölder} holds, and let $u_N$ be some upper bound on the $W_p$-optimal uniform quantization error for some fixed minimizer $\gamma^*$ of $\F$. Then, letting $\lambda_N = u_N^{-(p-\beta)}$, we have
\begin{equation}
\label{eq:general_asymptotic_rate_partial}
|\min \F_N^m - \min \F^m| = \O(u_N^\beta).
\end{equation}
Moreover, for any sequence of minimizers $\delta_{Y_N^*} \in \argmin \F_N^m$, $N \in \N$, every weak limit point is a minimizer of $\F^m$, and for every $j \in \{1,\dots,D\}$ we have
\begin{equation}
\label{ineq:estimate_for_error_on_marginals_partial}
W_{p,\max}(\rho_j, \pi^j_\#\delta_{Y_N^*})
= \O(u_N).
\end{equation}
\end{theorem}

Once again thanks to the direct generalization of Mérigot and Mirebeau's result stated in Proposition~\ref{prop:uniform_quantization_Merigot-Mirebeau}, we deduce the following immediate corollary.

\begin{corollary}
\label{cor:convergence_with_dimension_partial}
Suppose that Assumption~\ref{a:p_and_beta-Hölder} holds and that the $\rho_j$ are all compactly supported. Let $d \in [1,D]$ be the box dimension of the support of some minimizer of $\F$. Then, letting $\lambda_N = \tau_{p,d}(N)^{-(p-\beta)}$ where $\tau_{p,d}$ is defined in~\eqref{eq:tau_of_N}, we have
\begin{equation}
\label{eq:asymptotic_rate_dimension_partial}
|\min\F_N^m - \min\F^m| = \O
\begin{cases}
N^{-\frac{\beta}{d}}
\quad &\mathrm{if} \quad d > p,
\\
(\log N)^{\frac{\beta}{p}} N^{-\frac{\beta}{p}}
\quad &\mathrm{if} \quad d = p,
\\
N^{-\frac{\beta}{p}}
\quad &\mathrm{if} \quad d < p,
\end{cases}
\end{equation}
where $\O$ hides a constant which depends on the support of the minimizer in question.
Moreover, for any sequence of minimizers $m\delta_{Y_N^*} \in \argmin \F_N^m$, $N \in \N$, every weak limit point is a minimizer of $\F^m$.
\end{corollary}

For a result analogous to Theorem~\ref{cor:convergence_with_comonotonicity}, one needs to control the uniform quantization error of submeasures of the form $(F_{\rho}^{-1})_\#\L_{(1-m,1)}$ in terms of that of $\rho$, for sufficiently well-behaved univariate probability measures $\rho$.

\begin{lemma}[Quantization of the restriction to an interval]
\label{lem:submeasure_quantization}
Let $\rho\in\P_p(\R)$ be absolute continuous with connected support. Suppose that, as $N \to \infty$,
\begin{equation*}
e_{p,N}(\rho)
\lesssim
\frac{(1 + \log N)^\zeta}{N^{\eta}}
\end{equation*}
for some exponents $\eta > 0$ and $\zeta \in \R$, where $\lesssim$ hides a constant which may depend on $\rho$ and on the exponents.
Then, for any $m\in(0,1)$, the measure $\rho_m := (F_\rho^{-1})_{\#}\L_{(0,m)}$ of mass $m$ satisfies
\begin{equation}
\label{ineq:quantization_of_restriction}
e_{p,N}(\rho_m)
\lesssim
\frac{(1 + \log N)^\zeta}{N^{\eta}}
\end{equation}
as $N \to \infty$, and the constant behind $\lesssim$ is independent of $m$. The same estimate holds for the measure $\tilde{\rho}_m := (F_\rho^{-1})_\#\L_{(1-m,1)}$.
\end{lemma}

\begin{remark}
The assumption that $m$ is rational makes the proof rather straightforward, but is likely unnecessary, as suggested by the multiplicative constant behind $\lesssim$ in~\eqref{ineq:quantization_of_restriction} being independent of $m$. However, one cannot directly extend our result to any arbitrary $m$. Indeed, and as highlighted in the proof, for a rational $m = \frac{a}{b}$ with $a,b$ coprime positive integers, the estimate~\eqref{ineq:quantization_of_restriction} only holds for $N \geq a$.
\end{remark}

\begin{proof}
Denote $u(N) = \frac{(1+\log N)^\zeta}{N^\eta}$, and write $m=\frac{a}{b}$ with $a < b$ coprime positive integers.
Let $(bK)^{-1}\sum_{i=1}^{bK} \delta_{x_i}$ be an optimal uniform quantizer of $\rho$ for $W_p$, where we suppose without loss of generality that $x_1 < x_2 < \dots < x_{bK}$. Since $m(bK) = aK$ is an integer, the $W_p$-optimal transport from $\rho$ to its uniform quantizer exactly send the measure $\rho_m$ to the leftmost $aK$ Dirac masses. As a result,
\begin{align*}
e_{p,aK}(\rho_m)
&\leq
\textstyle
W_p(\rho_m, (bK)^{-1}\sum_{i=1}^{aK}\delta_{x_i}) \\
&\leq
\textstyle
W_p(\rho, (bK)^{-1}\sum_{i=1}^{bK}\delta_{x_i}) \\
&= e_{p,bK}(\rho) \\
&\leq e_{p,aK}(\rho),
\end{align*}
where the last inequality comes from the straightforward observation that for an absolutely continuous measure on $\R$ with \textit{connected} support, the uniform quantization error is decreasing in the number of Dirac masses.
We thus have $e_{p,N}(\rho_m) \leq e_{p,N}(\rho)$ for any $N$ that is a multiple of $a$.
Now consider an arbitrary $N$, and denote $N = aK + r$ its Euclidean division by $a$.
Thanks to $\rho_m$ having connected support as well,
\begin{align*}
e_{p,N}(\rho_m)
&\leq e_{p,aK}(\rho) = u(aK) \\
&= \frac{u(aK)}{u(aK+r)} u(\underbrace{aK+r}_{N}).
\end{align*}
Since $u(n) = \frac{(1+\log n)^\zeta}{n^\eta}$ and $r \in \{0,1,\dots,a-1\}$, elementary considerations show that the ratio $\frac{u(aK)}{u(aK+r)}$ is bounded by $2^\eta(1+\log 2)^{\max\{-\zeta,0\}}$ as soon as $K \geq 1$, so for any $N \geq a$, which concludes the proof.
\end{proof}

\begin{theorem}
\label{cor:convergence_with_comonotonicity_partial}
Suppose that Assumption~\ref{a:p_and_beta-Hölder} holds, that $m \in (0,1)$ is a rational, and that the cost function $c$ is supermodular, and monotone increasing in each variable. Suppose moreover that the marginal probability measures $\rho_j$ are absolutely continuous with connected supports, and that there exists exponents $\eta > 0$, $\zeta \in \R$ such that for every $j \in \{1,\dots,D\}$,
\begin{equation*}
e_{p,N}(\rho_j) \lesssim h_N := \frac{(1 + \log N)^\zeta}{N^\eta},
\end{equation*}
where $\lesssim$ hides a constant which may depend on $\rho_j$ and on the exponents $\eta,\zeta$.
Then, letting $\lambda_N = h_N^{-(p-\beta)}$, we have
\begin{equation}
|\min\F_N^m - \min\F^m| =
\O(h_N^\beta).
\end{equation}
Moreover, for any sequence of minimizers $\delta_{Y_N^*} \in \argmin \F_N^m$, $N \in \N$, every weak limit point is a minimizer of $\F^m$.
\end{theorem}

\section{Numerics}

We now illustrate our discretization method with various numerical simulations using 
corresponding to the different cases dealt with in the previous section.

Each simulation is made with some fixed number of points $N$, and we compute a numerical solution of the corresponding Problem~\eqref{eq:discretized_risk_problem} or~\eqref{eq:discretized_partial_risk_problem} using the Limited-memory BFGS algorithm provided by \texttt{SciPy}. The latter, a quasi-Newton method, only requires the value of the optimized functional and that of its gradient at the current point. We have already seen that computing the value and gradient of the spectral risk measure $\Rr_\alpha(c_\#\delta_Y)$ is immediate, since this quantity is a linear combination of the $c$-values at the points of the cloud, and since the cost function $c$ is explicit. The computation of the penalty terms and of their gradients is a bit more involved, as it involves solving a semi-discrete optimal transport problem, partial or not depending on the chosen framework, for each marginal. In this paper, we fix $p=2$ and restrict ourselves to unidimensional marginals for the simulations, which considerably eases the computation related to the penalty terms.
The initial point cloud $Y_{\mathrm{init}}$ is constructed by drawing $N$ independent points from the uniform measure on the product of the supports $\spt\,\mu_j$.

\begin{remark}[General rule of thumb for the penalty coefficient]
We observe that when the penalty coefficient is qualitatively large, the projections of the numerical solution on each of the axes match well with the respective prescribed marginals $\rho_j$, but the point cloud in question lacks structure, in the sense that the input variables are poorly correlated.
On the other hand, a penalty coefficient that is qualitatively too small leads to a much clearer structure, but at the cost of a poor approximation of the marginals.
To address this issue, we start by solving the discretized problem with a reasonably small penalty coefficient, in order to get the relevant structure for the solution, and gradually increase the coefficient until the projections on the different axes are satisfyingly close to the prescribed marginals.
Although deriving a universal rule for this sequence of penalty coefficients would likely be quite delicate, this broad strategy can serve as a general rule of thumb to compute satisfying numerical solutions.
\end{remark}

\subsection{Computing the penalty terms}
Let us first give the expression of the respective gradients of the unidimensional semi-discrete optimal transport functional $\mathcal{G} : Z \in \R^N \mapsto W_2(\rho,\delta_Z)^2$ and of its partial version $\mathcal{G}^m : Z \in \R^N \mapsto W_{2,\max}(\rho,m\delta_Z)^2$, where $\rho \in \P_{2,\mathrm{ac}}(\R)$ is some fixed univariate probability density with finite variance and where $m \in (0,1)$. Once again, we write $Z = (z_1,\dots,z_N)$. Since $\mathcal{G}$ and $\mathcal{G}^m$ are only differentiable outside the generalized diagonal, we assume that the $z_i$ are pairwise distinct, so that without loss of generality $z_1 < \dots < z_N$. For the sake of conciseness, we only derive the gradient of $\mathcal{G}^m$, as the gradient of $\mathcal{G}$ can be derived similarly.
Let us consider the dual formulation
\begin{equation}
\label{eq:unidimensional_W2_dual_formulation}
\mathcal{G}^m(Z) =
\max_{\psi \in (0,+\infty)^N}
\left\{
\int_{\R}
\min\{0,\min_i [|x-z_i|^2 - \psi_i]\}
d\rho(x)
+ \frac{m}{N} \sum_{i=1}^N \psi_i
\right\},
\end{equation}
which stems from the Kantorovich duality for partial optimal transport. We refer to~\cite[Theorem~1.42]{santambrogio2015optimal} for the duality theory of standard optimal transport, and to~\cite{cances2025unidimensional} for the case of partial transport.
Note that by dominated convergence, the optimality condition on $\psi$ is that $\rho(\RLag_i(Z;\psi)) = \frac{m}{N}$ for all $i$, where
$\RLag_i(Z;\psi) = \Lag_i(Z;\psi)\cap B(x_i,\sqrt{\psi_i^*})$ is the \textit{restricted Laguerre cell} corresponding to $\psi$. Here, $\Lag_i(Z;\psi), \ i\in\{1,\dots,N\}$ denote the cells of the (standard) \textit{Laguerre tessellation} induced by $\psi$:
\begin{equation}
\label{eq:Laguerre_cell}
\Lag_i(Z;\psi) =\{x \in \R : |x-y_i|^2 -\psi_i < |x-y_k|^2 - \psi_k, \ \forall k \neq i \}.
\end{equation}
Let us denote $\RLag_i^m(\rho,Z)$ the restricted Laguerre cells induced by the (unique) optimal $\psi$ for the partial optimal transport from $\rho$ to $m\delta_Z$.
The envelope theorem then yields
\begin{align}
\label{eq:gradient_of_partial_1D_OT_functional}
\frac{\partial\mathcal{G}^m}{\partial z_i}(Z)
&=
2\int_{R\Lag_i^m(\rho,Z)}(z_i-x)d\rho(x)
\nonumber
\\
&=
\frac{2m}{N}[z_i-
b_i^m(\rho,Z)],
\end{align}
where
$b_i^m(\rho,Z)
=
\fint_{\RLag_i^m(\rho,Z)} x \, d\rho(x)
$
is the $\rho$-barycenter of the optimal restricted cell corresponding to $z_i$.
The gradient of $\mathcal{G}$ is derived in the exact same way but via the standard Kantorovich duality, and we have
\begin{align}
\label{eq:gradient_of_1D_OT_functional}
\frac{\partial\mathcal{G}}{\partial z_i}(Z)
&=
\frac{2}{N}[z_i-
b_i(\rho,Z)],
\end{align}
with $b_i(\rho,Z) = \fint_{\Lag_i(\rho,Z)}x\,d\rho(x)$ the $\rho$-barycenter of the optimal (unrestricted) Laguerre cell corresponding to $z_i$.

In the balanced case, the optimal Laguerre tessellation is simply given by the uniform decomposition of $\rho$ into $N$ equal-mass parts with respective supports mutually ordered, namely $\Lag_i(\rho,Z) = (F_\mu^{-1}(\frac{i-1}{N}),F_\mu^{-1}(\frac{i}{N}))$. Note that the latter only depends on the \textit{order} of the Dirac positions.
In the partial transport case, this is usually not the case anymore, since, broadly speaking, the ``gaps'' corresponding to the mass that is not transported will roughly be located between the pairs of consecutive cells whose associated Dirac masses are far enough from each other, see Figure~\ref{fig:1D_OT_balanced_vs_partial}.

To solve~\eqref{eq:unidimensional_W2_dual_formulation}, we use a code developed by Hugo Leclerc, in which the restricted Laguerre cells (which are intervals) are directly parameterizing by their respective endpoints. The code relies on a rather efficient heuristic that successively merges well-chosen adjacent cells in order to reduce the problem's complexity and number of variables.

\begin{figure}[!ht]
\centering
\begin{subfigure}{0.45\textwidth}
\centering
\includegraphics[width=0.75\textwidth]{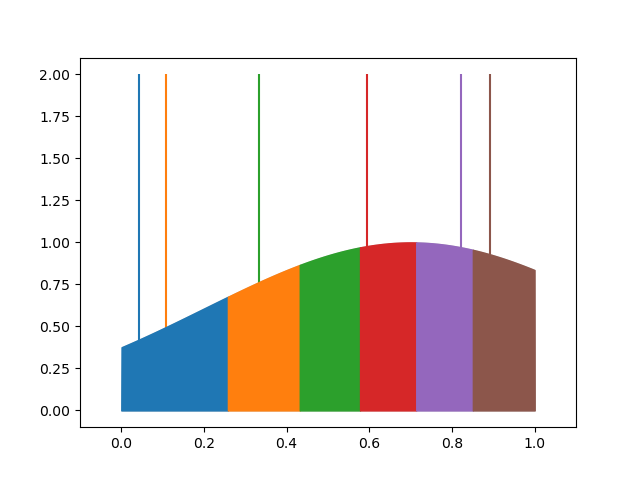}
\caption{Balanced transport}
\label{fig:1D_OT}
\end{subfigure}
\begin{subfigure}{0.45\textwidth}
\centering
\includegraphics[width=0.75\textwidth]{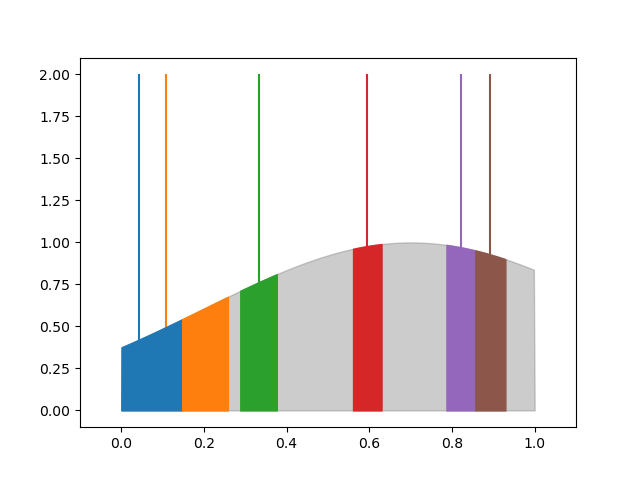}
\caption{Partial transport, with $m=\frac{1}{2}$}
\label{fig:1D_partial_OT}
\end{subfigure}
\caption{We represent the decomposition of the unidimensional probability density induced by its optimal transport to the sum of Dirac masses, whose respective positions are indicated by the vertical lines. The left-hand side corresponds to standard optimal transport, with the Dirac weights summing to one, while the right-hand corresponds to partial optimal transport, with the Dirac weights summing to $m=\frac{1}{2}$.}
\label{fig:1D_OT_balanced_vs_partial}
\end{figure}

\subsection{Squared sum of coordinates}

In~\cite[Remark~2.12]{pass2015multi}, Pass mentions a now well-known example of a cost function for which solutions of the standard multimarginal optimal transport problem with $D$ marginals may have support of dimension $D-1$. The surplus cost function in question is $c(x_1,\dots,x_D) = -|x_1+\dots+x_D|^2$, and indeed it is immediate to check that any probability measure concentrated on the plane $x_1+\dots+x_D=0$ is optimal for its marginals. In fact, fixing univariate probability measures $\mu_1,\dots,\mu_D \in \P(\R)$, Jensen's inequality yields that for any $\gamma \in \Gamma(\mu_1,\dots,\mu_D)$ we have
\[
\int_{\R^D} |x_1+\dots+x_D|^2 d\gamma
\geq
\left(\int_{\R^D} (x_1+\dots+x_D) d\gamma\right)^2
= \left(\sum_{j=1}^D \mathbb{E}\mu_j \right)^2.
\]
Hence any coupling of the $\mu_j$ that is concentrated on the hyperplane $x_1+\dots+x_D = \sum_{j} \mathbb{E}\mu_j$ must be optimal.
Our numerical simulations in Figure~\ref{fig:squared_sum} show such an optimal coupling with two-dimensional support for $\mu_1$ the uniform measure on $[0,2]$, $\mu_2$ the triangle density $\mathrm{Tri(0,1,2)}$ and $\mu_3$ the Wigner semicircle density with support $[-1,1]$. The black vector vector in the 3D graphs has base-point $(\mathbb{E}\mu_1,\mathbb{E}\mu_2,\mathbb{E}\mu_3) = (1,1,0)$ and is collinear with $(1,1,1)$.

\begin{figure}[!ht]
\centering
\begin{subfigure}{0.9\textwidth}
\centering
\includegraphics[width=0.6\textwidth]{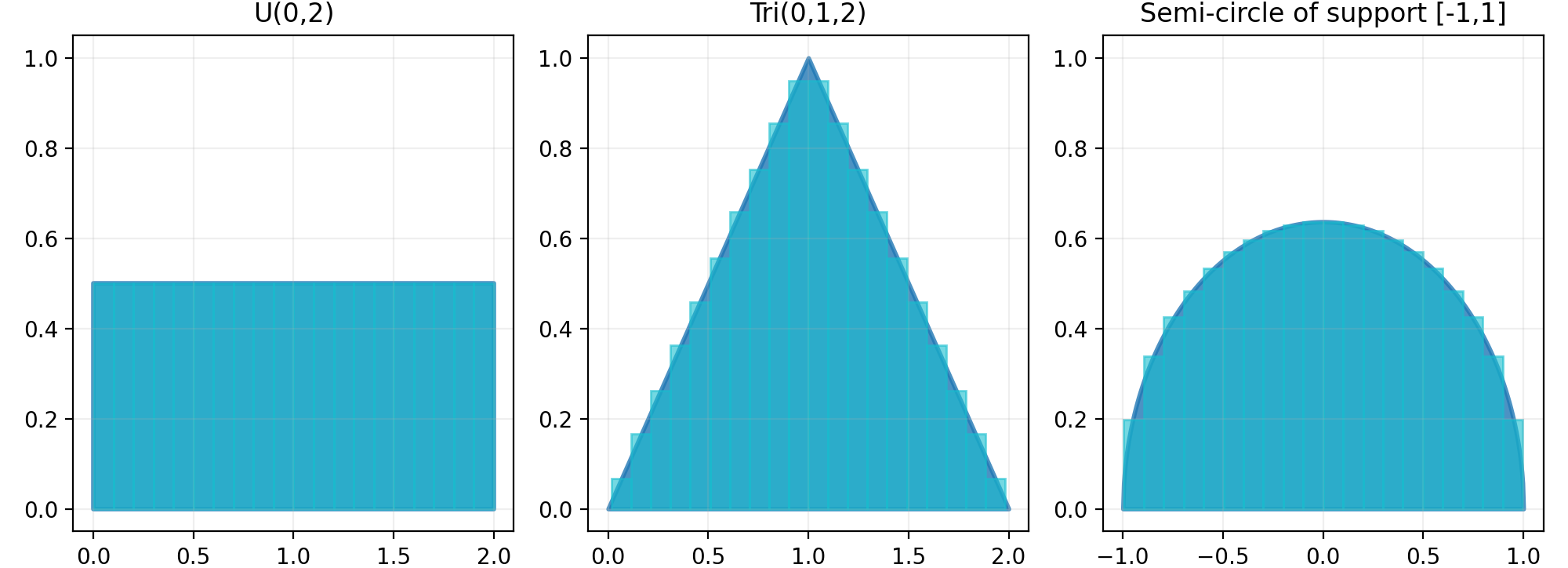}
\end{subfigure}

\bigskip
\begin{subfigure}{0.3\textwidth}
\centering
\includegraphics[width=0.7\textwidth]{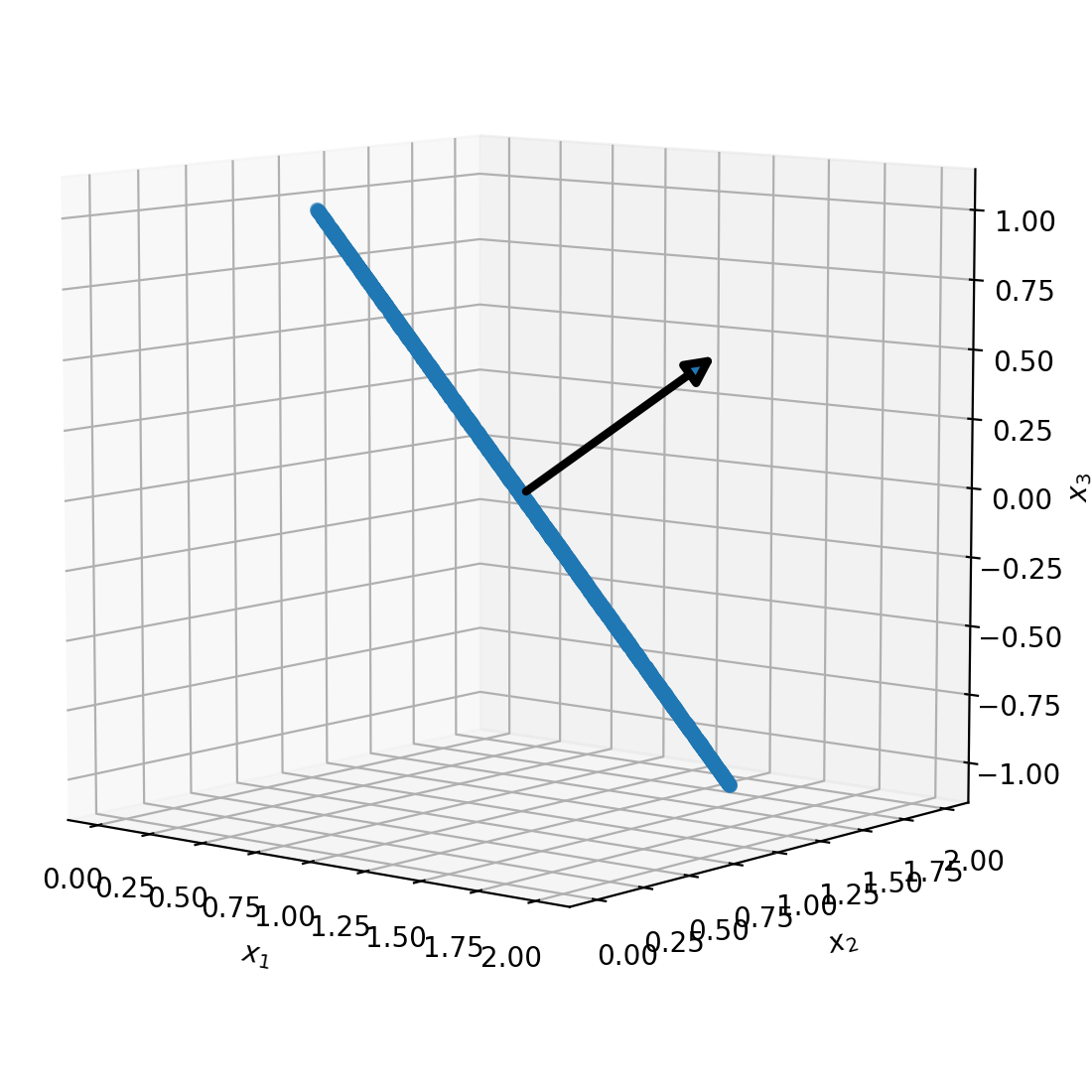}
\end{subfigure}
\begin{subfigure}{0.3\textwidth}
\centering
\includegraphics[width=0.7\textwidth]{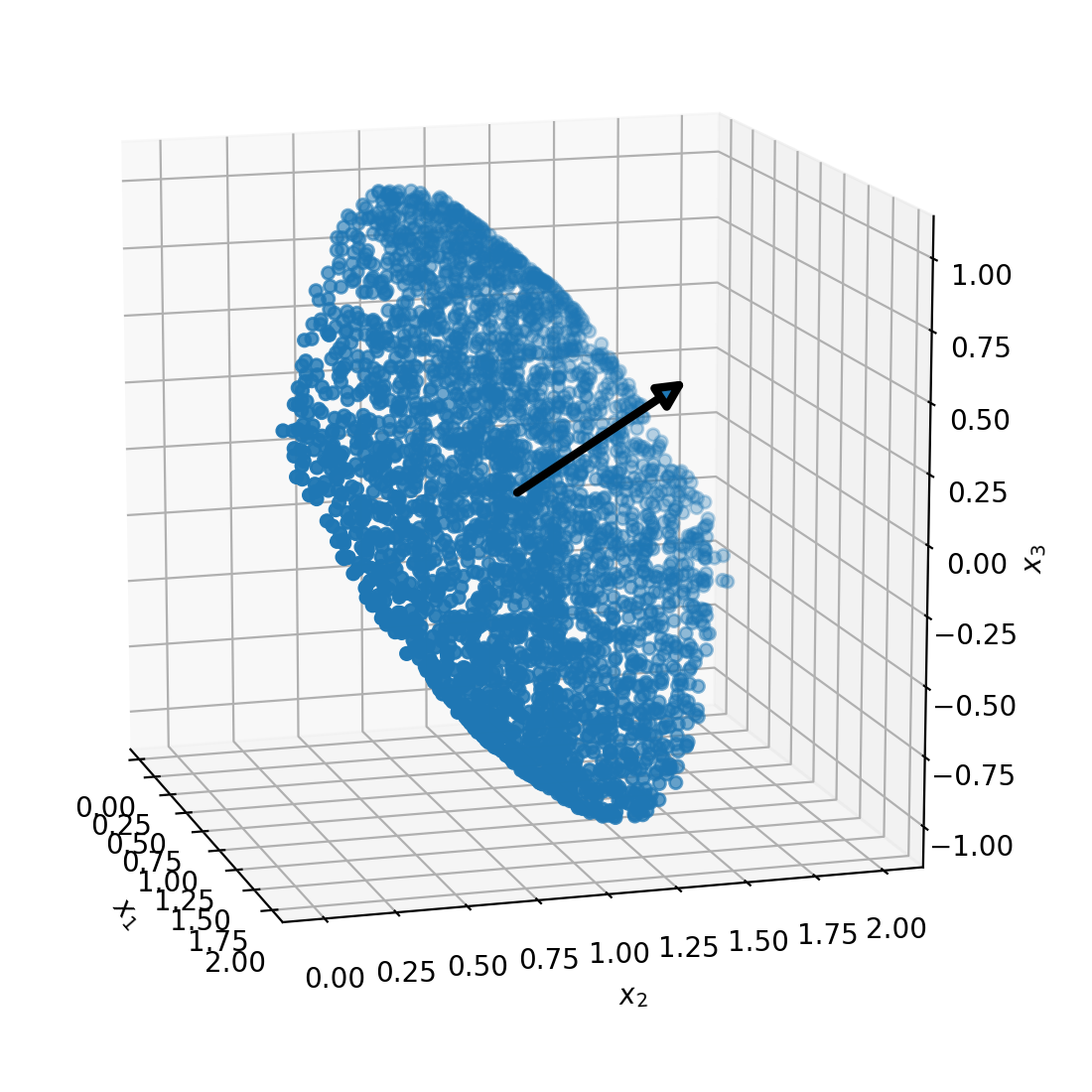}
\end{subfigure}
\begin{subfigure}{0.3\textwidth}
\centering
\includegraphics[width=0.7\textwidth]{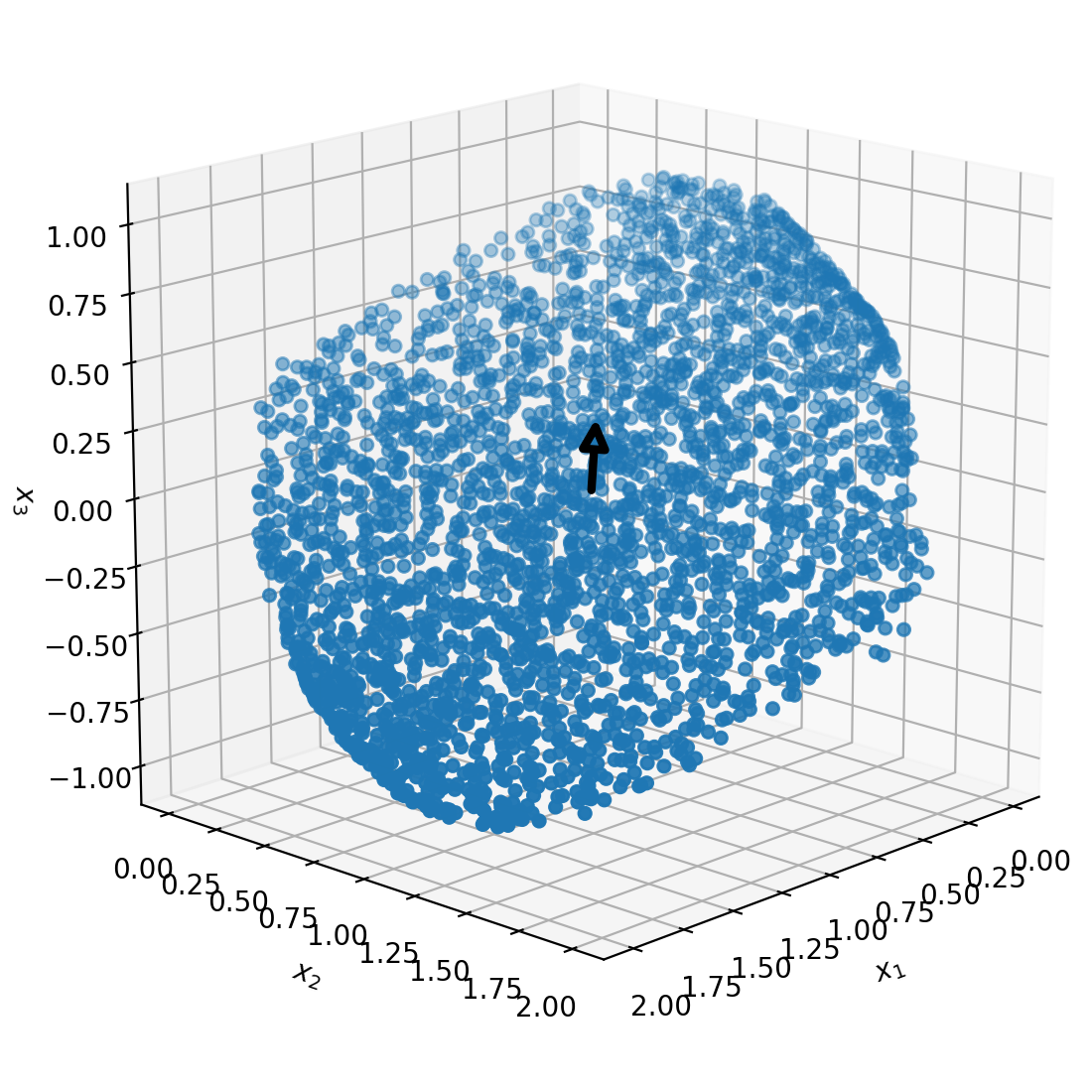}
\end{subfigure}
\caption{
Numerical solution for the standard multimarginal problem with three marginals and surplus cost $c(x_1,x_2,x_3) = -|x_1+x_2+x_3|^2$.
The three marginal densities are shown in the first row, overlapped with the respective histograms of the corresponding projections of our numerical solution. The different 3D views in the second row show that, as expected, the solution concentrates on the plane of equation $x_1+x_2+x_3 = 0$.
We used $N = 3\ 000$ points and the sequence of penalty coefficients
$\lambda \in 10^k : k = -2-,-2,\dots,4\}$.}
\label{fig:squared_sum}
\end{figure}

\subsection{Wasserstein barycenter}
The multimarginal formulation of the partial Wasserstein barycenter problem reads
\begin{equation}
\min_{\gamma \in \Gamma_m(\rho_1,\dots\rho_D)}
\int_{\R^D} \sum_{j,k=1}^D \lambda_j \lambda_j |x_j - x_k|^2 d\gamma(x).
\end{equation}

In Figure~\ref{fig:example_Brendan}, we reproduce numerically the example given by Kitagawa and Pass to prove ~\cite[Proposition~4.1]{kitagawa2015multi}. Their example illustrates non-monotonicity of the partial barycenter in its mass $m$, even for $D=2$ measures.

It turns out that our discretization does not perform well on the partial barycenter problem, as it tends to fall into nonlocal minima.
This issue is illustrated in Figures~\ref{fig:translated_pyramid} and~\ref{fig:one_to_two_pyramids}, which both correspond to the partial barycenter computation of $D=2$ measures, for the sake of simplicity.
The numerical solution in Figure~\ref{fig:translated_pyramid} is fine, since it selects the common part of both marginals, which has exactly the prescribed mass. Also, note that the underlying partial transport plan is indeed concentrated on the line $x_1=x_2$.
Figure~\ref{fig:one_to_two_pyramids} shows another example, in which we once again prescribe the mass of the barycenter to be that of the common part of the two densities. Even though the numerical solution still has has one-dimensional support, it misses some of the common mass and selects it elsewhere, yielding a strictly positive total cost.
This raises the question of the existence of a critical point as $N\to\infty$.
As mentioned above, we initialize the Limited-memory BFGS algorithm by taking a cloud of points independently drawn from the uniform measure on the product of the supports of the marginals.
Perhaps a more astute initialization would allow for better numerical solutions of our discretized partial barycenter problem, but it remains to be found.

\begin{figure}
\begin{subfigure}{0.5\textwidth}
\centering
\includegraphics[width=0.4\textwidth]{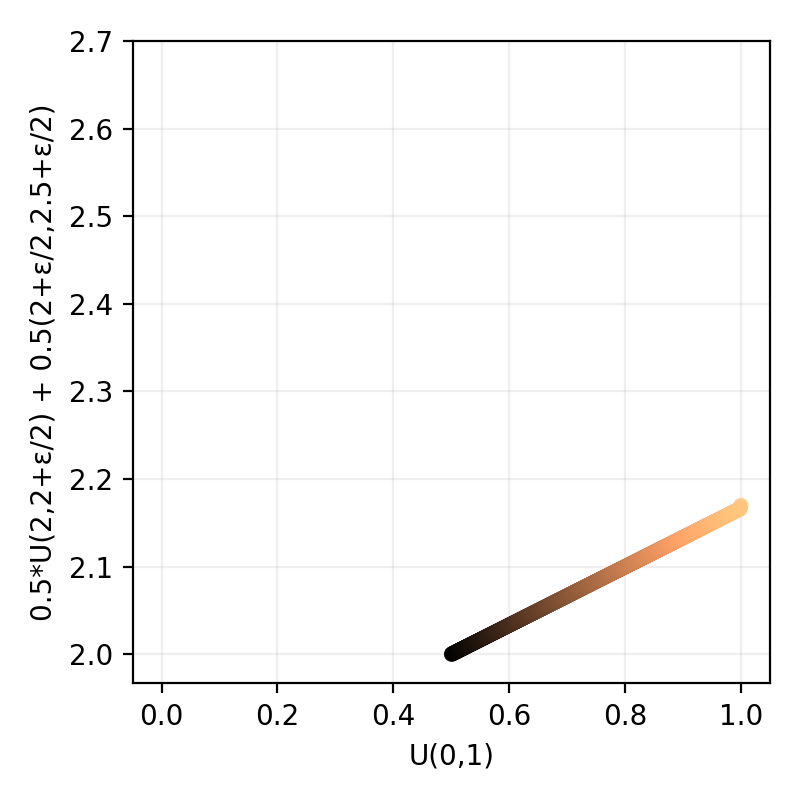}
\caption{Partial transport plan, $m=0.5$}
\end{subfigure}\hfill
\begin{subfigure}{0.5\textwidth}
\centering
\includegraphics[width=0.7\textwidth]{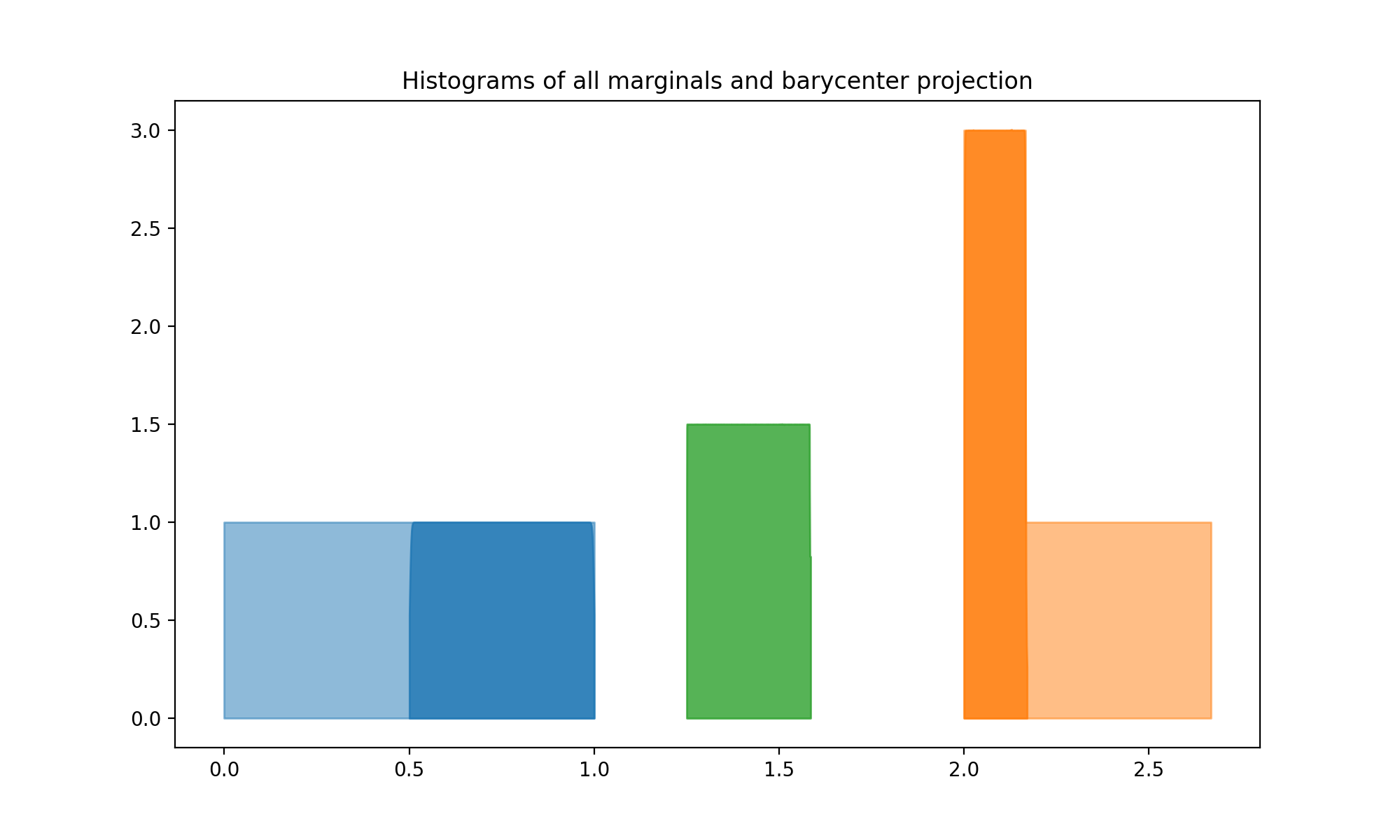}
\caption{Partial barycenter, $m=0.5$}
\end{subfigure}

\medskip
\begin{subfigure}{0.5\textwidth}
\centering
\includegraphics[width=0.4\textwidth]{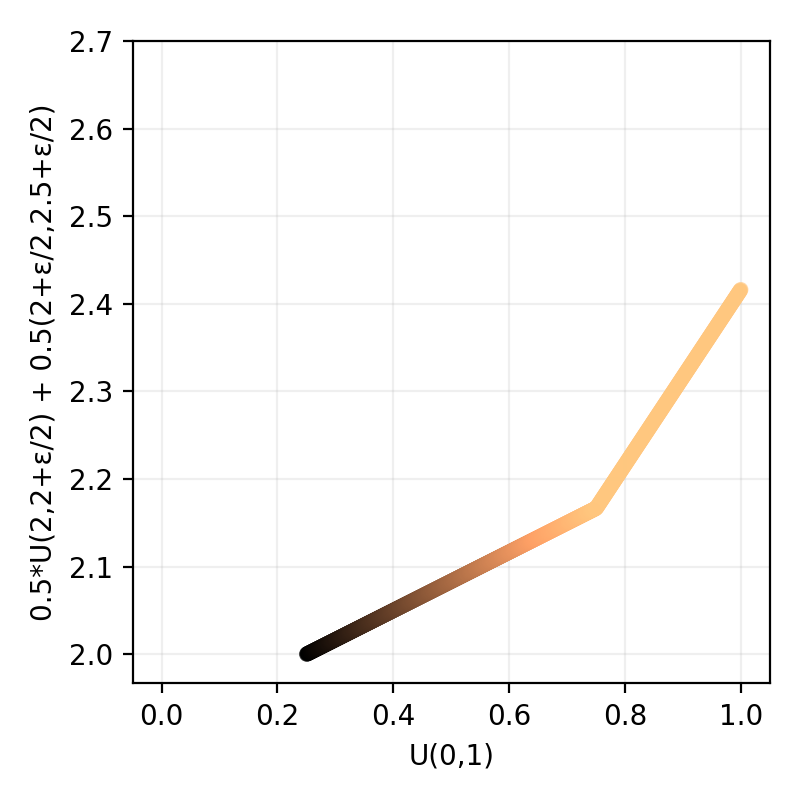}
\caption{Partial transport plan, $m=0.75$}
\end{subfigure}\hfill
\begin{subfigure}{0.5\textwidth}
\centering
\includegraphics[width=0.7\textwidth]{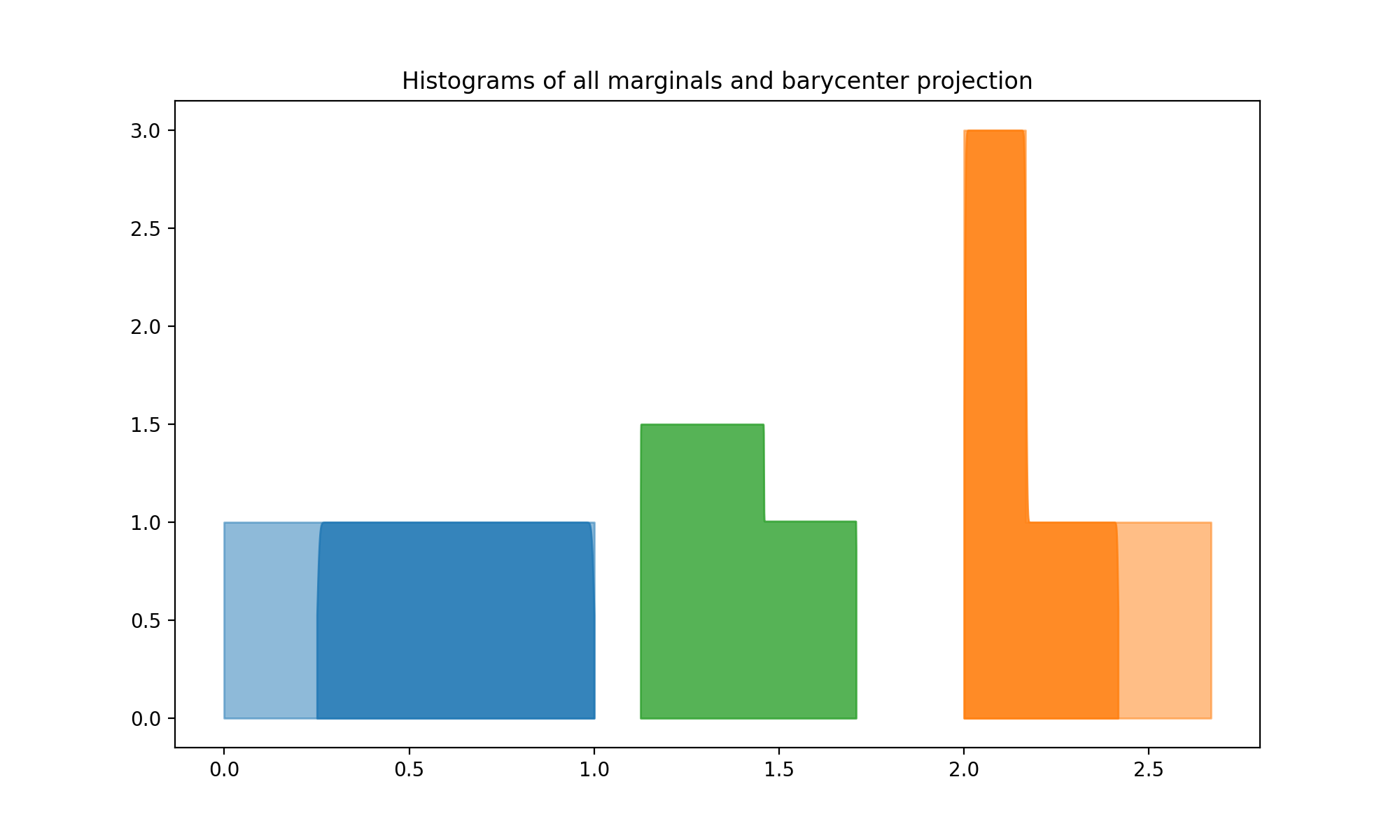}
\caption{Partial barycenter, $m=0.75$}
\end{subfigure}
\caption{
Numerical simulation for Kitagawa and Pass' example in the proof of~\cite[Proposition~4.1]{kitagawa2015multi}, which shows non-monotonicity of the partial barycenter (in green) in its mass $m$.
The blue and orange densities correspond to the two probability marginals, with their active parts darkened.
Both for these active parts and for the partial barycenter, we use a simple kernel density estimation to convert the corresponding point clouds obtained numerically into densities.
The value $\eps$ involved in the marginal densities defined by the aforementioned authors was set to $\frac{1}{3}$.
We used $N = 1\ 000$ points, equal weights $\lambda_1 = \lambda_2 = \frac{1}{2}$, and the sequence of penalty coefficients
$\lambda \in \{10^k : k = -1-,0,\dots,4\}$.
}
\label{fig:example_Brendan}
\end{figure}

\begin{figure}[!ht]
\centering
\begin{subfigure}{0.3\textwidth}
\centering
\includegraphics[width=0.9\textwidth]{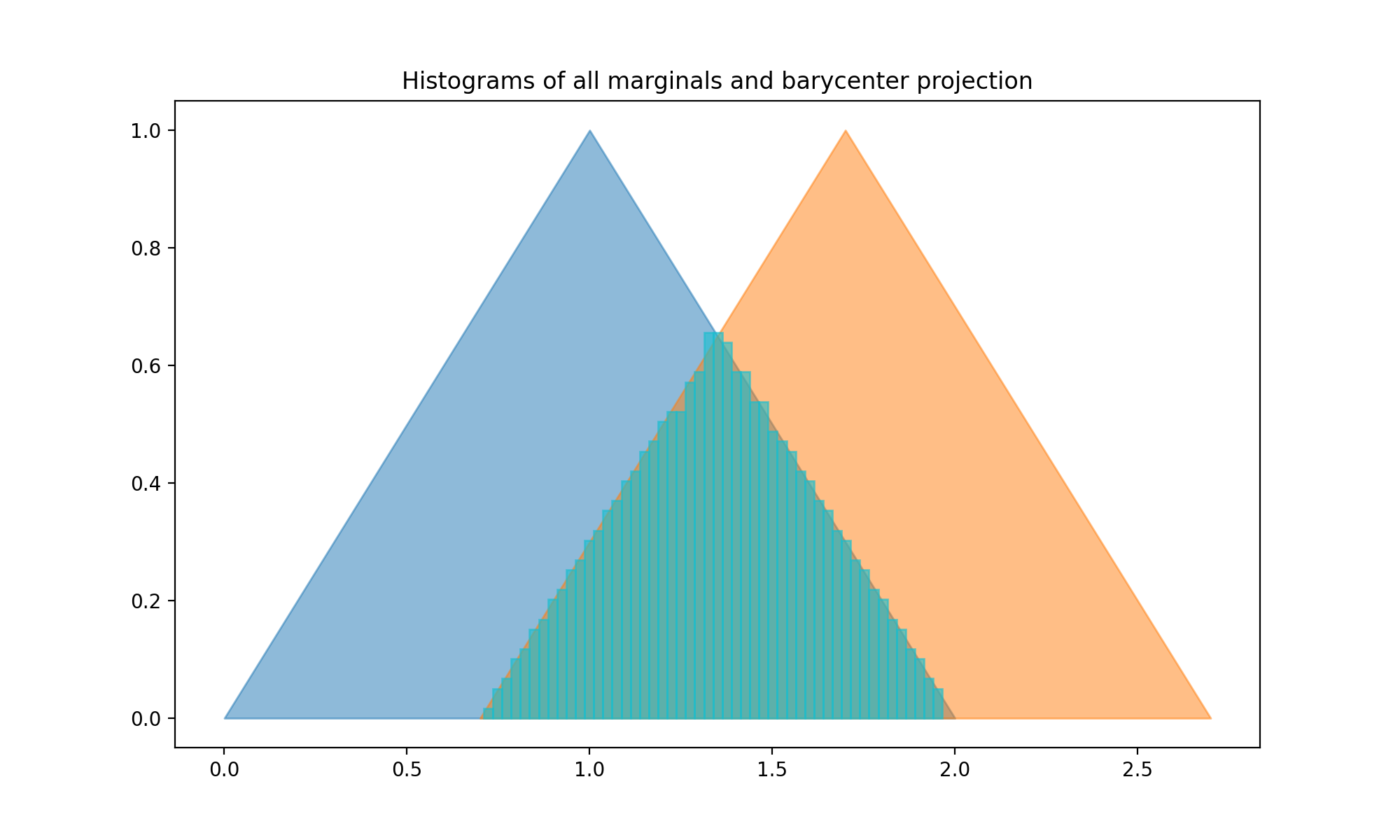}
\caption{Partial barycenter}
\label{fig:translated_pyramid-bary}
\end{subfigure}
\begin{subfigure}{0.3\textwidth}
\centering
\includegraphics[width=0.8\textwidth]{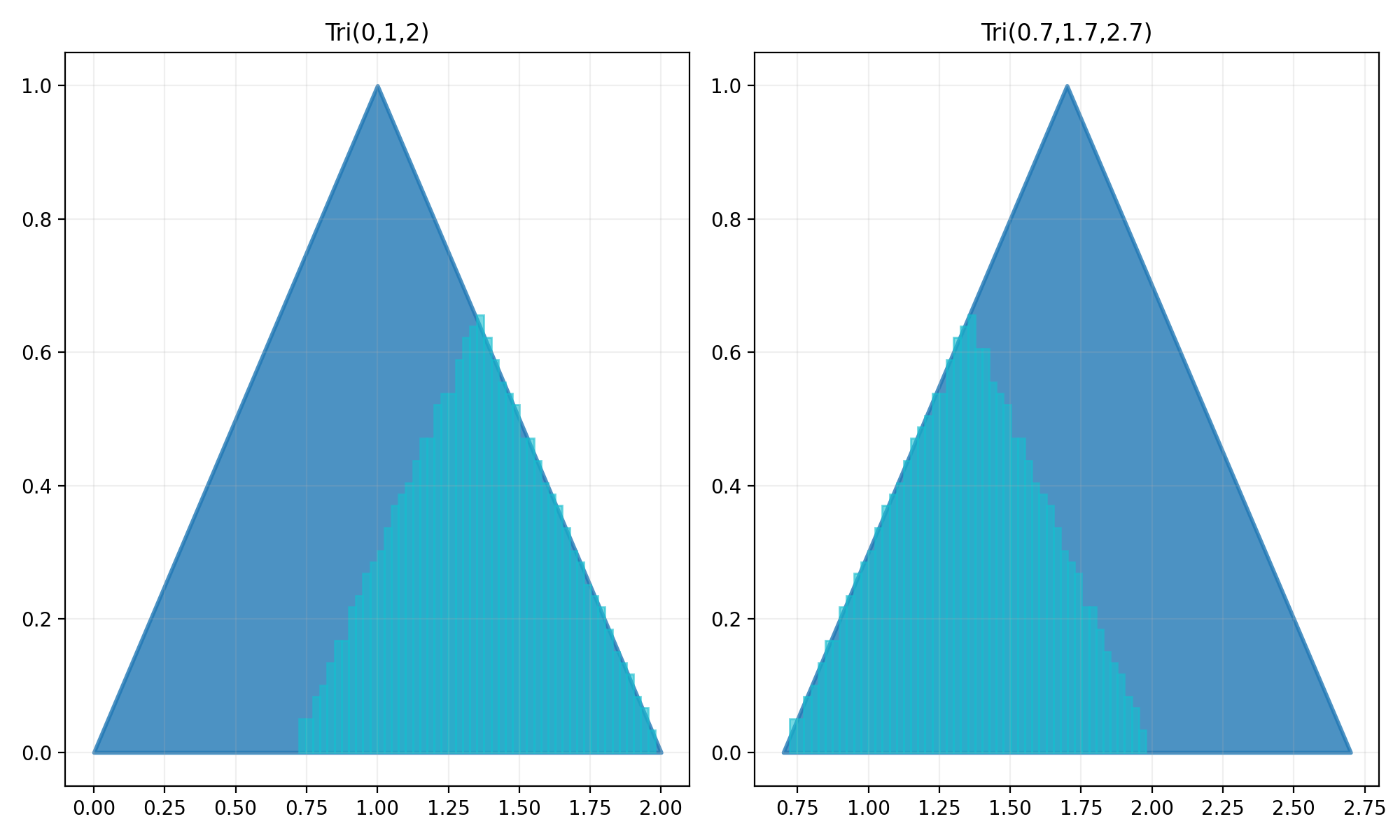}
\caption{Active submeasures}
\label{fig:translated_pyramid-histograms}
\end{subfigure}
\begin{subfigure}{0.3\textwidth}
\centering
\includegraphics[width=0.5\textwidth]{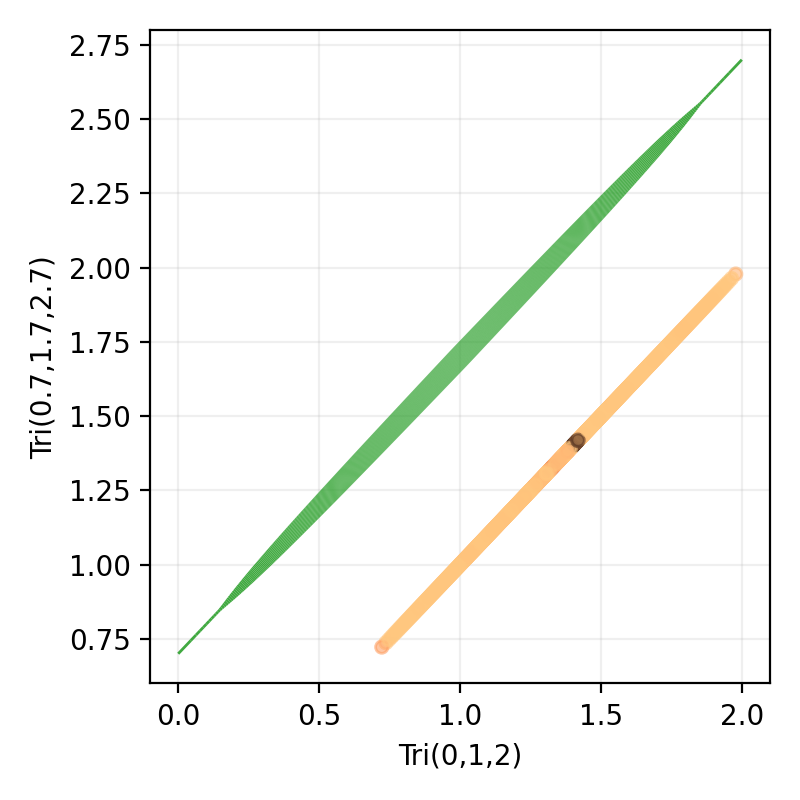}
\caption{Partial transport plan}
\label{fig:translated_pyramid-plan}
\end{subfigure}
\caption{
Numerical solution for the multimarginal formulation of the partial barycenter problem between the triangle density $\mathrm{Tri}(0,1,2)$ and its translation by $0.7$. We used $N = 1\ 000$ points, mass $m = 0.4225$, equal weights $\lambda_1 = \lambda_2 = \frac{1}{2}$, and the sequence of penalty coefficients
$\lambda \in \{ 10^k : k = -3-,-2,\dots,2\}$. The chosen mass is exactly equal to the amount of common mass of the two marginal probability measures. The orange point cloud in~\ref{fig:translated_pyramid-plan} represents the partial transport plan we find whereas the green line is the standard optimal transport plan between the two marginals.}
\label{fig:translated_pyramid}
\end{figure}

\begin{figure}[!ht]
\centering
\begin{subfigure}{0.3\textwidth}
\centering
\includegraphics[width=0.9\textwidth]{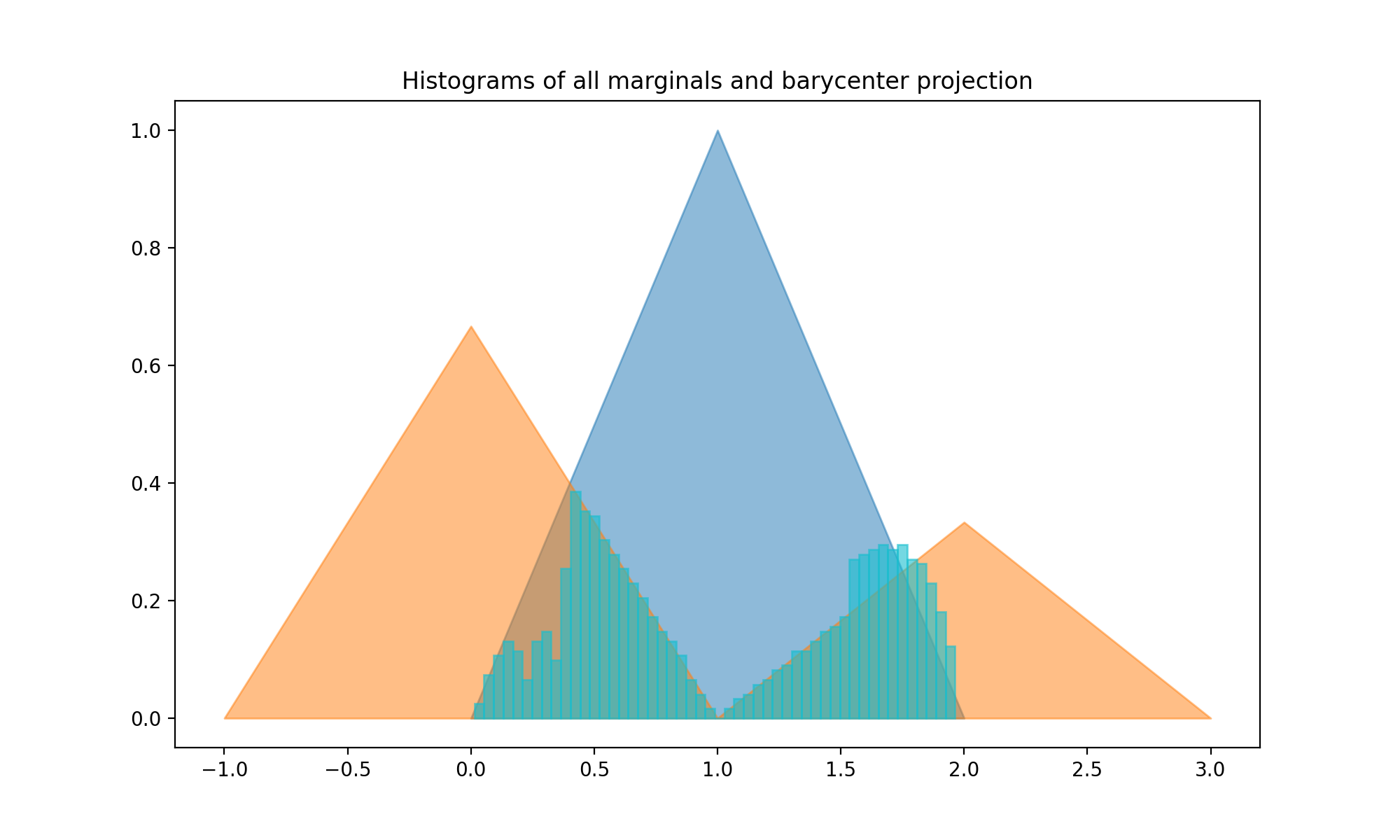}
\caption{Partial barycenter}
\label{fig:one_to_two_pyramids-bary}
\end{subfigure}
\begin{subfigure}{0.3\textwidth}
\centering
\includegraphics[width=0.9\textwidth]{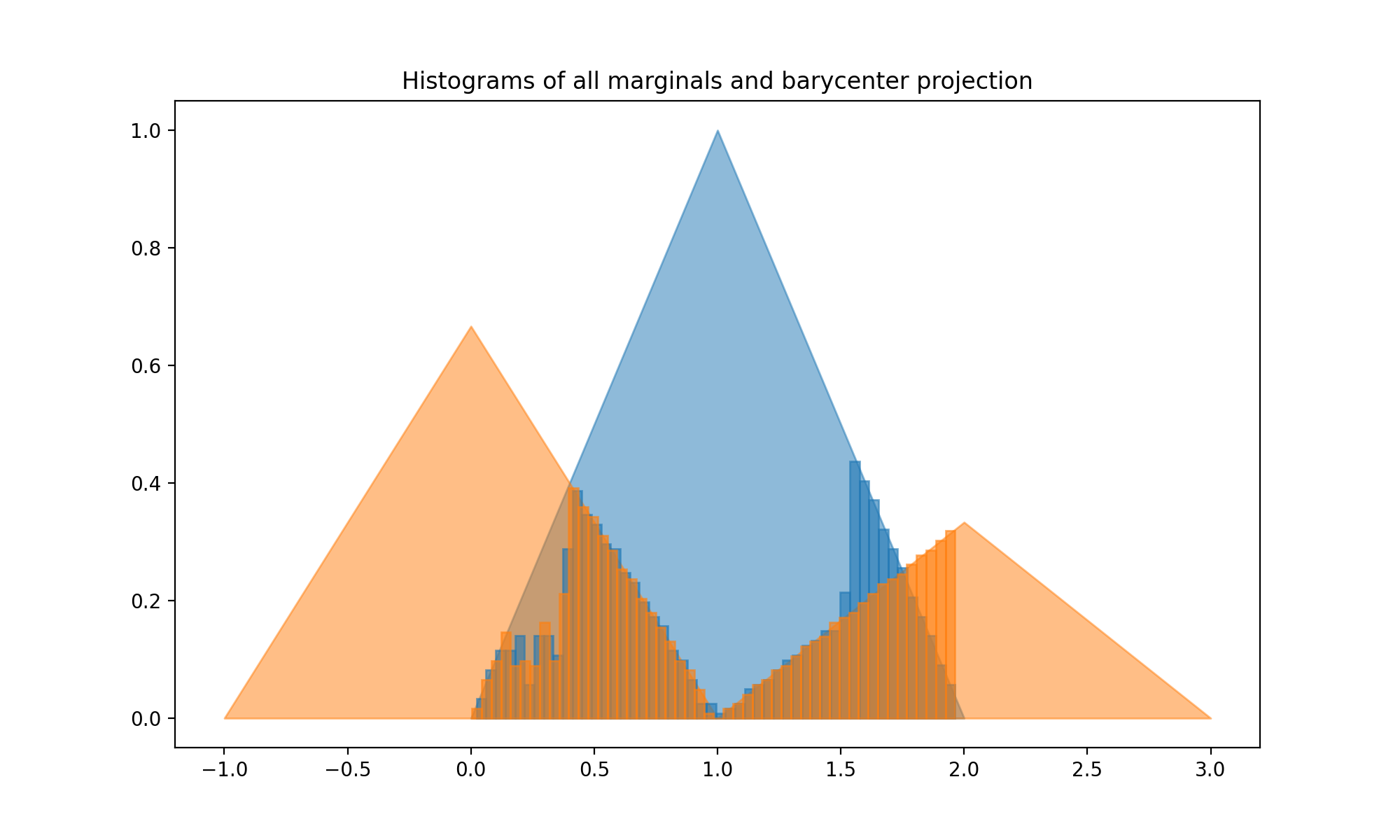}
\caption{Active submeasures}
\label{fig:one_to_two_pyramids-histograms}
\end{subfigure}
\begin{subfigure}{0.3\textwidth}
\centering
\includegraphics[width=0.5\textwidth]{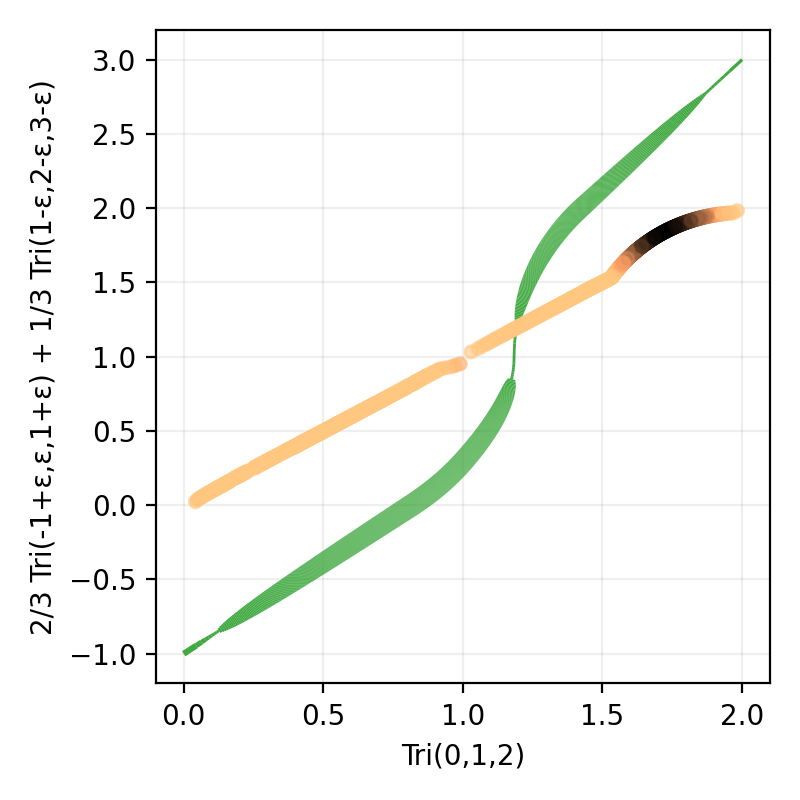}
\caption{Partial transport plan}
\label{fig:one_to_two_pyramids-plan}
\end{subfigure}
\caption{Numerical solution for the multimarginal formulation of the partial barycenter problem between the triangle density $\mathrm{Tri}(0,1,2)$ and the density $\frac{2}{3}\mathrm{Tri}(-1,0,1) + \frac{1}{3}\mathrm{Tri}(1,2,3)$. We used $N = 1\ 000$ points, mass $m = 0.4225$, equal weights $\lambda_1 = \lambda_2 = \frac{1}{2}$, and the sequence of penalty coefficients
$\lambda \in 10^k : k = -3-,-2,\dots,6\}$. The chosen mass is exactly equal to the amount of common mass of the two marginal probability measures, and in particular, we see that the numerical solution is far from optimal.
}
\label{fig:one_to_two_pyramids}
\end{figure}

\subsection{Repulsive cost}

The standard multimarginal optimal transport problem with Coulomb cost $c$ and univariate marginal $\rho \in \mathcal{P}(\R)$ reads
\begin{equation}
\label{eq:DFT_problem}
\min_{\gamma\in\Gamma^D(\rho)}
\int_{\R^D} c(x_1,\dots,x_D) d\gamma(x_1,\dots,x_D),
\end{equation}
where
\begin{equation}
c(x_1,\dots,x_D) =
\sum_{1\leq j < k \leq D}
\frac{1}{|x_j-x_k|}
\end{equation}
defines the Coulomb cost, and $\Gamma^D(\rho)$ is the set of probability measures on $\R^D$ whose respective projections on each of the $D$ axes are all equal to $\rho$.
This problem arises naturally in density functional theory. It is well known that thanks to the equal marginals, to linearity of the objective, and to symmetry of the cost with respect to permutation of the arguments, at least one solution $\gamma \in \Gamma^D(\rho)$ of Problem~\eqref{eq:DFT_problem} is \textit{symmetric}, in the sense that it is invariant with respect to permutation of the axes.

The spectral risk measure version of Problem~\eqref{eq:DFT_problem} reads
\begin{equation}
\label{eq:risk_DFT_problem}
\max_{\gamma\in\Gamma^D(\rho)}
\Rr_\alpha((-c)_\#\gamma),
\end{equation}
whereas the partial transport version reads
\begin{equation}
\label{eq:partial_DFT_problem}
\min_{\gamma\in\Gamma^D_m(\rho)}
\int c \, d\gamma,
\end{equation}
where $\Gamma^D_m(\rho)$ is the set of positive measures of mass $m$ on $\R^D$ whose respective projections on each of the $D$ axes are all dominated by $\rho$.
We emphasize that these two problems are purely artificial, as they have no physical motivation.
Nonetheless, they make for a good illustration of our discretization method in the case of a repulsive cost function.

\begin{remark}
For numerics, we replaced the pairwise interactions $|x_j-x_k|^{-1}$ in the Coulomb cost $c$ by $(1+|x_j-x_k|)^{-1}$, in order to keep the cost function bounded.
\end{remark}

Figure~\ref{fig:DFT-CVaR_simulations} shows the numerical solutions obtained for the Lagrangian discretization of Problem~\eqref{eq:partial_DFT_problem} with mass $m = 0.5$ and uniform marginal $\rho = \L_{(0,1)}$, for various numbers $D$ of marginals.
One can see that, up to restricting to the active submeasures, the solution has exactly the Monge type cyclical structure described in~\cite[Theorem~1.1]{colombo2015multimarginal} by Colombo et al. This is not surprising, as any symmetric solution $\gamma$ of Problem~\eqref{eq:partial_DFT_problem} has equal active measures $\pi^j_\#\gamma$ and solves the standard Coulomb cost problem with marginal $\tilde{\rho} = \pi^1_\#\gamma$, for which the result in question holds.
If the dependence structure of the $D$ identical active submeasures is well understood in light of~\cite{colombo2015multimarginal}, we are not aware of any characterization of this common active submeasure.

In Figure~\ref{fig:DFT-quadratic_simulations}, we present the numerical solutions for Problem~\eqref{eq:risk_DFT_problem}, once again with marginal $\rho = \L_{(0,1)}$, but this time for the spectral value at risk induced by the quadratic spectral function $\alpha(t) = (t+\eta)^2-\eta^2$ (in particular, no mass is left out). The offset $\eta=0.1$ ensures that that nonzero derivative at $t=0$. Of course, since the objective function is concave, the same argument as for the standard (linear) problem implies that there is always a solution that is symmetric. In that respect, the symmetries of each of the nine point clouds is of little surprise.
But interestingly, the numerical simulations strongly suggests that solutions have one-dimensional supports.

\begin{figure}
\begin{subfigure}{0.33\textwidth}
\includegraphics[width=\textwidth]{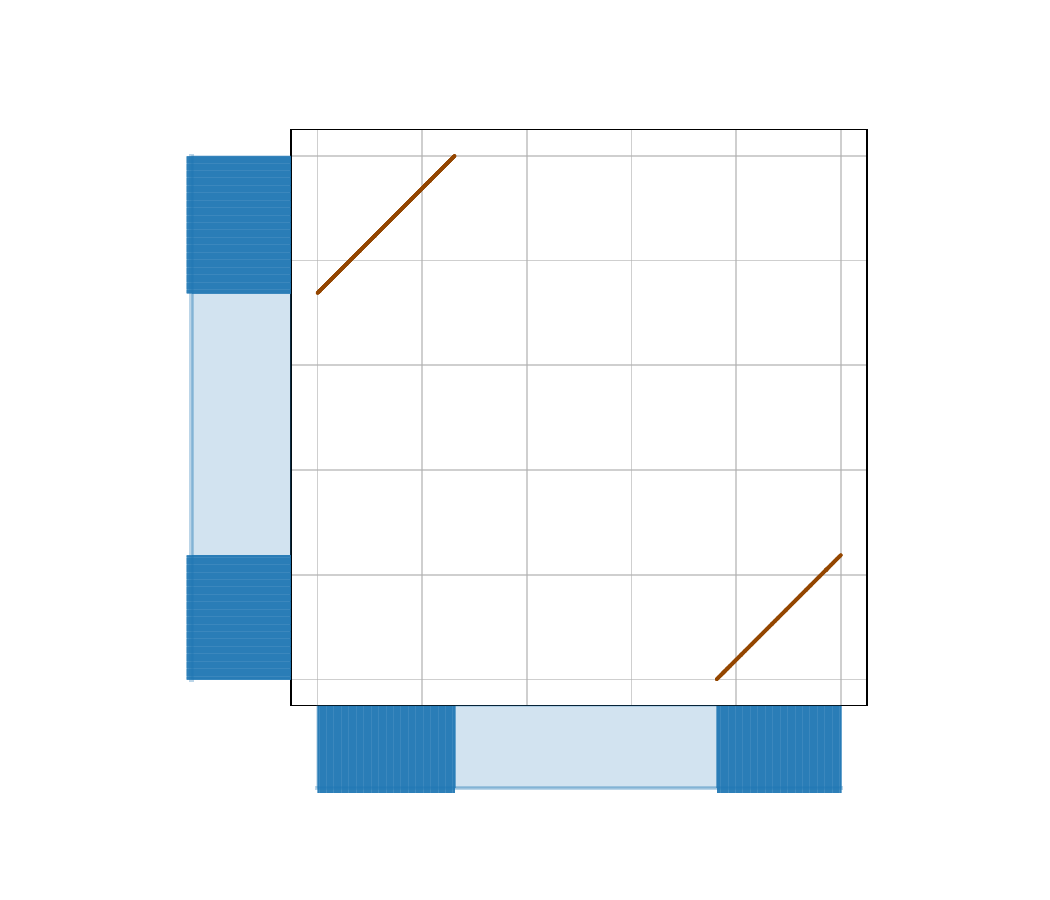}
\caption{$D = 2$}
\label{subfig:DFT-CVaR-D2}
\end{subfigure}\hfill
\begin{subfigure}{0.33\textwidth}
\includegraphics[width=\textwidth]{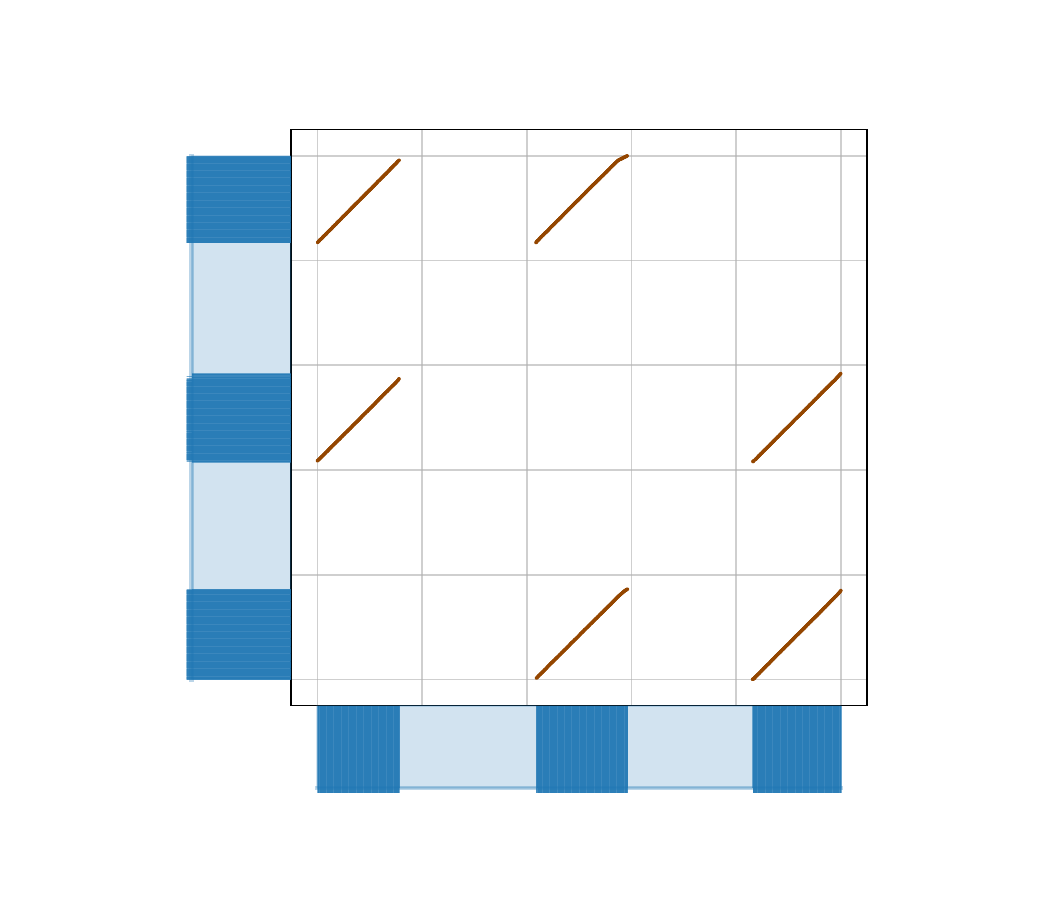}
\caption{$D = 3$}
\label{subfig:DFT-CVaR-D3}
\end{subfigure}\hfill
\begin{subfigure}{0.33\textwidth}
\includegraphics[width=\textwidth]{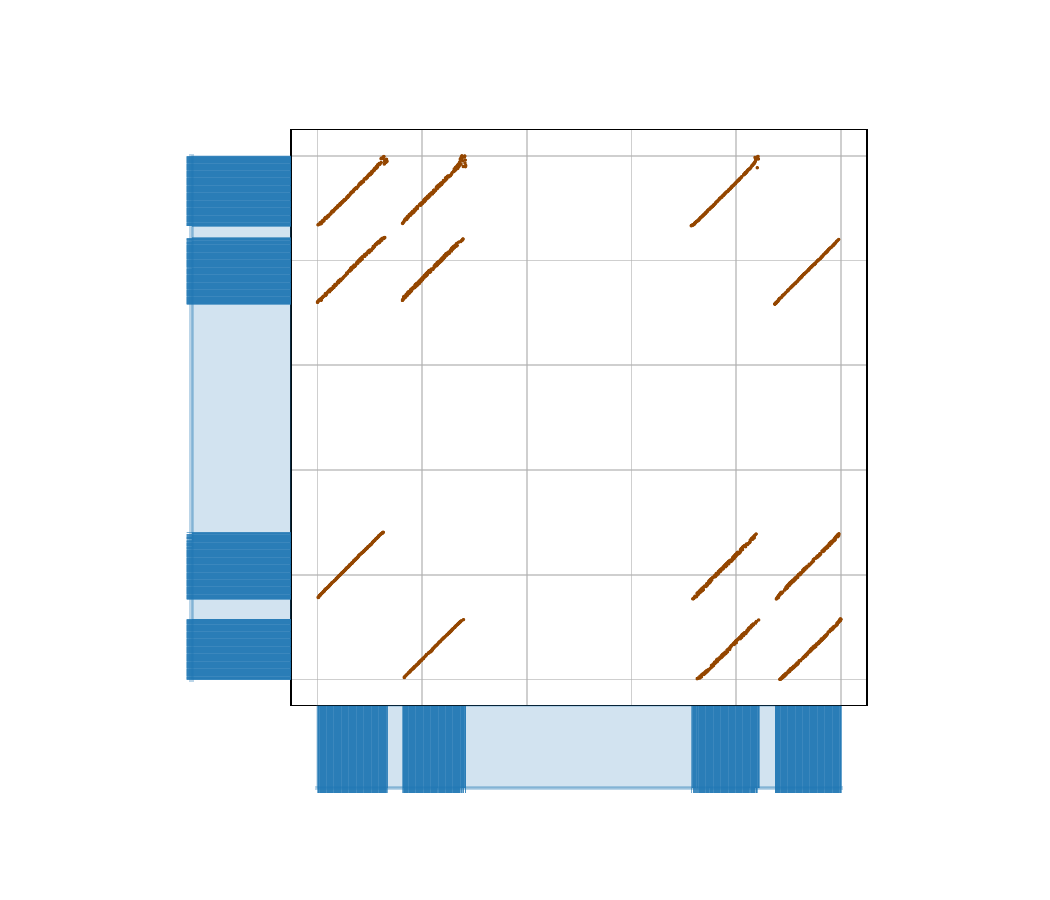}
\caption{$D = 4$}
\label{subfig:DFT-CVaR-D4}
\end{subfigure}

\bigskip
\begin{subfigure}{0.33\textwidth}
\includegraphics[width=\textwidth]{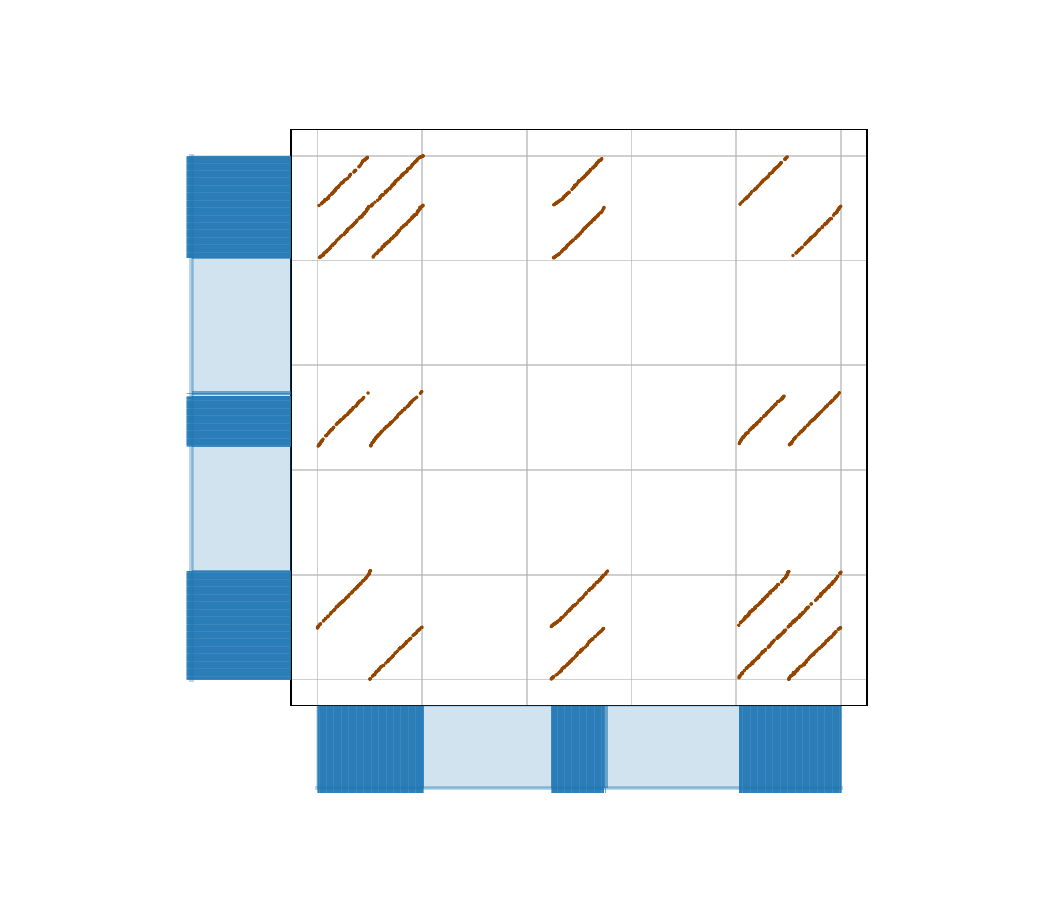}
\caption{$D = 5$}
\label{subfig:DFT-CVaR-D5}
\end{subfigure}\hfill
\begin{subfigure}{0.33\textwidth}
\includegraphics[width=\textwidth]{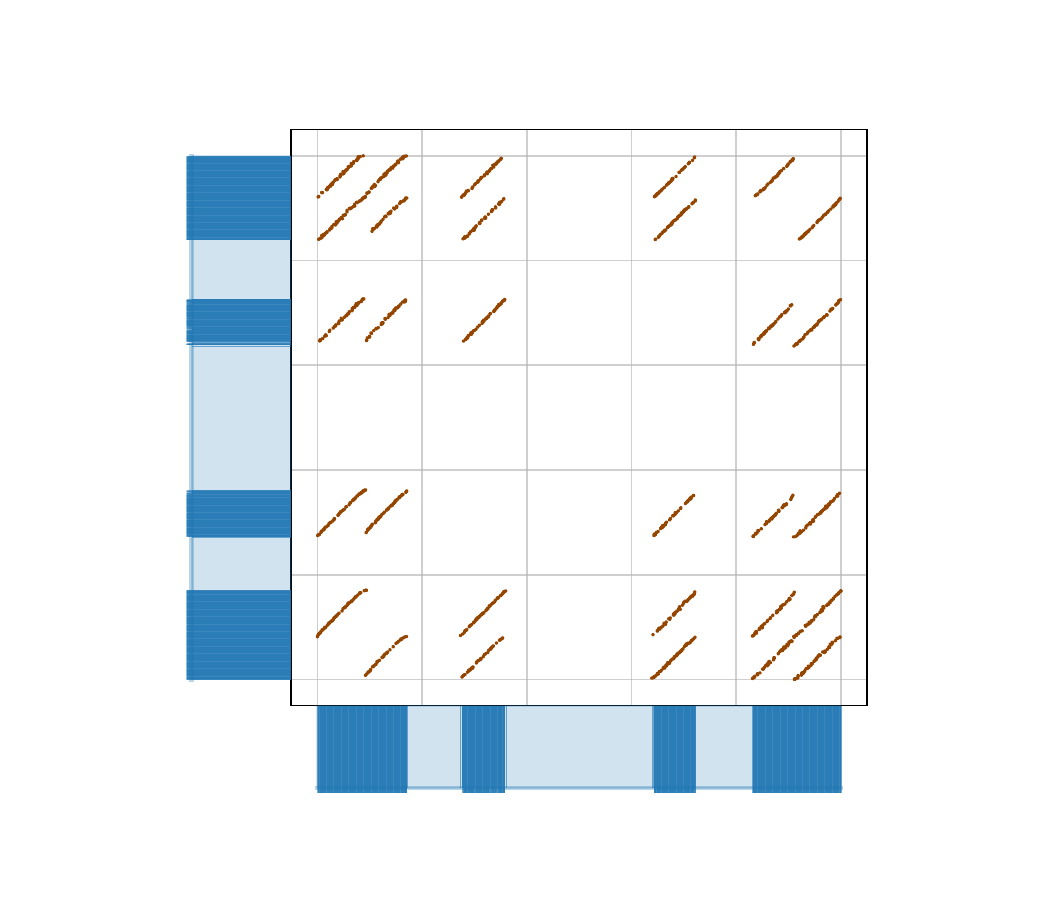}
\caption{$D = 6$}
\label{subfig:DFT-CVaR-D6}
\end{subfigure}\hfill
\begin{subfigure}{0.33\textwidth}
\includegraphics[width=\textwidth]{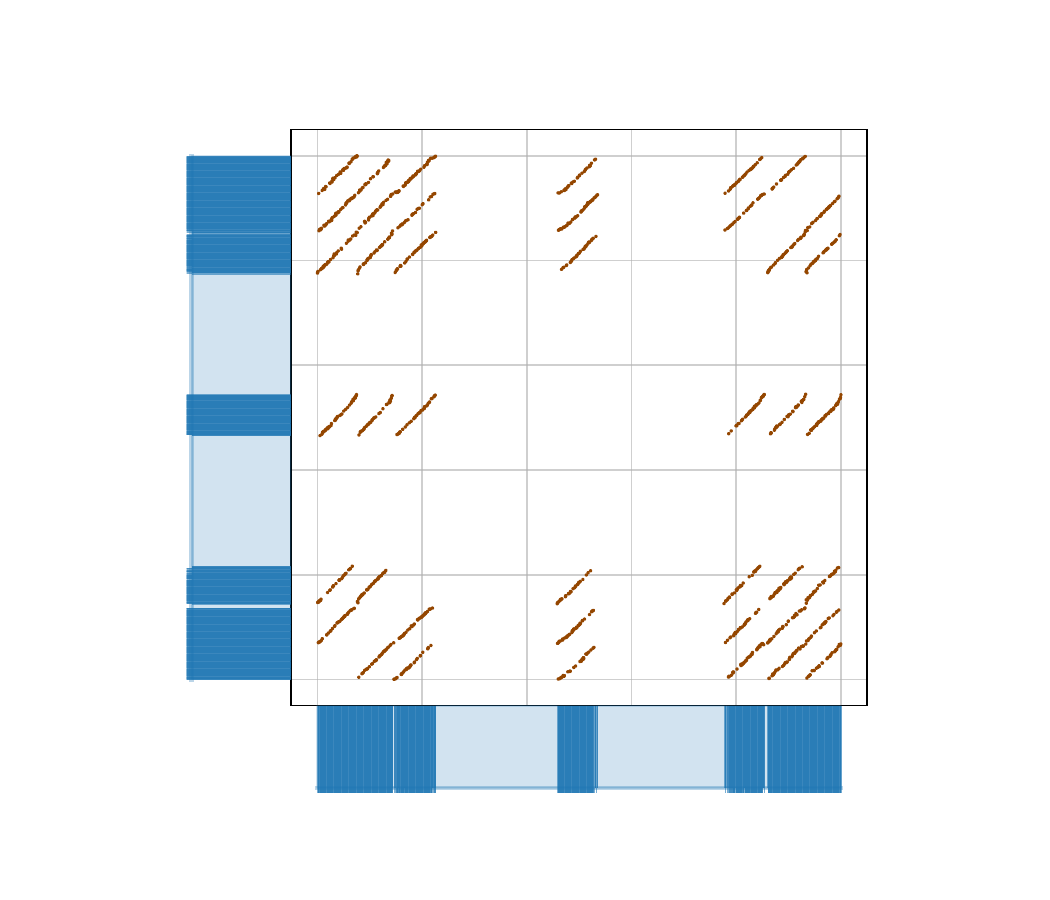}
\caption{$D = 7$}
\label{subfig:DFT-CVaR-D7}
\end{subfigure}
\caption{Projections on the first two axes of the numerical solutions found for Problem~\eqref{eq:risk_DFT_problem}, with $\rho = \L_{(0,1)}$ the uniform measure on $[0,1]$ and $\alpha \propto \mathbbm{1}_{(1-m,1)}$ the spectral function corresponding to $\mathrm{CVaR}_m$. Each subfigure corresponds to a different number $D$ of marginals.
We used $N = 1\ 500$ points, mass $m =0.5$, and the sequence of penalty coefficients $\lambda \in \{10^k : k=-3,-2,\dots,6 \}$.
The blue area at the bottom and left sides of each figure represent the active submeasures.
}
\label{fig:DFT-CVaR_simulations}
\end{figure}

\begin{figure}
\begin{subfigure}{0.29\textwidth}
\includegraphics[width=\textwidth]{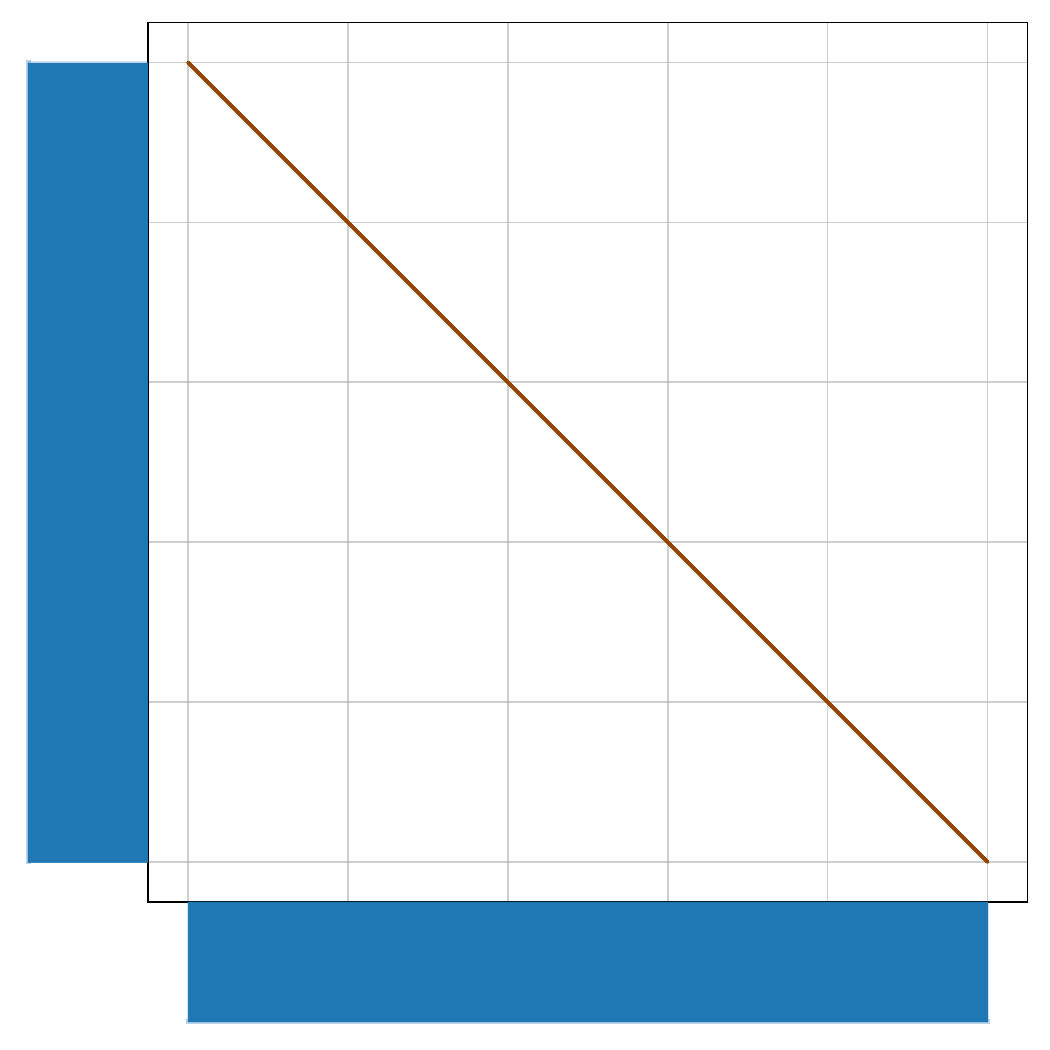}
\caption{$D = 2$}
\end{subfigure}
\hfill
\begin{subfigure}{0.29\textwidth}
\includegraphics[width=\textwidth]{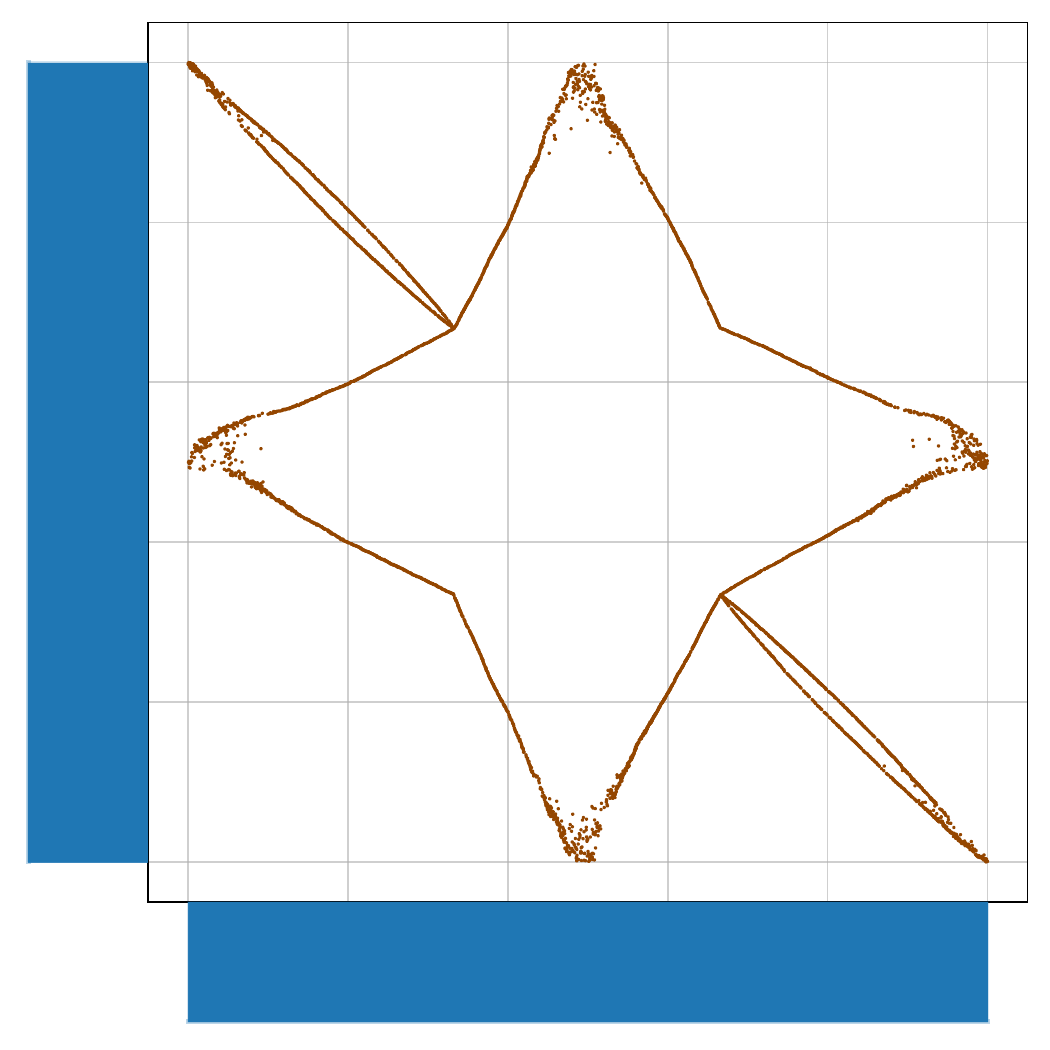}
\caption{$D = 3$}
\end{subfigure}
\hfill
\begin{subfigure}{0.29\textwidth}
\includegraphics[width=\textwidth]{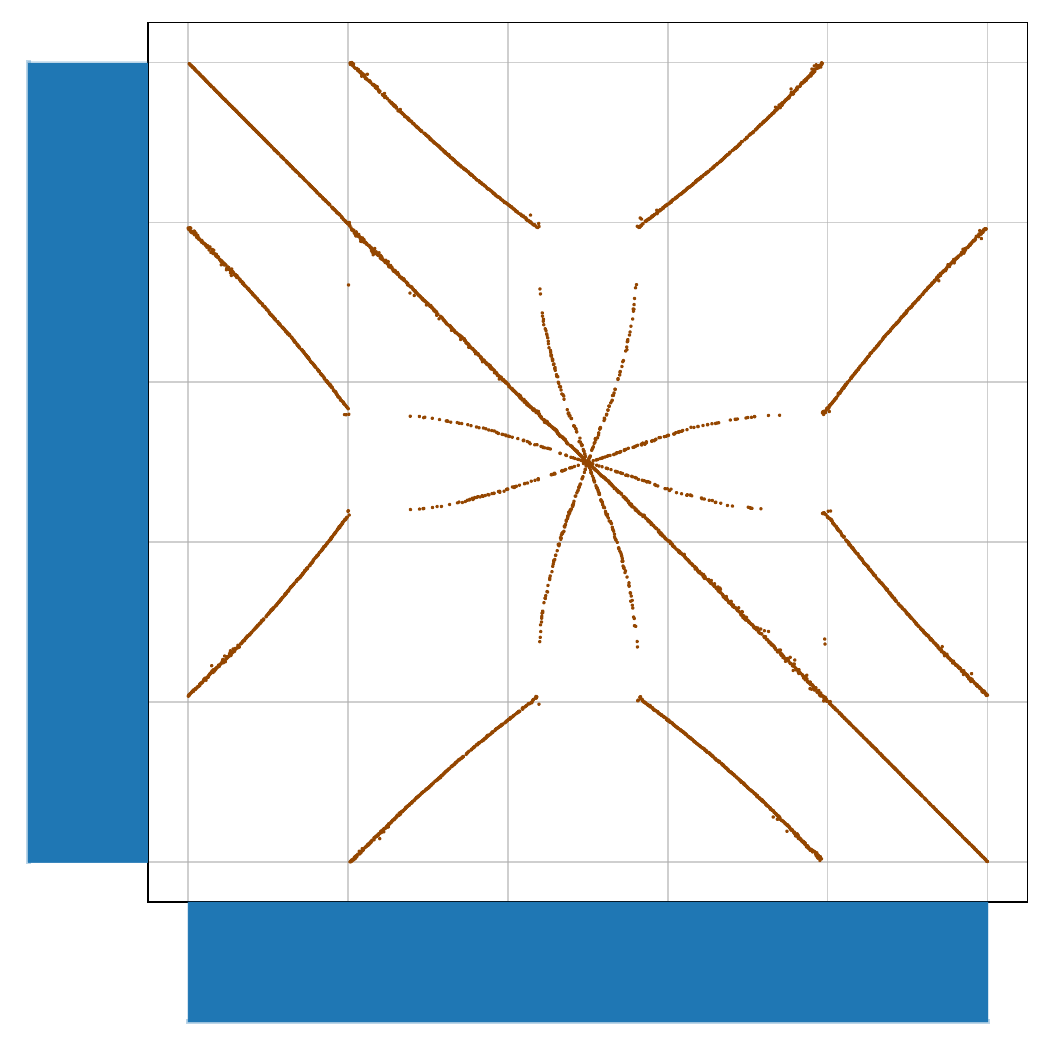}
\caption{$D = 4$}
\end{subfigure}

\medskip
\begin{subfigure}{0.29\textwidth}
\includegraphics[width=\textwidth]{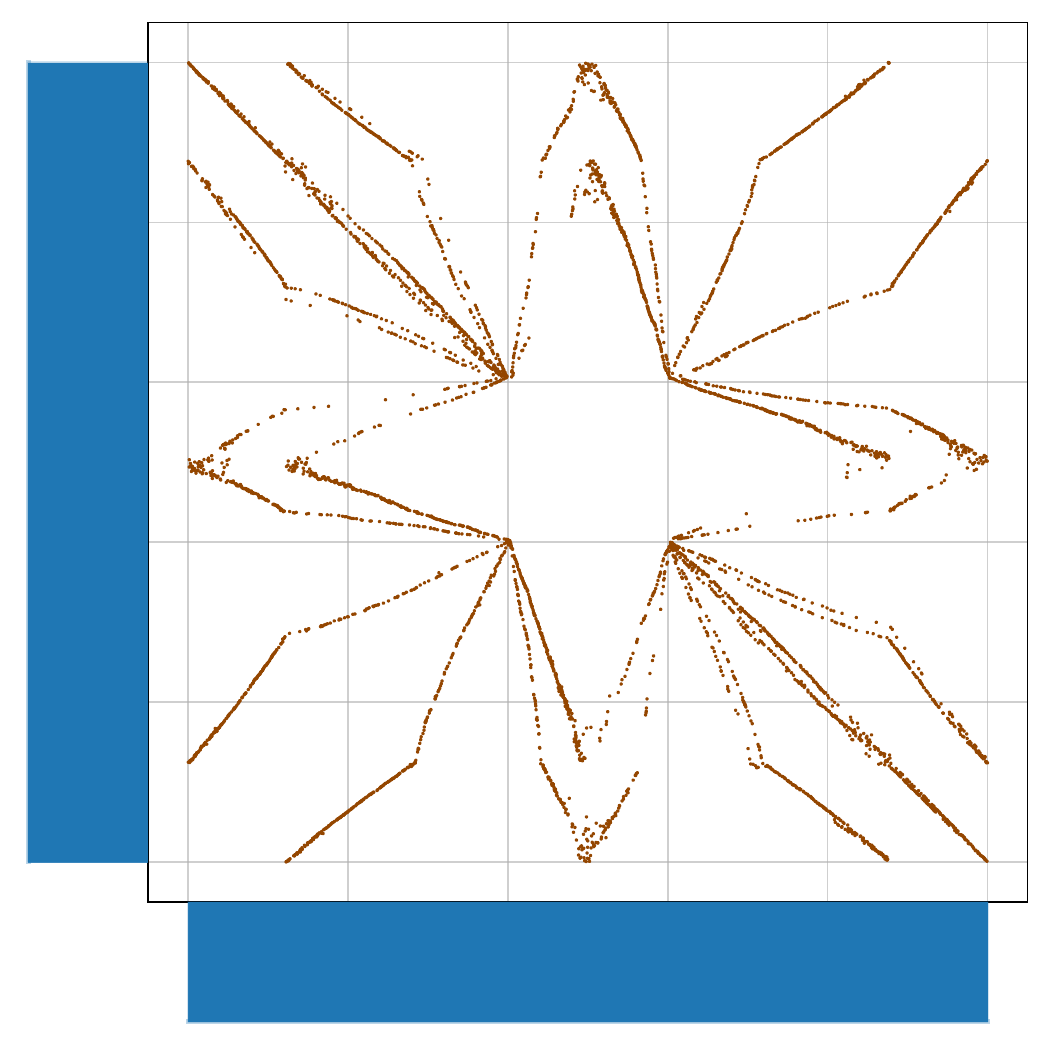}
\caption{$D = 5$}
\end{subfigure}
\hfill
\begin{subfigure}{0.29\textwidth}
\includegraphics[width=\textwidth]{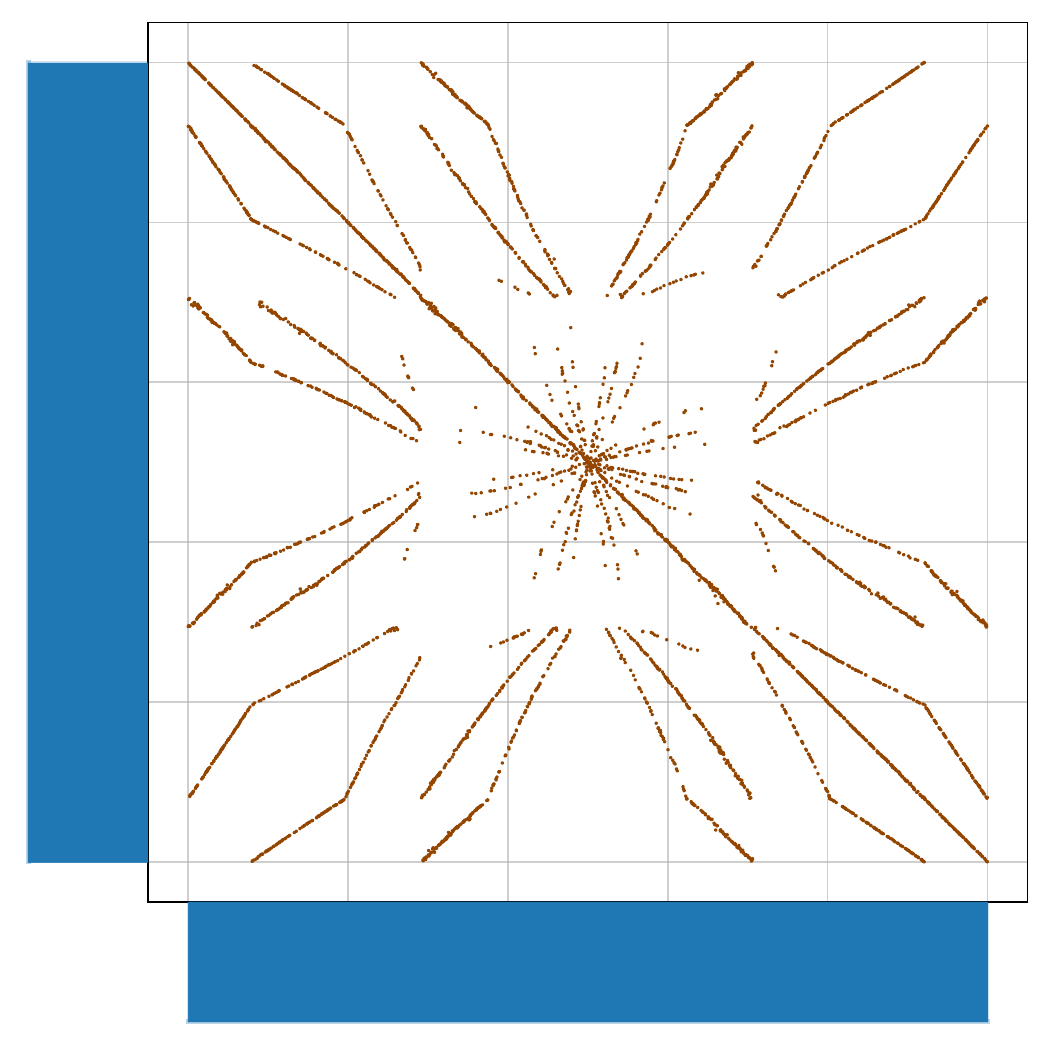}
\caption{$D = 6$}
\end{subfigure}
\hfill
\begin{subfigure}{0.29\textwidth}
\includegraphics[width=\textwidth]{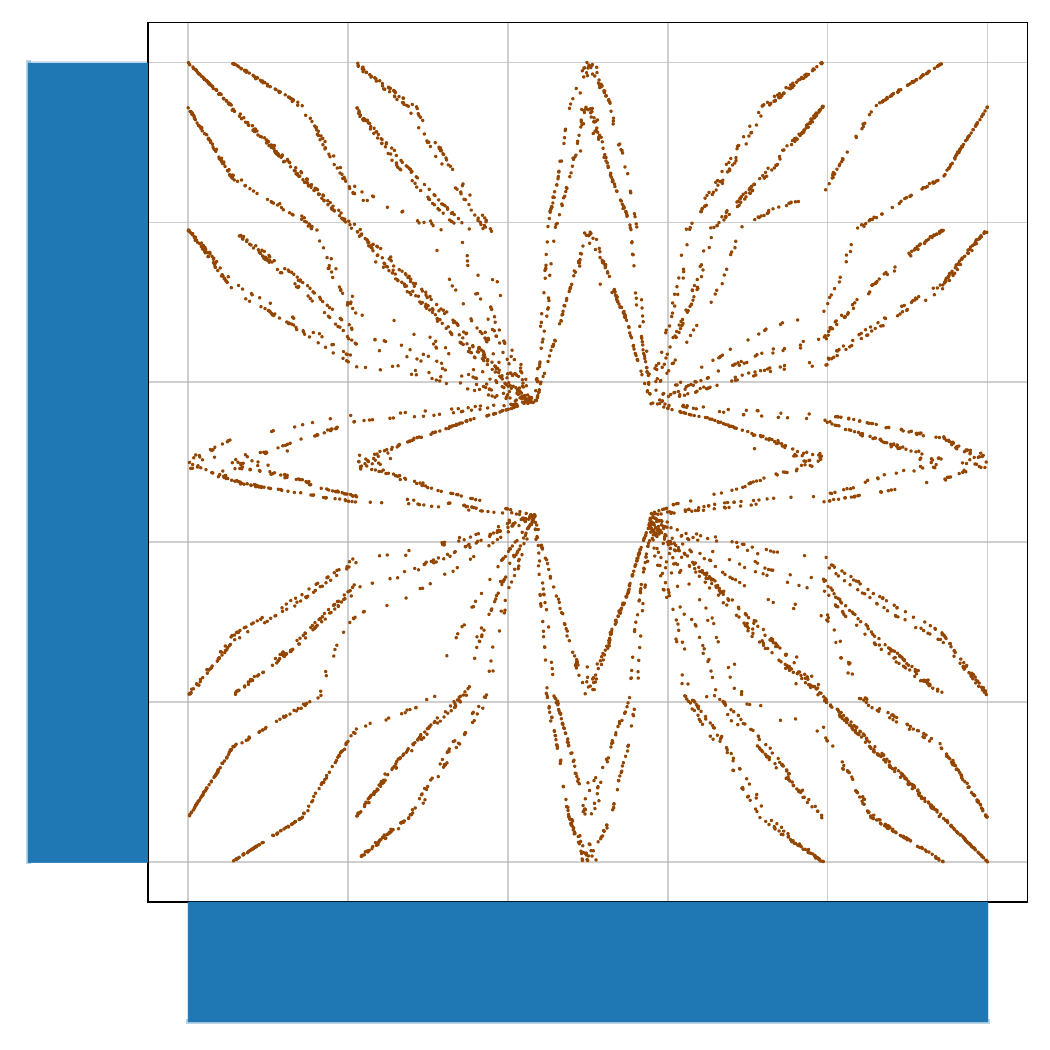}
\caption{$D = 7$}
\end{subfigure}

\smallskip
\begin{subfigure}{0.29\textwidth}
\includegraphics[width=\textwidth]{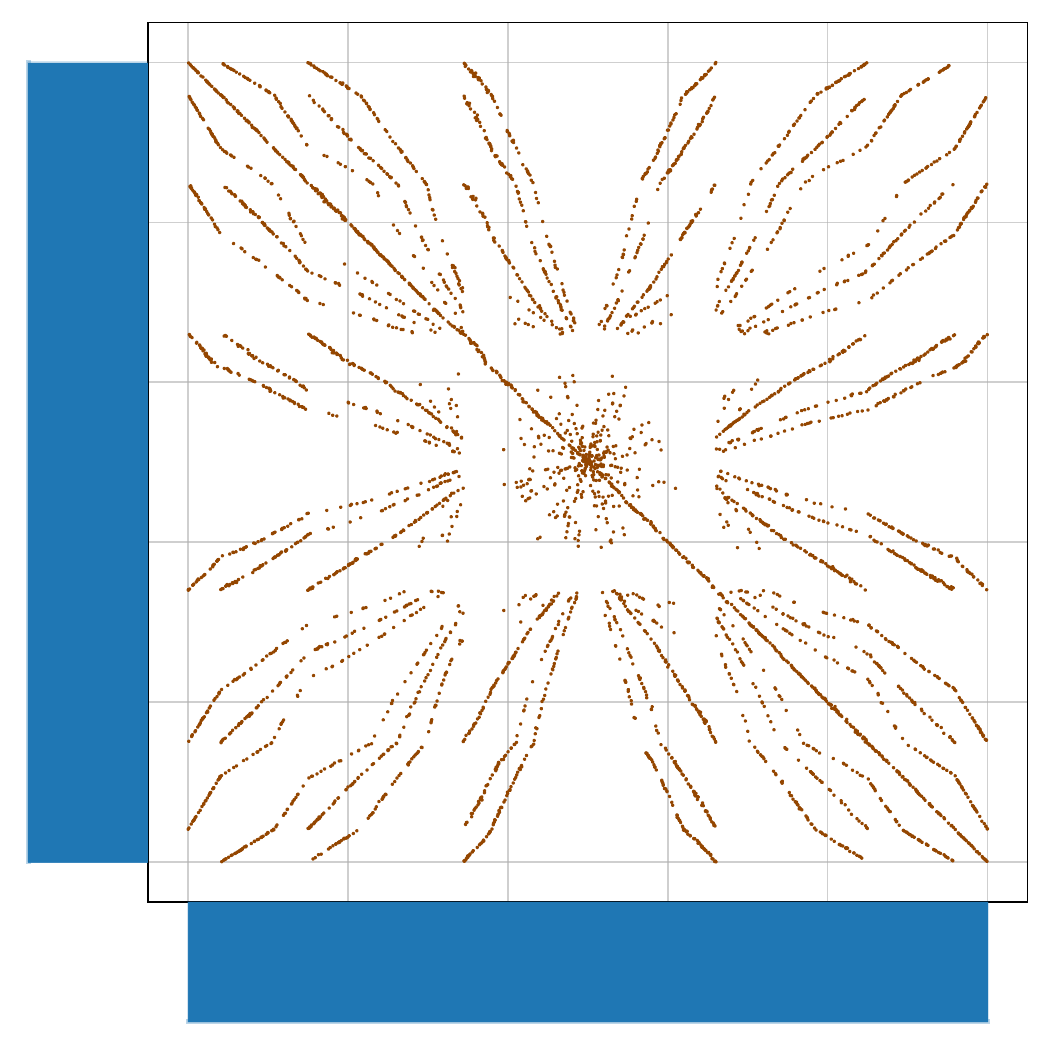}
\caption{$D = 8$}
\end{subfigure}
\hfill
\begin{subfigure}{0.29\textwidth}
\includegraphics[width=\textwidth]{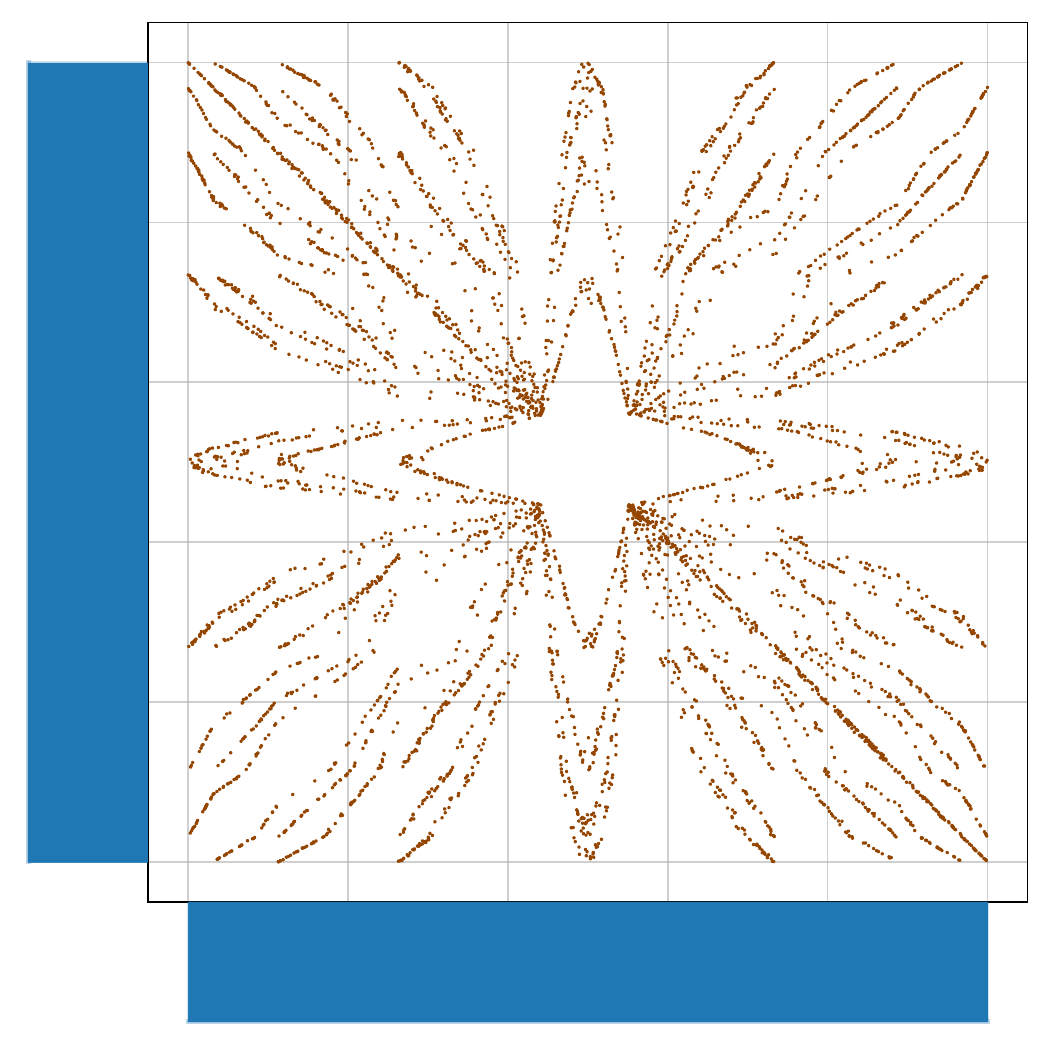}
\caption{$D = 9$}
\end{subfigure}
\hfill
\begin{subfigure}{0.29\textwidth}
\includegraphics[width=\textwidth]{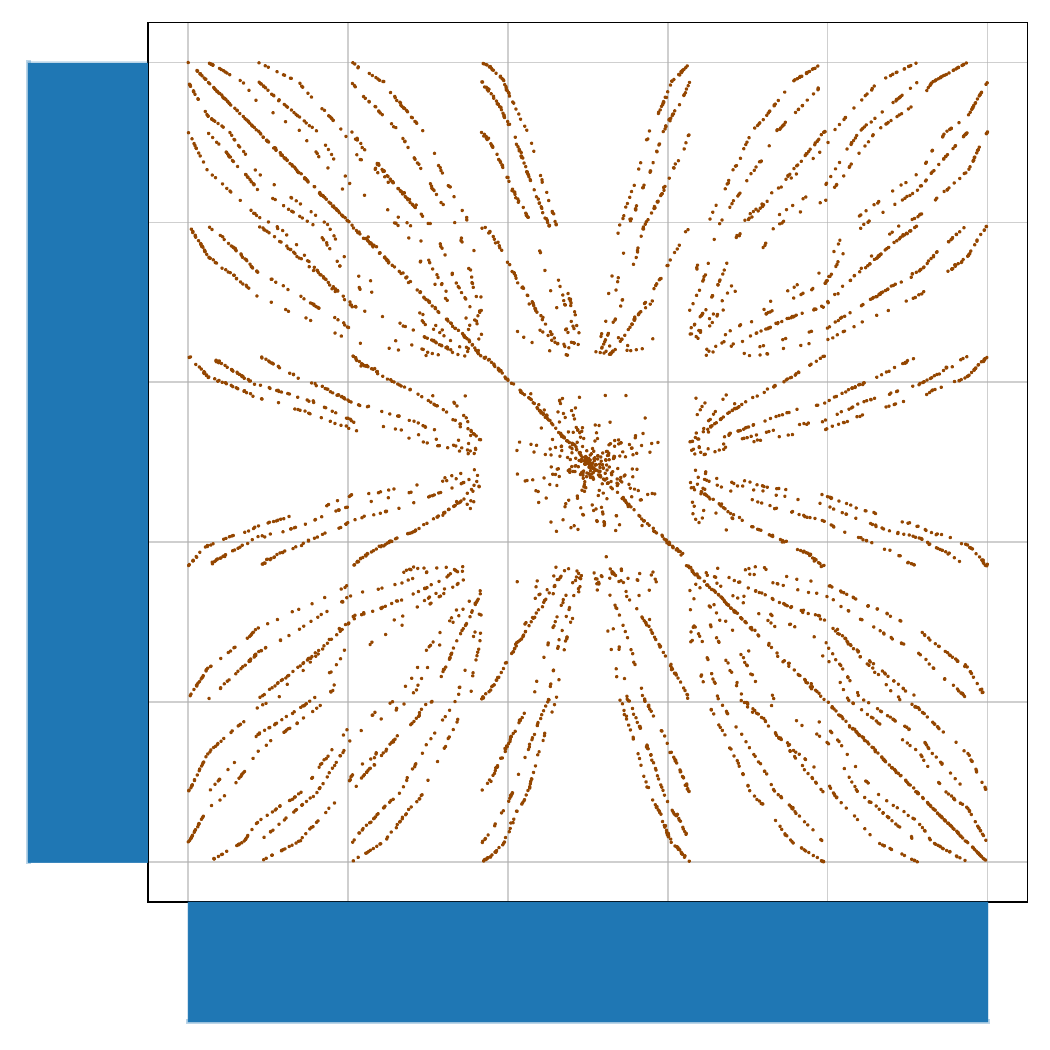}
\caption{$D = 10$}
\end{subfigure}
\caption{Projections on the first two axes of the numerical solutions found for Problem~\eqref{eq:risk_DFT_problem}, with $\rho = \L_{(0,1)}$ the uniform measure on $[0,1]$ and the quadratic spectral function $\alpha(t) \propto (t+\eta)^2-\eta$, with offset $\eta = 0.1$ ensuring nonzero derivative at $t=0$. Each subfigure corresponds to a different number $D$ of marginals.
We used $N = 5\ 000$ points, and the sequence of penalty coefficients $\lambda \in \{10^k : k=-3,-2,\dots,6 \}$.
The blue area at the bottom and left sides of each figure represent the active submeasures.
}
\label{fig:DFT-quadratic_simulations}
\end{figure}

\begin{figure}
\centering
\begin{subfigure}{0.29\textwidth}
\centering
\includegraphics[width=\textwidth]{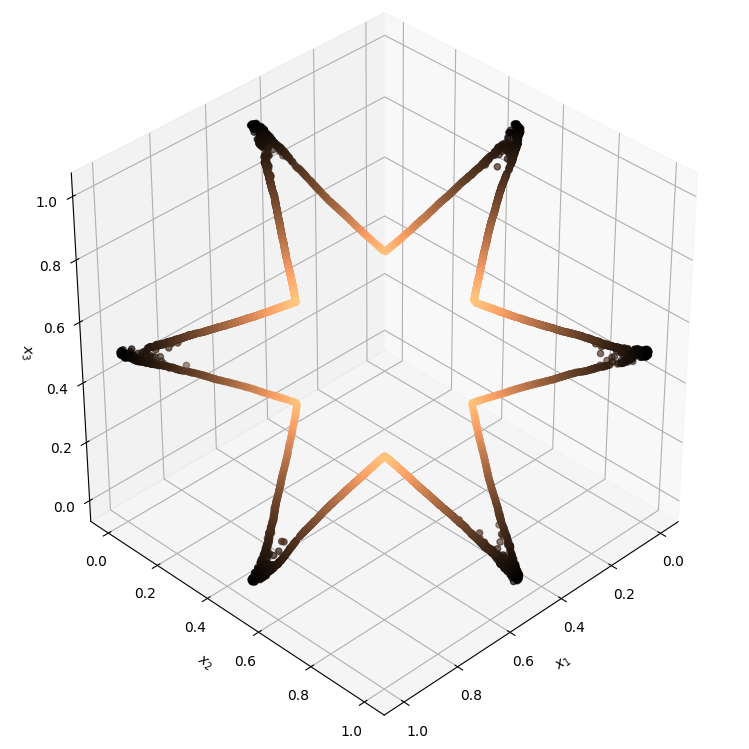}
\end{subfigure}
\hfill
\begin{subfigure}{0.38\textwidth}
\centering
\includegraphics[width=\textwidth]{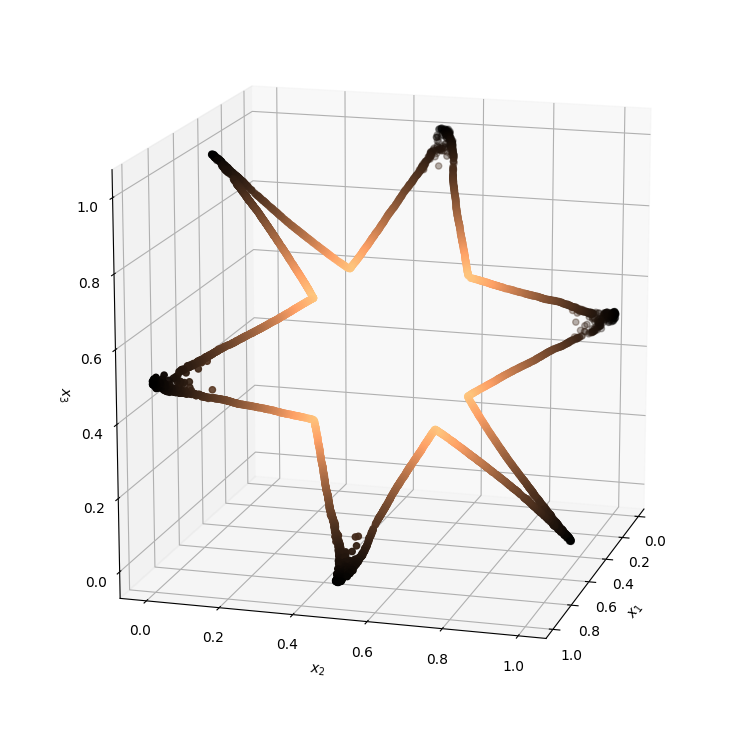}
\end{subfigure}
\hfill
\begin{subfigure}{0.29\textwidth}
\centering
\includegraphics[width=\textwidth]{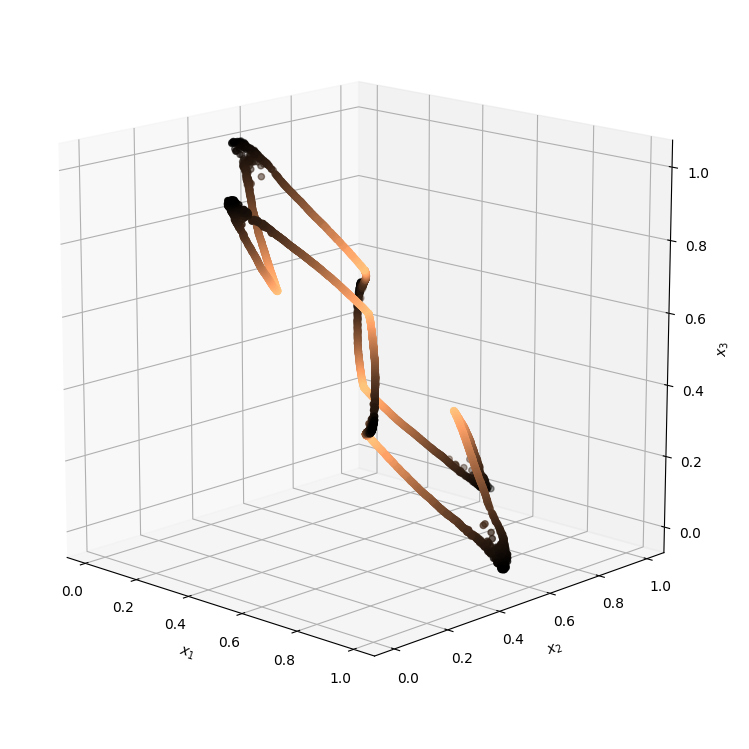}
\end{subfigure}
\caption{
Different views of the numerical solution corresponding to $D=3$ in Figure~\ref{fig:DFT-quadratic_simulations}. The color map represents the value of the Coulomb cost, gradually increasing from black to orange.
}
\label{fig:DFT-quadratic_D2}
\end{figure}

\subsection{Problem from~\cite{iooss2015review}}

In the example of Iooss and Lemaître's review~\cite{iooss2015review},
the random variable corresponding to the danger is the maximal annual overflow $S$ of the river.
It is modeled by the following expression, which stems from a simplification of the unidimensional hydro-dynamical equations of Saint-Venant,
\begin{equation}
\label{eq:maximal_annual_overflow}
S = Z_{\mathrm{v}} - H_{\mathrm{d}} - C_{\mathrm{b}} + \left( \frac{Q}{B K_{\mathrm{s}} \sqrt{\frac{Z_{\mathrm{m}}-Z_{\mathrm{v}}}{L}}} \right)^{0.6}.
\end{equation}
The marginal laws of the eight input variables are indicated in Table~\ref{tab:input_variables}. Since the first three terms of~\eqref{eq:maximal_annual_overflow} are decoupled from the other variables, they will not have any impact on the solution $\gamma$ of the risk estimation problem, and in particular the variables $H_{\mathrm{d}}$ and $C_{\mathrm{b}}$ can both be left out in the numerical simulations.
We therefore take the expression of the last term as the cost function $c : \X_1\times\dots\times\X_6 \to \R$, whose six arguments are the respective values of $Q$, $K_{\mathrm{s}}$, $Z_{\mathrm{v}}$, $Z_{\mathrm{m}}$, $L$, and $B$, with $\X_j$ are the respective supports of the corresponding marginal laws. The latter are represented in Figure~\ref{fig:density_profiles}.
As underlined in~\cite{ennaji2024robust}, this reduced cost function is strictly compatible, so that Lemma~\ref{lem:risk_problem_and_supermodularity} holds, with the adequate sign modifications required for its extension to compatibility, see Remark~\ref{rem:compatibility}. We denote the random variable corresponding to the reduced cost by
\begin{equation*}
S_{\mathrm{r}}
= \left( \frac{Q}{B K_{\mathrm{s}} \sqrt{\frac{Z_{\mathrm{m}}-Z_{\mathrm{v}}}{L}}} \right)^{0.6}.
\end{equation*}

\begin{table}[!ht]
\begin{center}
\begin{tabular}{lccc}
Input & Description & Unit & Probability distribution
\\
\hline
$Q$ & Maximal annual flowrate & m$^3$/s &
\begin{tabular}{@{}c@{}}
Truncated Gumbel ${\mathcal G}(1013, 558)$
\\
on $[500,3000]$
\end{tabular}
\\[10pt]
$K_{\mathrm{s}}$ & Strickler coefficient & - &
\begin{tabular}{@{}c@{}}
Truncated normal ${\mathcal N}(30,8)$
\\
on $[15,+\infty)$
\end{tabular}
\\[10pt]
$Z_{\mathrm{v}}$ & River downstream level & m & Triangular ${\mathcal T}(49,50,51)$
\\
$Z_{\mathrm{m}}$ & River upstream level & m & Triangular ${\mathcal T}(54,55,56)$
\\
$H_{\mathrm{d}}$ & Dyke height & m & Uniform ${\mathcal U}[7, 9]$
\\
$C_{\mathrm{b}}$ & Bank level & m & Triangular ${\mathcal T}(55,55.5,56)$
\\
$L$ & Length of the river stretch & m & Triangular ${\mathcal T}(4990,5000,5010)$
\\
$B$ & River width & m & Triangular ${\mathcal T}(295,300,305)$
\\
\hline
\end{tabular}
\caption{Input variables of the flood model and their probability distributions.}
\label{tab:input_variables}
\end{center}
\end{table}

\begin{figure}[!ht]
\centering
\includegraphics[width=0.75\textwidth]{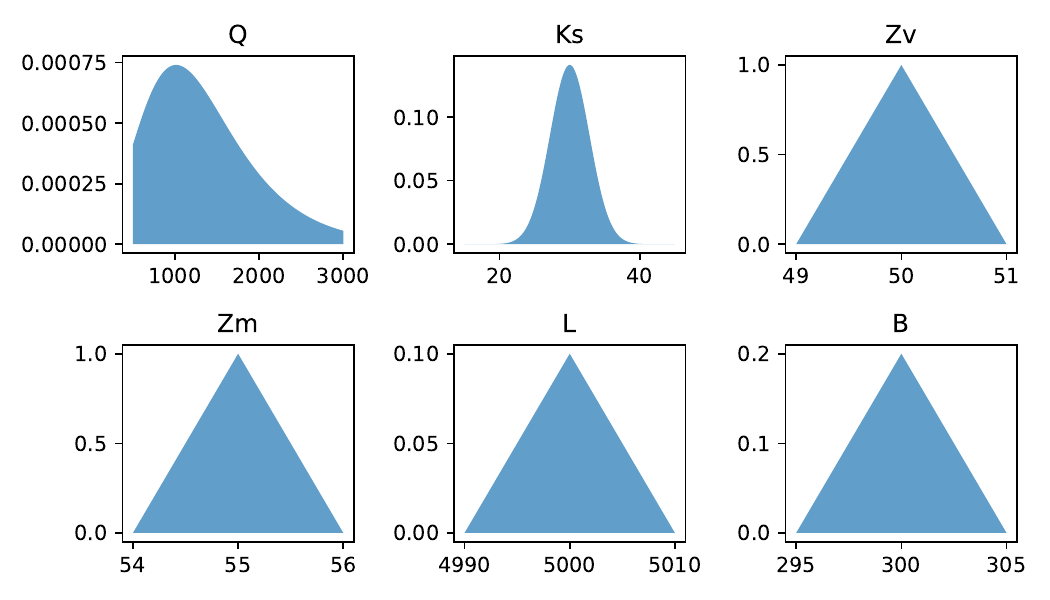}
\caption{Respective density profiles of the input variables.}
\label{fig:density_profiles}
\end{figure}

\begin{figure}[!ht]
\begin{center}
\begin{subfigure}{0.49\textwidth}
\centering
\includegraphics[width=0.9\textwidth]{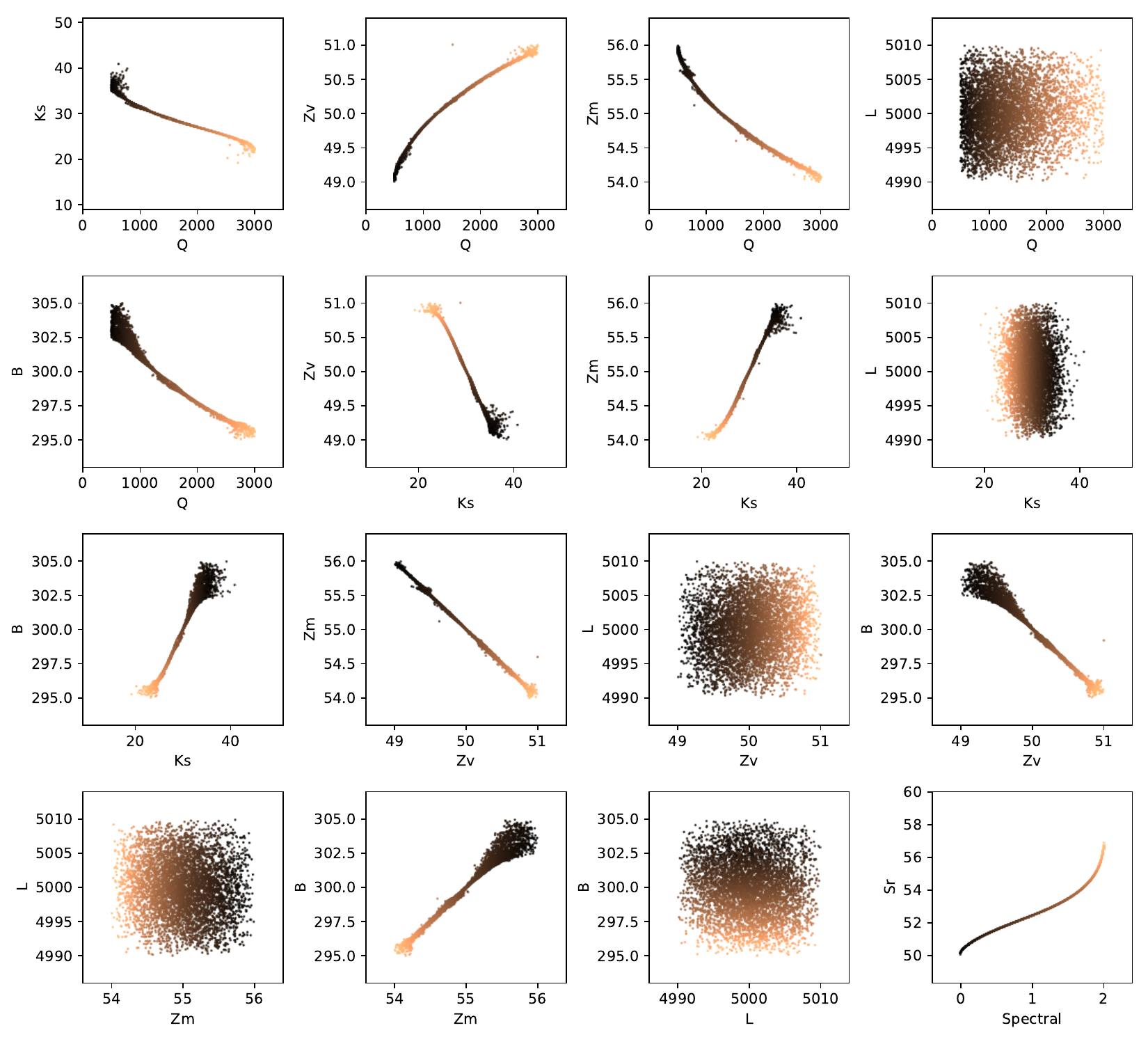}
\caption{Numerical solution}
\label{subfig:river_problem_numerical}
\end{subfigure}
\hfill
\begin{subfigure}{0.49\textwidth}
\centering
\includegraphics[width=0.9\textwidth]{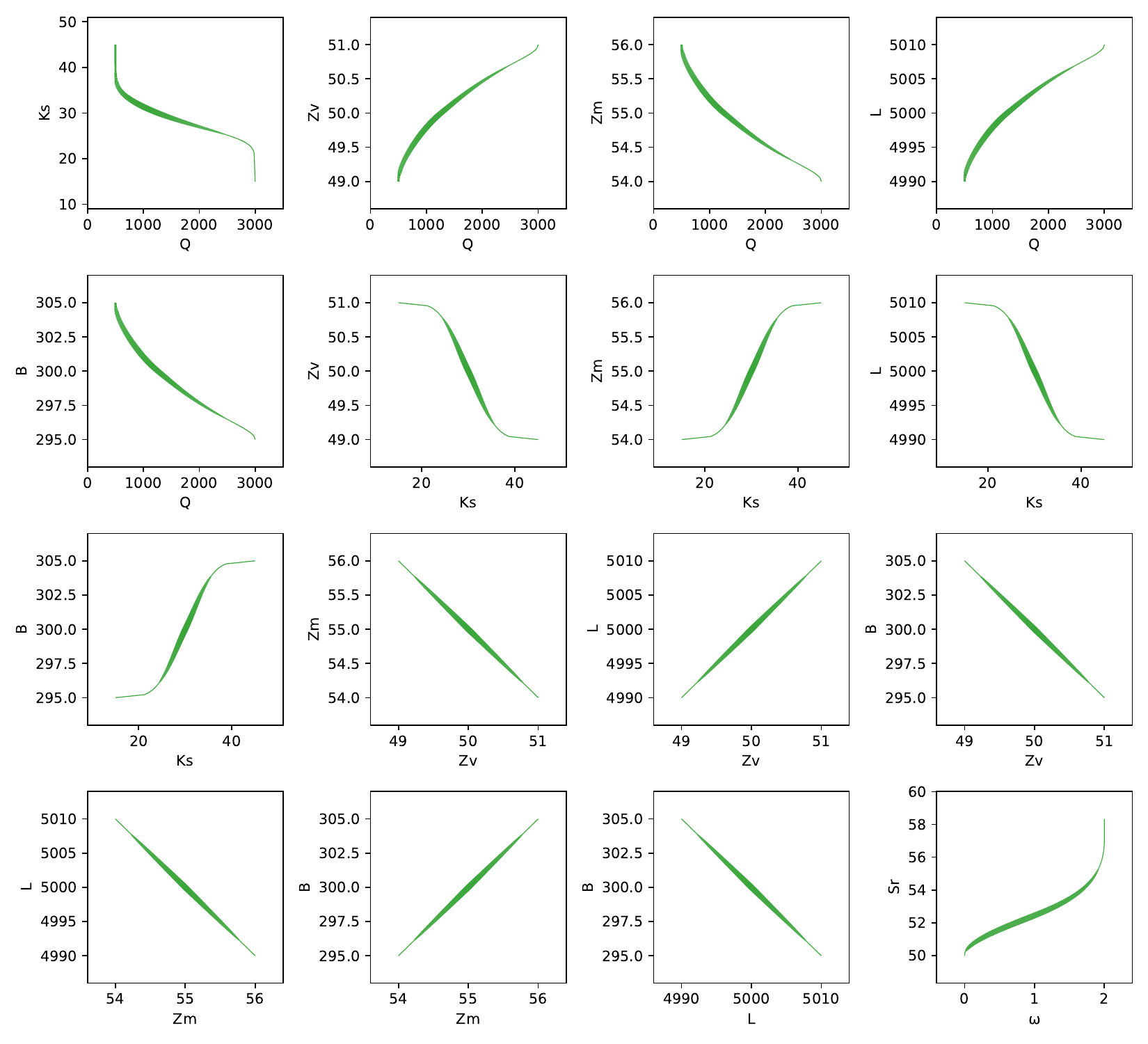}
\caption{Comonotone plan}
\label{subfig:river_problem_true}
\end{subfigure}
\end{center}
\caption{
The left-hand side subfigure shows the projections of the numerical solution corresponding to the linear spectral function $\alpha(t) = 2t$, on every fifteen possible pairs of the six main variables $Q$, $K_{\mathrm{s}}$, $Z_{\mathrm{v}}$, $Z_{\mathrm{m}}$, $L$, and $B$. The bottom right graph of this left-hand side subfigure corresponds to the joint projection of the auxiliary variable of law $\rho_0 = \alpha_\#\L_{(0,1)} = \U_{(0,2)}$ and of the cost variable $S_r$. The color map represents the value of $S_{\mathrm{r}}$, gradually increasing from black to orange.
We used $N = 5\ 000$ points, and the sequence of penalization coefficients $\lambda \in \{10^k : k = -2,-1,\dots,2\}$.
The right-hand side subfigure shows the same projections, but for the $s$-monotone coupling, which is the true solution thanks to compatibility of the cost function $s$.
This coupling was computed directly via the quantile functions of the marginal densities.
}
\label{fig:river_problem}
\end{figure}

\begin{figure}[!ht]
\centering
\includegraphics[width=0.75\textwidth]{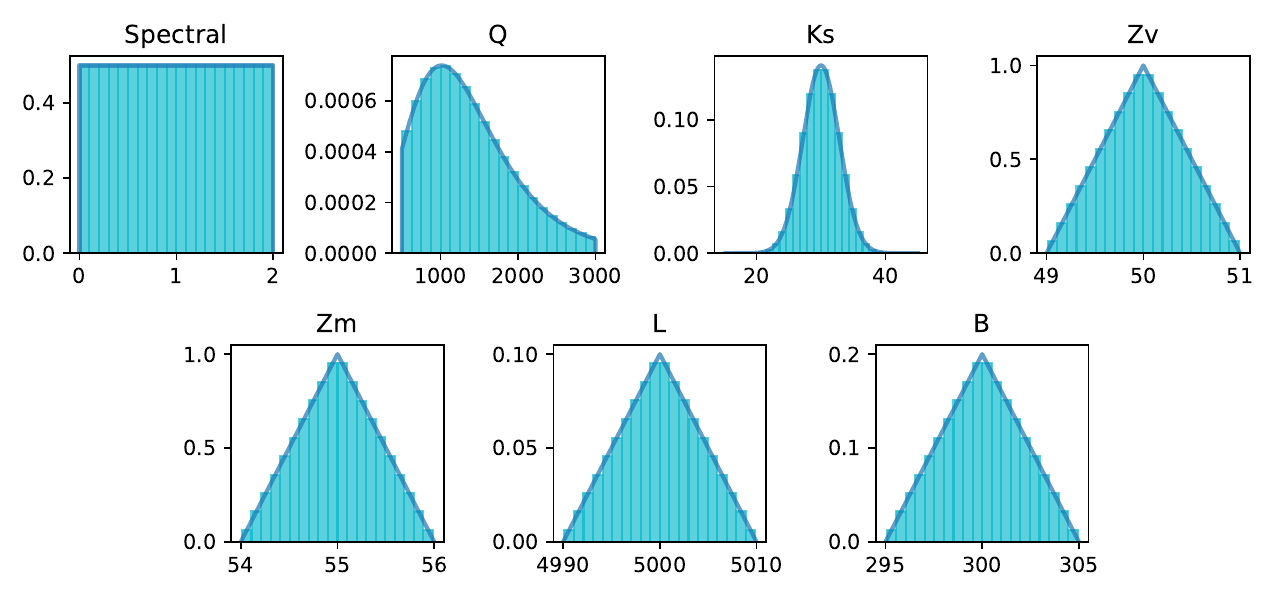}
\caption{Histograms for the numerical solution corresponding to a linear spectral function $\alpha(t) \propto t$. The top left graph corresponds to the density $\rho_0 = \alpha_\#\L_{(0,1)} = \U_{(0,2)}$.}
\label{fig:histograms_alpha_linear_river}
\end{figure}

\begin{figure}[!ht]
\centering
\includegraphics[width=0.5\textwidth]{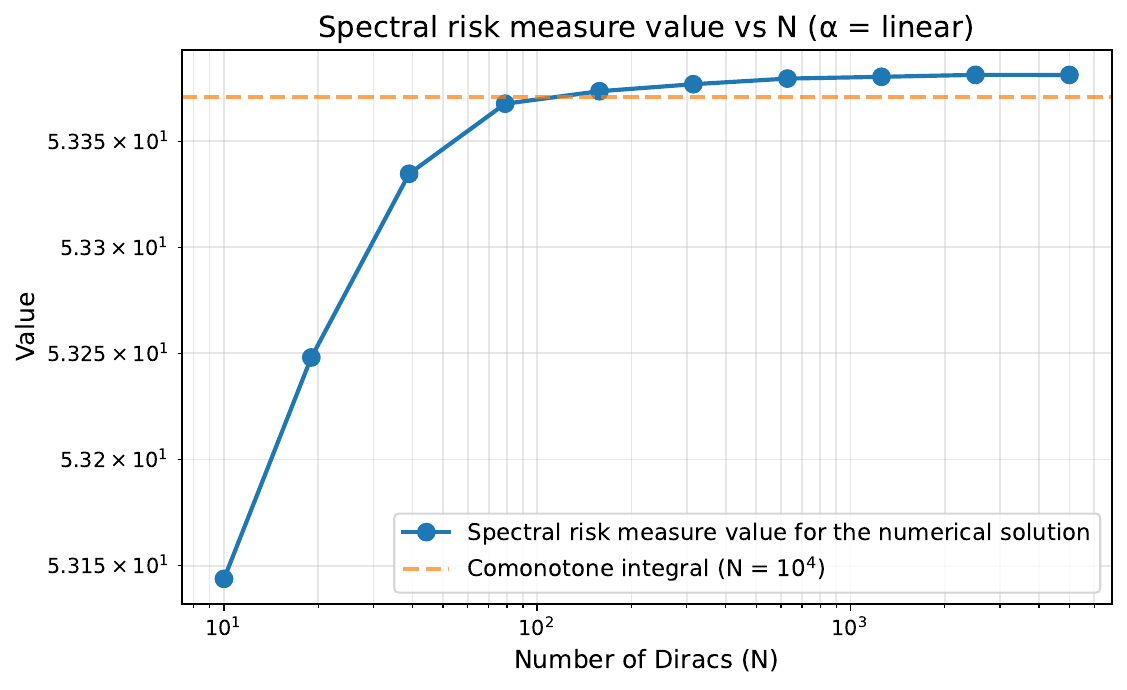}
\caption{Value $\Rr_\alpha(c_\#\delta_{Y_N})$ for the numerical solution in terms of the number of Dirac points, for the linear spectral function $\alpha(t) \propto t$. The reference value was computed using an explicit uniform approximations of the comonotone plan, with $N = 10^4$ Dirac points.}
\label{fig:R_alpha_value_as_function_of_N_alpha_linear_river}
\end{figure}

The $s$-monotone plan --- and a fortiori the $c$-monotone plan --- can be directly computed from the quantile functions of the marginals $\rho_0$, $\rho_1$, \dots, $\rho_D$, and is represented in Figure~\ref{subfig:river_problem_true}. We plot its projections on all fifteen possible pairs of the six main variables $Q$, $K_{\mathrm{s}}$, $Z_{\mathrm{v}}$, $Z_{\mathrm{m}}$, $L$, and $B$, as well as the joint measure $\tau$ it induces between $\rho_0$ and the cost random variable $S_{\mathrm{r}}$.

Figures~\ref{subfig:river_problem_numerical} and~\ref{fig:histograms_alpha_linear_river} show the numerical solution we find for the linear spectral function $\alpha(t) \propto t$.
We used $N=5\ 000$ Dirac points, and the sequence of penalty coefficients
$\lambda
\in \{10^{-2}, 10^{-1}, 10^0, 10^1, 10^2\}
$.
Note that whereas the pairwise correlations between the five variables $Q$, $K_{\mathrm{s}}$, $Z_{\mathrm{v}}$, $Z_{\mathrm{m}}$, and $B$ are fairly well reproduced, the variable $L$ do not seem to be correlated to any of the five other variables.
This is due to the very small relative range of $\rho_5 = \mathrm{Law}(L)$ (Figure~\ref{fig:density_profiles}), which makes its impact on the cost much weaker than that of the other variables. And indeed, changing $\rho_5$ for a probability density with a wider range yields a numerical solution with all six variables clearly correlated.
In each graph of Figure~\ref{subfig:river_problem_numerical}, the colormap represents the value of $S_r$ at each Dirac mass, with increasing value from black to orange. As expected, the correlations are more fuzzy in the areas corresponding to low values of $S_r$, since these areas less heavily weighted by the spectral function.

In Figure~\ref{fig:R_alpha_value_as_function_of_N_alpha_linear_river}, we plot the spectral risk measure value $\Rr_\alpha(\delta_{Y_N})$ of the computed numerical solution $Y_N$, as a function of $N$ the number of Dirac points.
The fact that the values exceed the reference value for large $N$ is due to the numerical solutions not being exactly optimal.

\paragraph*{\bf Acknowledgments}
\section*{Acknowledgments}
This work has been partially supported by the UDOPIA doctoral program, as well as the Agence nationale de la recherche, through the ANR project GOTA (ANR-23-CE46-0001) and the PEPR PDE-AI project (ANR-23-PEIA-0004).

\bibliographystyle{plain}
\bibliography{bibliography}

\end{document}

%% file: expected_shortfall.tikz
\definecolor{custom_blue}{RGB}{31,119,180}

\begin{tikzpicture}[scale=1.,
declare function={
f(\x) = {ifthenelse({\x>0},{exp(-1/(\x))},0.0)}; %
g(\x) = {f(\x)/(f(\x)+f(1-\x))}; %
h(\x) = {g(\x-1)}; %
bump(\x) = {1-g(\x^2)}; %
}]

\begin{axis}[xmin=-1.2, xmax=1.2, ymin=-0.2, ymax=1.2,
ticks=none, scale=0.8,
restrict y to domain=-0.2:1.2,
axis x line=center, axis y line=none,
axis on top, axis equal]

\fill[custom_blue, opacity=0.8, domain=0.5:1, variable=\x]
(0.5, 0)
-- plot (\x, {2*(-(\x-0.5)*(\x-0.4)*(\x+0.5) + 0.3)*bump(\x)*bump(\x-0.1)})
-- (1, 0)
-- cycle;

\draw[black, domain=0.5:1, variable=\x]
(0.5, 0)
-- plot (\x, {2*(-(\x-0.5)*(\x-0.4)*(\x+0.5) + 0.3)*bump(\x)*bump(\x-0.1)})
-- (1, 0)
-- cycle;

\addplot[black, samples=100, smooth, domain=-1:1, label={x}]
plot (\x, {2*(-(x-0.5)*(x-0.4)*(x+0.5) + 0.3)*bump(x)*bump(x-0.1)});

\addplot[black, samples=100, smooth, domain=-1.2:-1, label={x}]
plot (\x, {0*x});

\addplot[black, samples=100, smooth, domain=1:1.2, label={x}]
plot (\x, {0*x});

\draw [-, thick, black] (0.585,0) -- (0.585,-0.05)
node[below] {\small $\mathrm{CVaR}_m(\mu)$};
\draw (-0.5,0.6) node {
density of $\mu$};
\draw (0.9,0.7) node {
mass $m$};
\draw [-] (0.8,0.6) -- (0.6,0.2);
\draw (1.1,0.15) node {
\small $\mathbb R$};

\end{axis}

\end{tikzpicture}

%% file: bibliography.bib
@article{rockafellar2014random,
  title={Random variables, monotone relations, and convex analysis},
  author={Rockafellar, R Tyrrell and Royset, Johannes O},
  journal={Mathematical Programming},
  volume={148},
  pages={297--331},
  year={2014},
  publisher={Springer}
}

@article{iooss2015review,
  title={A review on global sensitivity analysis methods},
  author={Iooss, Bertrand and Lemaître, Paul},
  journal={Uncertainty management in simulation-optimization of complex systems: algorithms and applications},
  pages={101--122},
  year={2015},
  publisher={Springer}
}

@article{ennaji2024robust,
  title={Robust risk management via multi-marginal optimal transport},
  author={Ennaji, Hamza and M\'{e}rigot, Quentin and Nenna, Luca and Pass, Brendan},
  journal={Journal of Optimization Theory and Applications},
  pages={1--28},
  year={2024},
  publisher={Springer}
}

@article{santambrogio2015optimal,
  title={Optimal transport for applied mathematicians},
  author={Santambrogio, Filippo},
  journal={Birk{\"a}user, NY},
  volume={55},
  number={58-63},
  pages={94},
  year={2015},
  publisher={Springer}
}

@article{cances2025unidimensional,
  title={Unidimensional semi-discrete partial optimal transport},
  author={Cances, Adrien and Leclerc, Hugo},
  journal={arXiv preprint arXiv:2509.08799},
  year={2025}
}

@article{carlier2003class,
  title={On a class of multidimensional optimal transportation problems},
  author={Carlier, Guillaume},
  journal={Journal of convex analysis},
  volume={10},
  number={2},
  pages={517--530},
  year={2003},
  publisher={HELDERMANN VERLAG LANGER GRABEN 17, 32657 LEMGO, GERMANY}
}

@article{deelstra2011overview,
  title={An overview of comonotonicity and its applications in finance and insurance},
  author={Deelstra, Griselda and Dhaene, Jan and Vanmaele, Michele},
  journal={Advanced mathematical methods for finance},
  pages={155--179},
  year={2011},
  publisher={Springer}
}

@phdthesis{pass2012structural,
  title={Structural results on optimal transportation plans},
  author={Pass, Brendan},
  year={2012},
  school={University of Toronto}
}

@article{pass2012local,
  title={On the local structure of optimal measures in the multi-marginal optimal transportation problem},
  author={Pass, Brendan},
  journal={Calculus of Variations and Partial Differential Equations},
  volume={43},
  number={3},
  pages={529--536},
  year={2012},
  publisher={Springer}
}

@article{pass2015multi,
  title={Multi-marginal optimal transport: theory and applications},
  author={Pass, Brendan},
  journal={ESAIM: Mathematical Modelling and Numerical Analysis},
  volume={49},
  number={6},
  pages={1771--1790},
  year={2015}
}

@article{kim2014general,
  title={A general condition for Monge solutions in the multi-marginal optimal transport problem},
  author={Kim, Young-Heon and Pass, Brendan},
  journal={SIAM Journal on Mathematical Analysis},
  volume={46},
  number={2},
  pages={1538--1550},
  year={2014},
  publisher={SIAM}
}

@inproceedings{kitagawa2015multi,
  title={The multi-marginal optimal partial transport problem},
  author={Kitagawa, Jun and Pass, Brendan},
  booktitle={Forum of Mathematics, Sigma},
  volume={3},
  pages={e17},
  year={2015},
  organization={Cambridge University Press}
}

@article{colombo2015multimarginal,
  title={Multimarginal optimal transport maps for one--dimensional repulsive costs},
  author={Colombo, Maria and De Pascale, Luigi and Di Marino, Simone},
  journal={Canadian Journal of Mathematics},
  volume={67},
  number={2},
  pages={350--368},
  year={2015},
  publisher={Cambridge University Press}
}

@article{merigot2016minimal,
  title={Minimal geodesics along volume-preserving maps, through semidiscrete optimal transport},
  author={M{\'e}rigot, Quentin and Mirebeau, Jean-Marie},
  journal={SIAM Journal on Numerical Analysis},
  volume={54},
  number={6},
  pages={3465--3492},
  year={2016},
  publisher={SIAM}
}

@article{bencheikh2022approximation,
  title={Approximation rate in Wasserstein distance of probability measures on the real line by deterministic empirical measures},
  author={Bencheikh, Oumaima and Jourdain, Benjamin},
  journal={Journal of Approximation Theory},
  volume={274},
  pages={105684},
  year={2022},
  publisher={Elsevier}
}
